\documentclass{amsart}
\usepackage{tikz,amsmath,amssymb}
\usepackage{tkz-euclide} 
\usetikzlibrary{calc}
\usetikzlibrary{arrows} 
\usetikzlibrary{decorations.markings}
\usepackage{enumerate}
\usepackage{quiver} 
\usepackage[colorlinks=true,citecolor=cyan,backref=page]{hyperref}
\usepackage[noabbrev,capitalise]{cleveref}
\usepackage[normalem]{ulem}

\usepackage[utf8]{inputenc}

\usepackage[dvips,color,all]{xy} 
\usepackage{ stmaryrd }
\usepackage{ mathrsfs }
\usetikzlibrary{decorations.pathreplacing,calligraphy}

\newcommand{\End}{\operatorname{End}}

\newcommand{\Hom}{\operatorname{Hom}}

\newcommand{\Ext}{\operatorname{Ext}}
\newcommand{\C}{\mathcal C}

\newcommand{\bZ}{\mathbb Z}

\newcommand{\newword}[1]{{\bf \emph{#1}}} 
\newcommand{\bt}{\operatorname{bot}}
\newcommand{\tp}{\operatorname{top}}
\newcommand{\la}{\langle}
\newcommand{\ra}{\rangle}

\newcommand{\vdim}{\mathbf{dim}}
\newcommand{\Rad}{\mathcal{R}}

\newtheorem{thm}{Theorem}[section]
\newtheorem{lemma}[thm]{Lemma}
\newtheorem{corollary}[thm]{Corollary}
\newtheorem{proposition}[thm]{Proposition}
\newtheorem{conjecture}[thm]{Conjecture}

\theoremstyle{definition}
\newtheorem{definition}[thm]{Definition}
\newtheorem{example}[thm]{Example}
\newtheorem{remark}[thm]{Remark}

\definecolor{dark-green}{RGB}{14,150,2}
\newcommand{\gpoint}{{\color{dark-green}{\circ}}}
\newcommand{\rpoint}{\textcolor{red}{\bullet}}
	\definecolor{darkgray}{rgb}{0.66, 0.66, 0.66}
\definecolor{bluem)}{rgb}{0.0, 0.5, 0.69}

\title [Perfectly clustering words and band bricks]
    {A generalization of perfectly clustering words and band bricks for certain gentle algebras}

\author[Dequêne]{Benjamin Dequêne}
\address[B.~Dequ\^ene]{Département de math\'ematiques, LaCIM, Universit\'e du Qu\'ebec \`a Montr\'eal}
\email{dequene.benjamin@courrier.uqam.ca}
\author[Lapointe]{Mélodie Lapointe}

\address[M.~Lapointe]{D\'epartement de math\'ematiques et de statistique, Université de Moncton, Canada}
\email{melodie.lapointe@umoncton.ca}
\author[Palu]{Yann Palu}
\address[Y.~Palu]{Laboratoire Amienois de Math\'ematique fondamentale et appliqu\'ee, UPJV, Amiens, France}
\email{yann.palu@u-picardie.fr}
\author[Plamondon]{Pierre-Guy Plamondon}
\address[P.-G.~Plamondon]{Laboratoire de Math\'ematiques de Versailles, UVSQ, CNRS, Universit\'e Paris-Saclay, Institut universitaire de France (IUF)}
\email{pierre-guy.plamondon@uvsq.fr}
\thanks{The first author acknowledges the Institut des Sciences Mathématiques (ISM) for its partial support. The second author acknowledges the support of the Natural Sciences and Engineering Research Council of Canada (NSERC), [funding reference number BP–545242–2020] and the support of the Fonds de Recherche du Qu\'ebec en Science et Technologies. The third, fourth, and sixth authors acknowledge the support of the French ANR grant CHARMS (ANR-19-CE40-0017-02). The fourth author acknowledges the support of the Insitut Universitaire de France (IUF). The fifth author was partially supported by NSERC Canada. The sixth author acknowledges an NSERC Discovery Grant and the support of the Canada Research Chairs program.}
\author[Reutenauer]{Christophe Reutenauer}
\address[C.~Reutenauer]{D\'epartement de math\'ematiques, Universit\'e du Qu\'ebec \`a Montr\'eal}
\email{Reutenauer.Christophe@uqam.ca}
\author[Thomas]{Hugh Thomas}
\address[H.~Thomas]{D\'epartement de math\'ematiques, Universit\'e du Qu\'ebec \`a Montr\'eal}
\email{thomas.hugh\_r@uqam.ca}
\dedicatory{Dedicated to the memory of Andrzej Skowro\'nski}

\begin{document}

\begin{abstract}
We generalize the perfectly clustering words of Simpson and Puglisi and relate them to band bricks over certain gentle algebras. This allows us to prove a generalization of a conjecture by the second author on perfectly clustering words. 
\end{abstract}

\maketitle

\tableofcontents

\section{Introduction} \label{sec:intro}

Christoffel words form a family of well-studied words on two letters associated with certain lattice paths. We refer to \cite{BLRS} for a comprehensive introduction. Their importance stems from their appearance in various contexts, e.g. when studying palindromes, (positive) primitive elements of the free group on two generators, continued fractions, Markoff numbers, and discretization of lines.
Perfectly clustering words, introduced in \cite{PS}, generalize (lower) Christoffel words to alphabets with any number of letters. They arise in connection with word combinatorics and free groups, number theory, and dynamical systems.

In this article, we unveil a new relationship between perfectly clustering words and representation theory, via a specific family of \newword{gentle algebras}. This brings new tools into the study of perfectly clustering words, and hence of Christoffel words, that we apply in order to prove a conjecture from \cite{L2020}.

We now describe some of our main results. 

\subsection{Perfectly clustering words}
Let $A$ be a totally ordered alphabet and let $w\in A^\ast$ be a non-empty word. The \newword{Burrows--Wheeler transformation} of $w$, denoted by $BW(w)$, is defined as follows. Consider all words $w_1,\ldots,w_r$ that are conjugate (i.e. cyclically equivalent) to $w$, ordered in lexicographic order. Then $BW(w) = a_1a_2\cdots a_r$ where $a_1,\ldots,a_r$ are the last letters of $w_1,\ldots,w_r$, respectively.

If $a_1\geq a_2\geq \cdots\geq a_r$, and if $w$ is primitive, meaning that we cannot write $w$ as a power of another word, then $w$ is called \newword{perfectly clustering}.

For example, 
if $w=acab$, then $w_1=aba\color{blue}{c}$, $w_2=aca\color{blue}{b}$, $w_3=bac\color{blue}{a}$, $w_4=cab\color{blue}{a}$, and $w$ is perfectly clustering.
On the contrary, if $w=acba$, then
$w_1=aac\color{blue}{b}$, $w_2=acb\color{blue}{a}$, $w_3=baa\color{blue}{c}$, $w_4=cba\color{blue}{a}$, and $w$ is not perfectly clustering. 

\subsection{The gentle algebra \texorpdfstring{$\Lambda_n$}{Lambda} \texorpdfstring{}{n}}

We relate perfectly clustering words to representation theory via representations of a particular finite-dimensional algebra $\Lambda_n$, over an algebraically closed field $k$. 
This algebra is given by generators and relations in the guise of the following quiver with relations:
$$\begin{tikzcd}
	{Q_n:\; 1} & 2 & \cdots & n & {R_n: \beta_i\alpha_{i+1}=0,\; \alpha_i\beta_{i+1}=0} 
	\arrow["{\alpha_1}", shift left=1, tail reversed, no head, from=1-1, to=1-2]
	\arrow["{\beta_1}"', shift right=1, tail reversed, no head, from=1-1, to=1-2]
	\arrow["{\beta_2}", shift left=1, from=1-3, to=1-2]
	\arrow["{\alpha_2}"', shift right=1, from=1-3, to=1-2]
	\arrow["{\alpha_{n-1}}"', shift right=1, from=1-4, to=1-3]
	\arrow["{\beta_{n-1}}", shift left=1, from=1-4, to=1-3]
\end{tikzcd}$$

The left modules over $\Lambda_n$ can equivalently be understood as representations of $(Q_n,R_n)$.
They are thus given by
\begin{itemize}
 \item a finite-dimensional vector space at each vertex,
 \item a linear map between the relevant vector spaces at each arrow, satisfying the relations in $R_n$.
\end{itemize}
The quiver with relations $(Q_n,R_n)$ being \newword{gentle}, its indecomposable representations are well understood \cite{BR}: they are either \newword{string} representations $M(\omega)$, given by certain words $\omega$ in the arrows of $Q_n$ and their formal inverses, or \newword{band} representations $B_{z,m,\lambda}$, given by certain non-oriented cycles $z$ in $Q_n$ and two parameters $m\in\mathbb{N}$ and $\lambda\in k^\times = k\setminus\{0\}$.

\subsection{The link between perfectly clustering words and representations of \texorpdfstring{$\Lambda_n$}{Lambda} \texorpdfstring{}{n}}

With each integer $i\in\{1,\ldots,n\}$, we associate the cycle $$z_i=\alpha_1\alpha_2\cdots\alpha_{i-1}\beta_{i-1}^{-1}\cdots\beta_2^{-1}\beta_1^{-1}.$$
With each primitive word $w=i_1\cdots i_r\in\{1,\ldots,n\}^\ast$ on $n$ letters, we associate the cycle $\varphi(w)=z_{i_1}\cdots z_{i_r}$.
Our first main result is

\begin{thm}[see \cref{thm: brick bands and pcw}]
    A primitive word $w$ on $n$ letters is perfectly clustering if and only if the band $\Lambda_n$-module $B_{\varphi(w),1,\lambda}$ is a brick for some (equivalently any) $\lambda\in k^\times$. 
\end{thm} 

In the statement above, a representation $M$ of $\Lambda_n$ is called a \newword{brick} if $id_M$ forms a basis of the $k$-vector space of endomorphisms of $M$, i.e. if $\End_{\Lambda_n}(M)\cong k$.
Our proof makes use of topological models \cite{ABCP,BCS,OPS,PPP} describing representations of $\Lambda_n$ in terms of certain arcs on an oriented surface.

\subsection{Consequences for words} The map~$\Phi$ due to Gessel and the fifth author (defined in \cref{ssec:BWandPCW}) is a map $\Phi$ from words in $A^*$ to multiset of conjagacy classes of primitive words. More precisely, for $w \in A^*$, $\Phi(w)$ is the multiset of conjugacy classes of words obtained from the inverse of its standard permutation. The map $\Phi$ gives an inverse map for $BW$ for perfectly clustering words. 

The second author, in her thesis, expressed a conjecture regarding  the number of distinct lengths of perfectly clustering words appearing in $\Phi(n^{\alpha_{n}} \dots 2^{\alpha_2})$. (For compatibility with the rest of the article, we work with the alphabet $\{2,\dots,n\}$.) 

\medskip

\begin{conjecture}[\cite{L2020}]~\label{lengthconj}
  Let $n \geqslant 2$ and  $(\alpha_2, \ldots, \alpha_{n})$ be a $(n-1)$-tuple of nonnegative integers. The number of distinct lengths of conjugacy classes of words appearing in $\Phi(n^{\alpha_{n}} \dots 2^{\alpha_2})$ is at most $\lceil (n-1)/2\rceil$. 
\end{conjecture} 

Thanks to a more general result in our setting (\cref{maxsize}),
we prove the following strengthening of the previous conjecture. 

\begin{corollary}[see \cref{length}]
 Let $n \geqslant 1$ and  $(\alpha_2, \ldots, \alpha_{n})$ be a $(n-1)$-tuple of nonnegative integers. The number of distinct conjugacy classes of words appearing in $\Phi(n^{\alpha_{n}} \dots 2^{\alpha_2})$ is at most $\lceil (n-1)/2\rceil$.
\end{corollary}

That is to say, rather than bounding the number of distinct lengths of conjugacy classes of words, we actually bound the number of distinct conjugacy classes.

\subsection{Consequences for \texorpdfstring{$\Lambda_n$}{Lambda} \texorpdfstring{modules}{n} \texorpdfstring{}{modules}}  Let $M$ be a $\Lambda_n$ module. The \newword{$g$-vector} of $M$ is an integer vector which records the multiplicities of indecomposable summands of the minimal projective presentation of $M$. It is defined more precisely in the next section. 

Two brick modules $B$ and $B'$ are \newword{mutually compatible} if $$\Hom_{\Lambda_n}(B,B') = 0 = \Hom_{\Lambda_n}(B',B).$$ A module $M$ is called a  \newword{semibrick} if it decomposes as a direct sum of mutually compatible brick modules. 

In this article, we focus on band modules that are also bricks, which we refer to as \newword{band bricks}.  We also consider direct sums of mutually compatible band bricks, which we refer to as \newword{band semibricks}.

We recall some important results of \cite{GeissLabardiniSchroer}, specialized to our setting.
Band semibricks of $\Lambda_n$ are uniquely determined by their $g$-vectors (see \cref{corr:semibricks-gvec}) up to choices of the parameters $\lambda$ of the brick modules.
Denote by  $\Sigma_n$ the set of cones in the space of $g$-vectors which are spanned by collections of $g$-vectors corresponding to mutually compatible band bricks. Then $\Sigma_n$ is a fan, and the $g$-vectors corresponding to the band bricks which are the direct summands of a band semibrick $M$ are the rays of the minimal cone of $\Sigma_n$ containing the $g$-vector of $M$.

In our setting, we explicitly describe the support of the fan $\Sigma_n$ (see \cref{corr:semibricks-gvec}).
For $n=4$, we explicitly describe the rays of the fan (\cref{thm:bricksfour}).
This recovers and extends results of Coll, Giaquinto, and Magnant \cite{CGM}.

\section{Gentle algebras}

\subsection{Definition of gentle algebras and their representations} \label{ssec:defgentle}
A \newword{quiver} is a $4$-tuple $Q = (Q_0,Q_1,s,t)$ where $Q_0$ and $Q_1$ are (finite) sets called respectively the \newword{vertices} and \newword{arrows}, and $s,t : Q_1 \longrightarrow Q_0$ are the \newword{source} and \newword{target} maps. We think of a quiver as a directed graph.
Let $Q$ be a quiver. A \newword{path} in $Q$ is either
a sequence of arrows $\gamma_r \dots \gamma_1$ such that  $t(\gamma_i) = s(\gamma_{i+1})$ for $1\leq i \leq r-1$, or a formal element $e_i$ at each vertex $i$ called a \newword{lazy path} of length zero. For any path $\gamma = \gamma_r \ldots \gamma_1$, we define $s(\gamma) = s(\gamma_1)$ to be the source of $\gamma$ and $t(\gamma) = t(\gamma_r)$ the target of $\gamma$. We further define $s(e_i)=t(e_i)=i$. 

Let $k$ be an
algebraically closed field. We write $kQ$ for the \newword{path algebra}
of $Q$ over $k$. It is a vector space over $k$ with as a basis the paths of $Q$, equipped with a multiplication: we define the product of two paths $\sigma$ and $\tau$ as their concatenation $\sigma\tau$ if this is a path;
otherwise the product is zero.  The product is then defined on linear combination of paths by bilinearity.  An \newword{oriented cycle} is a path of length at least one whose starting and ending points coincide. 

Assume that for each vertex of $Q$ there are at most two outgoing arrows and at most two incoming arrows. A quotient of a path algebra $kQ/I$ is called \newword{gentle} if
$I$ is an ideal generated by a collection~$R$ of paths of length two such that, at
every vertex~$i$,
\begin{itemize}
  \item for any arrow $\alpha$ with target $i$ there is at most one
     arrow $\beta$ with source $i$ and $\beta\alpha \not\in R$, and at most one
     arrow $\gamma$ with source $i$ and $\gamma\alpha\in R$,
  \item dually, for any arrow $\delta$ with source $i$, there is at most one
     arrow $\beta$ with target $i$ such that $\delta\beta\not\in R$,
    and at most one  arrow $\gamma$ with target $i$  such that $\delta\gamma \in R$,
  \item for any oriented cycle $\alpha_r\cdots\alpha_1$ of $Q$, there is some
    pair of consecutive arrows in $R$ (i.e., some $i$ with $\alpha_{i+1}\alpha_i \in R$ or $\alpha_1\alpha_r \in R$).
    \end{itemize}

A pair $(Q,R)$ is called a \newword{gentle quiver} if $Q$ is a quiver and $R$ is a collection of paths satisfying the above conditions.
All the above conditions except the one requiring some relations in each cycle
can be summarized
as saying that the local configuration at any vertex $v$ can be obtained from
the figure below by removing any number of arrows (and possibly identifying some vertices). 

$$
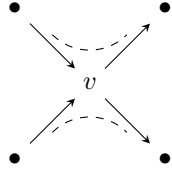
\begin{tikzpicture} 
  \node (a) at (-1,1) {$\bullet$};
  \node (b) at (-1,-1) {$\bullet$};
  \node (c) at (0,0) {$v$};
  \node (d) at (1,1) {$\bullet$};
  \node (e) at (1,-1) {$\bullet$};
  \draw[-stealth] (a)--(c); \draw[-stealth] (b) -- (c);
  \draw[-stealth] (c) --(d); \draw[-stealth] (c) --(e);
  \draw [dashed] (-.5,.65) arc (225:315:.7);
  \draw [dashed] (-.5,-.65) arc (135:45:.7);
  \end{tikzpicture}$$
  
Here $v$ is the central vertex; the dashed arcs designate relations. If an arrow is removed, the
relation containing it is also removed.

A \newword{representation}~$V$ of a gentle quiver $(Q,R)$ over $k$ is an assignment of a
$k$-vector space~$V_i$ to each vertex~$i$ of $Q$ and a linear map~$V_\gamma:V_i\to V_j$ for each arrow~$i \xrightarrow{\gamma} j$  of $Q$, such that the linear combinations of compositions
corresponding to elements of $R$ are zero.

A \newword{morphism} between two representations $V,W$ is a collection of
linear maps $f_i:V_i\rightarrow W_i$ for each vertex $i\in Q$, such that
for each arrow $i \xrightarrow{\gamma} j$ of $Q$, we have $f_j\circ V_\gamma = W_\gamma \circ f_i$.
An \newword{isomorphism} is an invertible morphism; a morphism $(f_i)_{i \in Q_0}$ is an
isomorphism if and only if $f_i$ is an isomorphism for every vertex $i$ of $Q$. 

The category of left $kQ/I$-modules is equivalent to the category of representations of $(Q,R)$ over $k$. Hence any time we talk about a representation, one can have in mind a module, and vice versa. For this reason we may freely use either terminology depending on the context. 

Given two representations $X$ and $Y$,
their direct sum is defined by $(X\oplus Y)_i = X_i \oplus Y_i$ and~$(X\oplus Y)_\gamma = X_\gamma \oplus Y_\gamma$ for each vertex~$i$ and each arrow~$\gamma$.
A non-zero representation is called \newword{indecomposable} if it
is not isomorphic to the direct sum of two non-zero representations.

Let~$\Lambda = kQ/I$.  For each~$i\in\{1, \ldots, n\}$, we denote by~$P_i = \Lambda e_i$ the indecomposable \newword{projective representation} at vertex~$i$.  Any representation~$X$ is part of an exact sequence
\[
 \bigoplus_{i=1}^n P_i^{\oplus b_i} \to \bigoplus_{i=1}^n P_i^{\oplus a_i} \to X \to 0
\]
called a \newword{projective presentation}.  Among all projective presentations of~$X$, there is a unique one (up to isomorphism) for which the~$a_i$ and~$b_i$ are minimal; it is called the \newword{minimal projective presentation} of~$X$.  The \newword{$g$-vector} of~$X$ is the integer vector~$g(X) = (a_1-b_1, \ldots, a_n-b_n)$, where the~$a_i$ and~$b_i$ are those of the minimal projective presentation.  Note that~$g(X\oplus Y) = g(X) + g(Y)$.

For a general quiver $Q$ and a general ideal $I$, it is not feasible to classify the indecomposable
representations. However, if $kQ/I$ is gentle, a classification of the
indecomposable modules was obtained by Butler and Ringel \cite{BR}. We now recall
this classification, which divides indecomposable modules into two types,
\newword{strings} and \newword{bands}.

\subsection{String representations} \label{ssec:stringrep}
 
To begin with, it is convenient to introduce a set of \newword{inverse arrows} of the quiver
$Q$. The set of inverse arrows of $Q$ is just the set of formal objects
$\alpha^{-1}$ for $\alpha\in Q_1$. We define $s(\alpha^{-1})=t(\alpha)$ and $t(\alpha^{-1})=s(\alpha)$.

A \newword{walk} in $Q$
is a sequence $(v_{r}, \alpha_r, v_{r-1}, \dots, v_1,\alpha_1,v_{0})$,
whose elements
alternate between the set of vertices and the set of
arrows and inverse arrows, beginning and ending with a vertex,
and such that
$s(\alpha_i)=v_{i-1}$, $t(\alpha_i)=v_{i}$. For any walk $w = (v_{r}, \alpha_r, v_{r-1}, \dots, v_1,\alpha_1,v_{0})$, we define $s(w) = v_0$ as the \textit{source} of $w$, and $t(w) = v_r$ as the \textit{target} of $w$. The vertices carry no information
except in the case of walks of length zero, so we usually write $w=\alpha_r\cdots\alpha_1$ instead. However, it is important to
remember that zero-length walks exist, and that they specify a (single) vertex.

A \newword{string walk} in a gentle quiver $(Q,R)$ is a walk~$w=\alpha_r\cdots\alpha_1$ in $Q$ such that for all~$i\in\{1, \ldots, r-1\}$, $\alpha_{i+1}\neq \alpha_i^{-1}$, and neither of~$\alpha_{i+1}\alpha_i$ and~$(\alpha_{i+1}\alpha_i)^{-1}$ is a path in~$R$.

  Every string walk $z$ defines a \newword{string representation} $M_z$, as follows.
  Let $z_r,z_{r-1}, \ldots, z_0$ be the vertices of~$Q$ successively visited by the walk~$z$, where~$z_r$ is the target of~$z$ and~$z_0$ its source. 

  Define the vector space at vertex $i$ by
  $$(M_z)_i = \bigoplus_{z_s=i} k\epsilon_s$$
  That is to say, the vector space at vertex $i$ has a basis consisting of
  vectors $\epsilon_s$ with $z_s=i$.

  Now, for each arrow $\gamma$ in $Q$, from, say, vertex $i$ to vertex $j$,
  we must define a linear map $(M_z)_\gamma:(M_z)_i\rightarrow (M_z)_j$. It
  suffices to define $(M_z)_\gamma (\epsilon_s)$, for $z_s=i$.

  We define it as:\begin{itemize}
  \item $\epsilon_{s+1}$ if $z$ travels from $z_s$ to $z_{s+1}$ forwards along
    $\gamma$, 
  \item $\epsilon_{s-1}$ if $z$ travels from $z_{s-1}$ to $z_s$ backwards along
    $\gamma$,
  \item zero, otherwise.
  \end{itemize}
  Note that the first two possibilities are disjoint by the first point in the
  definition of a string walk. 

  Then $M_z$ does indeed define a representation of $kQ/I$ where we recall that $I$ is the ideal generated by $R$; the
  composition of the linear maps corresponding to any pair of arrows in $R$ is
  zero by the second point in the definition of a string walk.

  It turns out that the representations $M_z$ are indecomposable. A walk and
  its inverse walk define isomorphic representations, but otherwise,
  different string walks yield non-isomorphic representations.

 \begin{example} \label{ex::stringrep} 
 Let us consider the following quiver $Q$
 $$\begin{tikzpicture}[->]
  \node (a) at (0,0) {$1$};
  \node (b) at (2,0) {$2$};
  \node (c) at (4,0) {$3$};
  \draw ([yshift=1mm]b.west)--node[above]{$\beta_{1}$}([yshift=1mm]a.east);
  \draw  ([yshift=-1mm]b.west)--node[below]{$\alpha_{1}$}([yshift=-1mm]a.east);
    \draw ([yshift=1mm]c.west)--node[above]{$\alpha_{2}$}([yshift=1mm]b.east);
  \draw  ([yshift=-1mm]c.west)--node[below]{$\beta_{2}$}([yshift=-1mm]b.east);
  \draw[dashed,-] ([xshift=.4cm]b.north) arc[start angle = 0, end angle = 180, x radius=.4cm, y radius =.2cm];
  \draw[dashed,-] ([xshift=-.4cm]b.south) arc[start angle = 180, end angle = 360, x radius=.4cm, y radius =.2cm];
  
\end {tikzpicture}
 $$ equipped with the ideal $I$ generated by $R = \{\alpha_1 \beta_2, \beta_1 \alpha_2\}$. For example $$\rho= (1, \alpha_1,2,\alpha_2,3,\beta_2^{-1},2,\alpha_2,3,\beta_2^{-1},2,\beta_1^{-1},1,\alpha_1,2,\alpha_2,3)$$
 is a string walk of $(Q,R)$. More visually, we can draw the same string walk as follows.

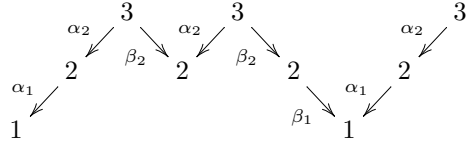
\begin{center}
$\displaystyle \xymatrix@R=1em@C=1em{
 & & 3 \ar[ld]_{\alpha_2} \ar[rd]_{\beta_2}   & & 3 \ar[ld]_{\alpha_2} \ar[rd]_{\beta_2} & & & &  3 \ar[ld]_{\alpha_2} \\
 & 2 \ar[ld]_{\alpha_1} & & 2& & 2  \ar[rd]_{\beta_1} & & 2 \ar[ld]_{\alpha_1}& \\
 1 & & & & & & 1 & & }$
\end{center}

By convention, arrows always point downwards. To get the string representation $M_\rho$ associated to a string $\rho$ combinatorially, we can replace each vertex of the string in the drawing above by the corresponding $\epsilon_i$. Each arrow defines the direction of the linear map that we will get. 

$$\displaystyle 
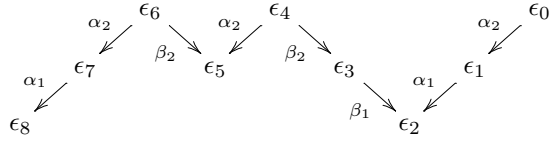
\xymatrix@R=1em@C=1em{
 & & \epsilon_6 \ar[ld]_{\alpha_2} \ar[rd]_{\beta_2}   & & \epsilon_4 \ar[ld]_{\alpha_2} \ar[rd]_{\beta_2} & & & &  \epsilon_0 \ar[ld]_{\alpha_2} \\
 & \epsilon_7 \ar[ld]_{\alpha_1} & & \epsilon_5& & \epsilon_3  \ar[rd]_{\beta_1} & & \epsilon_1 \ar[ld]_{\alpha_1}& \\
 \epsilon_8 & & & & & & \epsilon_2 & & }$$

Then by squashing the string, we get the representation below.

$$\begin{tikzpicture}[->,scale=0.9]
  \node (a) at (-3,0) {$\langle \epsilon_2,\epsilon_8 \rangle$};
  \node (b) at (2,0) {$\langle \epsilon_1, \epsilon_3, \epsilon_5, \epsilon_7 \rangle$};
  \node (c) at (7,0) {$\langle \epsilon_0, \epsilon_4,\epsilon_6 \rangle$};
  \draw ([yshift=1mm]b.west)--node[above]{$\begin{matrix}\epsilon_2 \longmapsfrom \epsilon_3\end{matrix}$}([yshift=1mm]a.east);
  \draw  ([yshift=-1mm]b.west)--node[below]{$\begin{matrix}
  \epsilon_8 \longmapsfrom \epsilon_7 \\
  \epsilon_2 \longmapsfrom \epsilon_1
\end{matrix}   $}([yshift=-1mm]a.east);
    \draw ([yshift=1mm]c.west)--node[above]{$\begin{matrix}
  \epsilon_7 \longmapsfrom \epsilon_6 \\
  \epsilon_5 \longmapsfrom \epsilon_4 \\
  \epsilon_1 \longmapsfrom \epsilon_0
\end{matrix}   $}([yshift=1mm]b.east);
  \draw  ([yshift=-1mm]c.west)--node[below]{$\begin{matrix}
  \epsilon_5 \longmapsfrom \epsilon_6 \\
  \epsilon_3 \longmapsfrom \epsilon_4 
\end{matrix}   $}([yshift=-1mm]b.east);
  \draw[dashed,-] ([xshift=1.5cm]b.north) arc[start angle = 0, end angle = 180, x radius=1.5cm, y radius =.5cm];
  \draw[dashed,-] ([xshift=-1.5cm]b.south) arc[start angle = 180, end angle = 360, x radius=1.5cm, y radius =.5cm];
  
\end {tikzpicture}
$$

\begin{center}
$ \cong \vcenter{\hbox{ \begin{tikzpicture}[->,scale=0.9]
  \node (a) at (-2.5,0) {$k^2$};
  \node (b) at (2,0) {$k^4$};
  \node (c) at (6.5,0) {$k^3$};
  \draw ([yshift=1mm]b.west)--node[above]{$\left[\begin{matrix}
  0&1&0&0 \\
  0&0& 0 & 0
\end{matrix} \right]  $}([yshift=1mm]a.east);
  \draw  ([yshift=-1mm]b.west)--node[below]{$ \left[\begin{matrix}
 1 & 0 & 0 & 0\\
 0 & 0 & 0 & 1
\end{matrix} \right]$}([yshift=-1mm]a.east);
    \draw ([yshift=1mm]c.west)--node[above]{$\left[\begin{matrix}
 1&0&0 \\
 0&0&0 \\
 0&1&0 \\
 0&0&1 
\end{matrix} \right]  $}([yshift=1mm]b.east);
  \draw  ([yshift=-1mm]c.west)--node[below]{$\left[\begin{matrix}
 0&0&0 \\
 0&1&0 \\
 0&0&1 \\
 0&0&0 
\end{matrix} \right]  $}([yshift=-1mm]b.east);
  \draw[dashed,-] ([xshift=1cm]b.north) arc[start angle = 0, end angle = 180, x radius=1cm, y radius =.5cm];
  \draw[dashed,-] ([xshift=-1cm]b.south) arc[start angle = 180, end angle = 360, x radius=1cm, y radius =.5cm];
  
\end {tikzpicture}}}
$
\end{center} 

\end{example}

\subsection{Band representations}\label{ssec::band}
The definition of bands is slightly more complicated than
the definition of strings. A
  \newword{band walk} is a non-lazy string walk subject to the additional conditions:
  \begin{itemize}
   \item its starting and ending points coincide,
   \item it is primitive (i.e., not a power of a smaller walk),
   \item all of its powers are strings.
  \end{itemize}

  Given a band walk $z$, a positive integer $m$, and $\lambda\in k^\times = k\setminus\{0\}$, we define a representation $B_{z,m,\lambda}$. Let $r$ be the length of $z$.

  We introduce a collection of linearly independent vectors $\epsilon_i^p$ for
  $0\leq i\leq r-1$ and $1\leq p \leq m$. We define the vector spaces of the representation by
  $$(B_{z,m,\lambda})_i = \bigoplus_{z_s=i, 0\leq s \leq r-1, 1\leq p\leq m}
    k\epsilon^p_{s}.$$
    This has the effect that each vertex visited by $z$ except the last one
    contributes $m$ basis elements to the vector space over the corresponding
    vertex.  We fix the convention that $\epsilon^p_{r}=\epsilon^p_0$.

    Now, let $\gamma$ be an arrow of $Q$ from vertex $i$ to vertex $j$. We
    define $(B_{z,m,\lambda})_\gamma$ almost in the same way as for a string walk.
    Specifically, $(B_{z,m,\lambda})_\gamma(\epsilon^p_s)$ is defined as follows
    unless $\gamma$ is the arrow which $z$ takes between $z_{r-1}$ and
    $z_r$:
\begin{itemize}
  \item $\epsilon^p_{s+1}$ if $z$ travels from $z_s$ to $z_{s+1}$ forwards along
    $\gamma$, 
  \item $\epsilon^p_{s-1}$ if $z$ travels from $z_{s-1}$ to $z_s$ backwards along
    $\gamma$,
  \item zero, otherwise
    \end{itemize}

    The remaining case splits into two possibilities. The
    first is when $z_{r-1}=i$, $z_r=j$, and $\gamma$ is the arrow which
    $z$ takes from $z_{r-1}$ to $z_r$. In this case, we
     define $(B_{z,m,\lambda})_\gamma$ restricted to $\bigoplus_{p=1}^m  k\epsilon^p_{r-1}$
    as a linear map
    to $\bigoplus_{p=1}^m k\epsilon^p_{r}$ given, with respect to these bases,
    by a single Jordan block with
    eigenvalue $\lambda$. 

    The second possibility is when $z_{r-1}=j$, $z_r=i$, and $z$ follows
    $\gamma$ in the reverse direction from $z_{r-1}$ to $z_r$. 
    In this case, we define $(B_{z,m,\lambda})_\gamma$ restricted to
    $\bigoplus_{p=1}^m k\epsilon^p_{r}$ as a linear map to
    $\bigoplus_{p=1}^m k\epsilon^p_{r-1}$ as a linear map given, with respect
     to these bases, by a single Jordan block of eigenvalue $\lambda^{-1}$.

This completes the definition of the band representation $B_{z,m,\lambda}$.

We note the two following points: 
\begin{itemize}
\item the band representation $B_{z,m,\lambda}$ is indecomposable;

\item two band representations $B_{z,m,\lambda}$ and $B_{y,n,\mu}$ are isomorphic if, and only if, $m = n$ and :
\begin{itemize}
\item $z$ is (cyclically) conjugate to $y$ and $\lambda = \mu$; or
\item $z$ is conjugate to the inverse of $y$ and $\lambda = \mu^{-1}$.
\end{itemize}
\end{itemize}

 \begin{example} \label{ex::bandrep} 
 Let us consider the same gentle quiver $(Q,R)$ as Example \ref{ex::stringrep}.
 $$\begin{tikzpicture}[->]
  \node (a) at (0,0) {$1$};
  \node (b) at (2,0) {$2$};
  \node (c) at (4,0) {$3$};
  \draw ([yshift=1mm]b.west)--node[above]{$\beta_{1}$}([yshift=1mm]a.east);
  \draw  ([yshift=-1mm]b.west)--node[below]{$\alpha_{1}$}([yshift=-1mm]a.east);
    \draw ([yshift=1mm]c.west)--node[above]{$\alpha_{2}$}([yshift=1mm]b.east);
  \draw  ([yshift=-1mm]c.west)--node[below]{$\beta_{2}$}([yshift=-1mm]b.east);
  \draw[dashed,-] ([xshift=.4cm]b.north) arc[start angle = 0, end angle = 180, x radius=.4cm, y radius =.2cm];
  \draw[dashed,-] ([xshift=-.4cm]b.south) arc[start angle = 180, end angle = 360, x radius=.4cm, y radius =.2cm];
  
\end {tikzpicture}
 $$  We can check that $\omega =  \alpha_1 \alpha_2 \beta_2^{-1} \beta_1^{-1}$ is a band walk of $(Q,R)$. More visually, we can draw a band the same way as
 a string. Note that we circle the first and the last vertices because we will identify them.

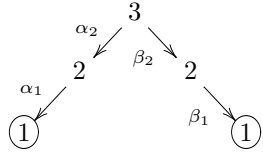
\begin{center}
$\displaystyle \xymatrix@R=1em@C=1em{
 & & 3 \ar[ld]_{\alpha_2} \ar[rd]_{\beta_2}  & & \\
 & 2 \ar[ld]_{\alpha_1} & & 2  \ar[rd]_{\beta_1} &  \\
*+[o][F-]{1} & & & & *+[o][F-]{1}  }$
\end{center}

Let $m > 0$ and $\lambda \in k^\times$. To get the band representation $B_{\omega,m,\lambda}$ associated to a band $\omega$ combinatorially, we can replace each vertex of the band by a copy of the vector space $k^m$. Each arrow defines an identity map between vector spaces, except the last arrow which defines a map given by a Jordan block matrix.  

\begin{center}
$\displaystyle \vcenter{\xymatrix@R=1em@C=1em{
 & & k^m \ar[ld]_{I_m} \ar[rd]_{I_m}  & & \\
 & k^m \ar[ld]_{J_m(\lambda)} & & k^m \ar[rd]_{I_m} &  \\
 *+[o][F-]{k^m} & & & & *+[o][F-]{k^m }}}$ where $\displaystyle J_m(\lambda) = \left( \begin{matrix}
 \lambda & 1 & & \\
 & \lambda & \ddots & \\
 & & \ddots & 1 \\
 & & & \lambda
 \end{matrix} \right)$.
\end{center}

Then by squashing the band as we do for strings to get string representations, except that we identify the first and last copies of $k^m$ (the circled ones), we get the following representation. 
$$ \begin{tikzpicture}[->]
  \node (a) at (-2,0) {$k^m$};
  \node (b) at (2,0) {$k^{2m}$};
  \node (c) at (6,0) {$k^m$};
  \draw ([yshift=1mm]b.west)--node[above]{$ \left[\begin{matrix}
I_m&0
\end{matrix} \right]$}([yshift=1mm]a.east);
  \draw  ([yshift=-1mm]b.west)--node[below]{$\left[\begin{matrix}
0&J_m(\lambda)
\end{matrix} \right]  $}([yshift=-1mm]a.east);
    \draw ([yshift=1mm]c.west)--node[above]{$\left[\begin{matrix}
 0 \\
 I_m
\end{matrix} \right]  $}([yshift=1mm]b.east);
  \draw  ([yshift=-1mm]c.west)--node[below]{$\left[\begin{matrix}
 I_m \\
 0
\end{matrix} \right]  $}([yshift=-1mm]b.east);
  \draw[dashed,-] ([xshift=1cm]b.north) arc[start angle = 0, end angle = 180, x radius=1cm, y radius =.5cm];
  \draw[dashed,-] ([xshift=-1cm]b.south) arc[start angle = 180, end angle = 360, x radius=1cm, y radius =.5cm];
  
\end {tikzpicture}
$$
 \end{example}

       Every indecomposable representation of $kQ/I$ belongs to exactly
       one isomorphism class of string or band representations, as defined
       above.
       We are using the fact that $k$ is algebraically closed here:
       over a non-algebraically closed field, we would need to use a
       more general definition of band representation.
       For a band~$z$ and an integer~$m$, we call the family~$\{B_{z,m,\lambda} \ | \ \lambda \in k^\times\}$ a \newword{one-parameter family of bands}. 
       
\subsection{Brick modules}\label{ssec:brick}

A module $M$ is called a \newword{brick} if the ring $\End(M):=\Hom(M,M)$ is a skew field.
  Since $\End(M)$ is a finite-dimensional algebra over the algebraically
  closed field $k$,  this condition is
  equivalent to saying that $\End(M)\simeq k$;  equivalently, 
  a vector space basis for $\End(M)$ is given by the identity map. 
  
   \begin{example} \label{ex::brickrep} 
Let us consider the same gentle quiver $(Q,R)$ as in previous examples.
 $$\begin{tikzpicture}[->]
  \node (a) at (0,0) {$1$};
  \node (b) at (2,0) {$2$};
  \node (c) at (4,0) {$3$};
  \draw ([yshift=1mm]b.west)--node[above]{$\beta_{1}$}([yshift=1mm]a.east);
  \draw  ([yshift=-1mm]b.west)--node[below]{$\alpha_{1}$}([yshift=-1mm]a.east);
    \draw ([yshift=1mm]c.west)--node[above]{$\alpha_{2}$}([yshift=1mm]b.east);
  \draw  ([yshift=-1mm]c.west)--node[below]{$\beta_{2}$}([yshift=-1mm]b.east);
  \draw[dashed,-] ([xshift=.4cm]b.north) arc[start angle = 0, end angle = 180, x radius=.4cm, y radius =.2cm];
  \draw[dashed,-] ([xshift=-.4cm]b.south) arc[start angle = 180, end angle = 360, x radius=.4cm, y radius =.2cm];
  
\end {tikzpicture}
$$ 

Let us consider $M = B_{\omega,1,\lambda}$  with $\omega$ being the band defined in Example \ref{ex::bandrep}. 

$$ \begin{tikzpicture}[->]
  \node (a) at (-1,0) {$k$};
  \node (b) at (2,0) {$k^2$};
  \node (c) at (5,0) {$k$};
  \draw ([yshift=1mm]b.west)--node[above]{$ \left[\begin{matrix}
1&0
\end{matrix} \right]$}([yshift=1mm]a.east);
  \draw  ([yshift=-1mm]b.west)--node[below]{$\left[\begin{matrix}
0& \lambda
\end{matrix} \right]  $}([yshift=-1mm]a.east);
    \draw ([yshift=1mm]c.west)--node[above]{$\left[\begin{matrix}
 0 \\
 1
\end{matrix} \right]  $}([yshift=1mm]b.east);
  \draw  ([yshift=-1mm]c.west)--node[below]{$\left[\begin{matrix}
 1 \\
 0
\end{matrix} \right]  $}([yshift=-1mm]b.east);
  \draw[dashed,-] ([xshift=1cm]b.north) arc[start angle = 0, end angle = 180, x radius=1cm, y radius =.5cm];
  \draw[dashed,-] ([xshift=-1cm]b.south) arc[start angle = 180, end angle = 360, x radius=1cm, y radius =.5cm];
  
\end {tikzpicture}
$$

Let $f = (f_1, f_2,f_3) \in \Hom(M,M)$. Then we know that $f_3 = a \cdot \mathsf{Id}_k$ for some $a \in k$ as $f_3$ has to be a linear endomorphism of the field $k$. Because of the commutative squares that $f$ must satisfy, we can easly deduce that $f_2 = a \cdot \mathsf{Id}_{k^2}$ and $f_1 = a \cdot \mathsf{Id}_k$. Hence we conclude that $M$ is a brick module.
 \end{example}

\section{Representations of gentle algebras and curves on surfaces}\label{sec:gentle-curves}

We will be using a topological model developed in \cite{ABCP,BCS,OPS,PPP} to study gentle algebras and their representations.  Although we will mainly be interested in the family of gentle algebras $\Lambda_n$ defined in the introduction, we will devote this section to the description of the general case.

\subsection{Surfaces and dissections} \label{ssec:surfdissec}

By \newword{surface} we will always mean an oriented compact surface~$S$ with a finite number (which may be zero) of open discs removed.  We denote the boundary of~$S$ by~$\partial S$.  Such a surface is determined by its genus and by the number of connected components of~$\partial S$; we will refer to these connected components as the \newword{boundary components} of~$S$.

A \newword{marked surface} is a pair~$(S,M)$, where~$S$ is a surface and~$M$ is a finite set of \newword{marked points} of~$S$ satisfying the following: 
\begin{itemize}
 \item $M$ is the disjoint union of two subsets~$M_{\gpoint}$ and~$M_{\rpoint}$;
 \item on each boundary component of~$S$, marked points in~$M_{\gpoint}$ and~$M_{\rpoint}$ alternate, and there is at least one of each;
 \item marked points are also allowed in the interior of the surface.
\end{itemize}

A \newword{$\gpoint$-arc} is a curve on~$(S,M)$ joining two points in~$M_{\gpoint}$; formally, it is a continuous map from the closed interval~$[0,1]$ to~$S$ with endpoints in~$M_{\gpoint}$ and with the image of its interior is disjoint from $M$. A \newword{$\rpoint$-arc} is defined similarly.

A $\gpoint$-arc is \newword{simple} if it does not intersect itself (except perhaps at its endpoints).

A~\newword{$\gpoint$-dissection} of~$(S,M)$ is a collection of pairwise non-intersecting simple~$\gpoint$-arcs which cut the surface into polygons, called the \newword{cells} of the dissection, each of which contains exactly one marked point in~$M_{\rpoint}$. We define a $\rpoint$-dissection similarly.

\begin{example}\label{exam::three-surfaces}
Here are examples of~$\gpoint$-dissections of a disc (that is, a sphere with one boundary component), an annulus (or a sphere with two boundary components), and a torus with one boundary component.

\begin{figure}[h!]
\begin{center}
\begin{tikzpicture}[mydot/.style={
    circle,
    thick,
    fill=white,
    draw,
    outer sep=0.5pt,
    inner sep=1pt
  }]
\tkzDefPoint(0,0){O}\tkzDefPoint(1.5,0){12}
\tkzDefPointsBy[rotation=center O angle 360/12](12,1,2,3,4,5,6,7,8,9,10){1,2,3,4,5,6,7,8,9,10,11}
\tkzDrawPoints[fill =red,size=4,color=red](2,4,6,8,10,12)
\tkzDrawCircle[line width=0.5mm,black](O,7)

\tkzDrawPoints[size=4,color=dark-green,mydot](O,1,3,5,7,9,11)

\draw[line width=0.4mm,dark-green](1) edge (11);
\draw[line width=0.4mm,dark-green](11) edge (O);
\draw[line width=0.4mm ,dark-green](O) edge (3);
\draw[line width=0.4mm,dark-green](3) edge (7);
\draw[line width=0.4mm,dark-green](7) edge (5);
\draw[line width=0.4mm,dark-green](7) edge (9);
\end{tikzpicture} \quad  \begin{tikzpicture}[mydot/.style={
    circle,
    thick,
    fill=white,
    draw,
    outer sep=0.5pt,
    inner sep=1pt
  }]
\tkzDefPoint(0,0){O}\tkzDefPoint(0,1.5){1}
\tkzDefPointsBy[rotation=center O angle 180](1){2}
\tkzDefPoint(0.7,0){3} \tkzDefPoint(1.1,0){8} \tkzDefPoint(-1.1,0){9}
\tkzDefPointsBy[rotation=center O angle 90](3,4,5){4,5,6}
\tkzDrawPoints[fill =red,size=4,color=red](2,3,5)
\tkzDrawCircle[line width=0.5mm,black](O,1)
\tkzDrawCircle[line width=0.5mm,black,fill=gray!40](O,3)
\tkzDrawPoints[size=4,color=dark-green,mydot](1,4,6)

\draw[line width=0.4mm,dark-green](1) edge (4);
\draw[line width=0.4mm,bend left = 60,dark-green](4) edge (8);
\draw[line width=0.4mm,bend left = 60,dark-green](8) edge (6);
\draw[line width=0.4mm,bend right = 60,dark-green](4) edge (9);
\draw[line width=0.4mm ,bend right =60,dark-green](9) edge (6);
\end{tikzpicture}
\quad \begin{tikzpicture}[mydot/.style={
    circle,
    thick,
    fill=white,
    draw,
    outer sep=0.5pt,
    inner sep=1pt
  }]
\tkzDefPoint(0,0){O}\tkzDefPoint(1.4,1.4){1}
\tkzDefPointsBy[rotation=center O angle 90](1,2,3){2,3,4}
\tkzDefPoint(0.9,0.9){5}
\tkzDefPointsBy[rotation=center O angle 90](5,6,7){6,7,8}

\tkzDrawSquare[line width=0.5mm](1,2)
\tkzDefPointsBy[rotation=center 1 angle 45](5){9}
\tkzDrawSector[rotate,line width=0.5mm, black, fill = gray!40](1,9)(-90)
\tkzDefPointsBy[rotation=center 2 angle 45](6){10}
\tkzDrawSector[rotate,line width=0.5mm, black, fill = gray!40](2,10)(-90)
\tkzDefPointsBy[rotation=center 3 angle 45](7){11}
\tkzDrawSector[rotate,line width=0.5mm, black, fill = gray!40](3,11)(-90)
\tkzDefPointsBy[rotation=center 4 angle 45](8){12}
\tkzDrawSector[rotate,line width=0.5mm, black, fill = gray!40](4,12)(-90)

\tkzDefPoint(1.4,0){13}
\tkzDefPoint(-1.4,0){14}
\tkzDefPoint(0,1.4){15}
\tkzDefPoint(0,-1.4){16}

\tkzDrawPoints[fill =red,size=4,color=red](5,7)
\tkzDrawPoints[size=4,color=dark-green,mydot](6,8)
\draw[line width=0.4mm,dark-green](6) edge (8);
\draw[line width=0.4mm,bend left=40,dark-green](6) edge (14);
\draw[line width=0.4mm,bend right=40,dark-green](6) edge (15);
\draw[line width=0.4mm,bend left=40,dark-green](8) edge (13);
\draw[line width=0.4mm,bend right=40,dark-green](8) edge (16);
\end{tikzpicture}
\caption{\label{fig:dissections} $\gpoint$-dissections of a disc (left), an annulus (middle), and a torus with one boundary component (right).}
\end{center}
\end{figure}
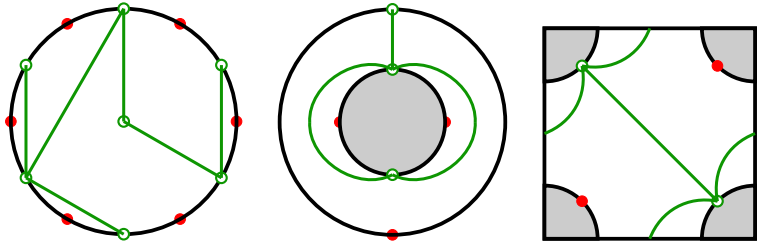

\end{example}

\subsection{The gentle quiver of a dissection} \label{ssec:gentlesurf}
Let~$\Delta^{\gpoint}$ be a~$\gpoint$-dissection of~$(S,M)$.  We define a quiver~$Q(\Delta^{\gpoint})$ and a set of relation~$R(\Delta^{\gpoint})$ as follows:
\begin{enumerate}
 \item the vertices of~$Q(\Delta^{\gpoint})$ are in bijection with the arcs in~$\Delta^{\gpoint}$; 
 \item for each configuration as drawn on the left in \cref{fig:rules_quiver_dissections} in~$\Delta^{\gpoint}$ (that is to say, for each common endpoint of a pair of arcs~$i$ and~$j$ around which~$j$ comes immediately after~$i$ in counter-clockwise orientation), there is an arrow~$i\to j$ in~$Q(\Delta^{\gpoint})$;
 \item for each configuration as drawn on the right in \cref{fig:rules_quiver_dissections} in~$\Delta^{\gpoint}$ with corresponding arrows~$i\xrightarrow{\alpha} j$ and~$j\xrightarrow{\beta} k$, the path~$\beta\alpha$ is a relation in~$R(\Delta^{\gpoint})$.
\end{enumerate}
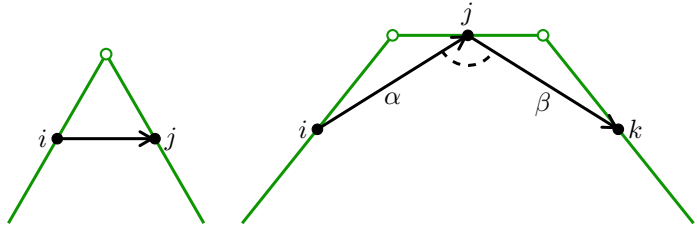
\begin{figure}[h!] 
\begin{center}
\begin{tikzpicture}[mydot/.style={
    circle,
    thick,
    fill=white,
    draw,
    outer sep=0.5pt,
    inner sep=1pt
  }, fl/.style={->,>=latex}]
\tkzDefPoint(0,0){O}\tkzDefPoint(0,1.5){1}
\tkzDefPointsBy[rotation=center O angle 120](1,2){2,3}
\tkzDefMidPoint(1,2)
\tkzGetPoint{i}
\tkzDefMidPoint(1,3)
\tkzGetPoint{j}

\draw[line width=0.4mm,dark-green](1) edge (2);
\draw[line width=0.4mm,dark-green](1) edge (3);
\draw[->, >= angle 60,line width=0.4mm,black] (i) -- (j);
\tkzDrawPoints[size=4,color=black](i,j);\tkzDrawPoints[size=4,color=dark-green,mydot](1);
\tkzLabelPoints[left](i)
\tkzLabelPoints[right](j)
\end{tikzpicture} \quad
\begin{tikzpicture}[mydot/.style={
    circle,
    thick,
    fill=white,
    draw,
    outer sep=0.5pt,
    inner sep=1pt
  }]
\tkzDefPoint(0,0){O}\tkzDefPoint(-1,1.5){1} \tkzDefPoint(1,1.5){2} \tkzDefPoint(-3,-1){3} \tkzDefPoint(3,-1){4}
\tkzDefPoint(0,1.1){5}

\tkzDefMidPoint(1,3)
\tkzGetPoint{i}
\tkzDefMidPoint(1,2)
\tkzGetPoint{j}
\tkzDefMidPoint(2,4)
\tkzGetPoint{k}
\tkzDefMidPoint(i,j)
\tkzGetPoint{a}
\tkzDefMidPoint(j,k)
\tkzGetPoint{b}
\draw[line width=0.4mm,dark-green](1) edge (2);
\draw[line width=0.4mm,dark-green](1) edge (3);
\draw[line width=0.4mm,dark-green](2) edge (4);
\draw[->, >= angle 60,line width=0.4mm,black] (i) -- (j);
\draw[->, >= angle 60,line width=0.4mm,black] (j) -- (k);
\tkzDrawArc[angles,line width=0.4mm,black,dashed](j,5)(-150,-30)
\tkzDrawPoints[size=4,color=black](i,j,k);\tkzDrawPoints[size=4,color=dark-green,mydot](1,2);
\tkzLabelPoints[left](i)
\tkzLabelPoints[above](j)
\tkzLabelPoints[right](k)
\tkzLabelPoint[below](a){$\alpha$}
\tkzLabelPoint[below](b){$\beta$}
\end{tikzpicture} \caption{\label{fig:rules_quiver_dissections} Drawings representing the construction rules of $Q(\Delta^{\gpoint})$ and $R(\Delta^\gpoint)$ described above; (2) on the left and (3) on the right.} \end{center} \end{figure}
\begin{example}\label{ex:quiverofsurfaces}
See in \cref{fig:twosurfquiver} the quivers with relations associated with the dissection drawn in  \cref{fig:dissections}.
 
 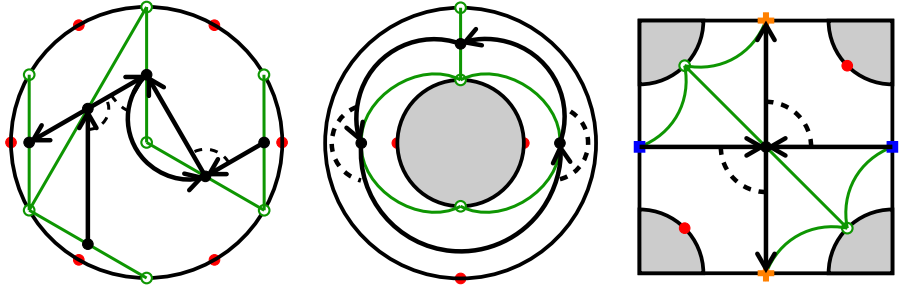
\begin{figure}[h!]
 \centering 
 \begin{tikzpicture}[mydot/.style={
    circle,
    thick,
    fill=white,
    draw,
    outer sep=0.5pt,
    inner sep=1pt
  },scale=1.2]
\tkzDefPoint(0,0){O}\tkzDefPoint(1.5,0){12}
\tkzDefPointsBy[rotation=center O angle 360/12](12,1,2,3,4,5,6,7,8,9,10){1,2,3,4,5,6,7,8,9,10,11}
\tkzDrawPoints[fill =red,size=4,color=red](2,4,6,8,10,12)
\tkzDrawCircle[line width=0.5mm,black](O,7)

\tkzDrawPoints[size=4,color=dark-green,mydot](O,1,3,5,7,9,11)

\draw[line width=0.4mm,dark-green](1) edge (11);
\draw[line width=0.4mm,dark-green](11) edge (O);
\draw[line width=0.4mm ,dark-green](O) edge (3);
\draw[line width=0.4mm,dark-green](3) edge (7);
\draw[line width=0.4mm,dark-green](7) edge (5);
\draw[line width=0.4mm,dark-green](7) edge (9);

\tkzDefMidPoint(1,11) \tkzGetPoint{a}
\tkzDefMidPoint(11,O) \tkzGetPoint{b}
\tkzDefMidPoint(O,3) \tkzGetPoint{c}
\tkzDefMidPoint(3,7) \tkzGetPoint{d}
\tkzDefMidPoint(7,5) \tkzGetPoint{e}
\tkzDefMidPoint(7,9) \tkzGetPoint{f}
\tkzDrawPoints[size=4,fill=black,color=black](a,b,c,d,e,f)

\draw[->, >= angle 60,line width=0.6mm,black](a) -- (b);
\draw[->, >= angle 60, line width=0.6mm,black](b) -- (c);
\draw[->, >= angle 60,line width=0.6mm,black](d) -- (c);
\draw[->, >= angle 60,line width=0.6mm,black](f) -- (d);
\draw[->, >= angle 60,line width=0.6mm,black](d) -- (e);
\tkzDefPoint(-0.1,-0.1){D}
\tkzCircumCenter(b,c,D)\tkzGetPoint{O1}
\tkzDrawArc[->,>= angle 60,line width=0.6mm,black](O1,c)(b)

\tkzDefBarycentricPoint(a=2,b=3) \tkzGetPoint{M}
\tkzDefBarycentricPoint(c=2,b=3) \tkzGetPoint{N}
\tkzDrawArc[line width=0.4mm,black,dashed](b,M)(N)
\tkzDefBarycentricPoint(f=2,d=11) \tkzGetPoint{P}
\tkzDefBarycentricPoint(c=2,d=3) \tkzGetPoint{Q}
\tkzDrawArc[line width=0.4mm,black,dashed](d,P)(Q)

\tkzDefBarycentricPoint(c=2,d=3) \tkzGetPoint{R}
\tkzDefPoint(-0.2,0.2){S}
\tkzDrawArc[line width=0.4mm,black,dashed](c,R)(S)
\end{tikzpicture} \quad
\begin{tikzpicture}[mydot/.style={
    circle,
    thick,
    fill=white,
    draw,
    outer sep=0.5pt,
    inner sep=1pt,
    },scale = 1.2]
\tkzDefPoint(0,0){O}\tkzDefPoint(0,1.5){1}
\tkzDefPointsBy[rotation=center O angle 180](1){2}
\tkzDefPoint(0.7,0){3} \tkzDefPoint(1.1,0){8} \tkzDefPoint(-1.1,0){9}
\tkzDefPointsBy[rotation=center O angle 90](3,4,5){4,5,6}
\tkzDrawPoints[fill =red,size=4,color=red](2,3,5)
\tkzDrawCircle[line width=0.5mm,black](O,1)
\tkzDrawCircle[line width=0.5mm,black,fill=gray!40](O,3)
\tkzDrawPoints[size=4,color=dark-green,mydot](1,4,6)

\draw[line width=0.4mm,dark-green](1) edge (4);
\draw[line width=0.4mm,bend left = 60,dark-green](4) edge (8);
\draw[line width=0.4mm,bend left = 60,dark-green](8) edge (6);
\draw[line width=0.4mm,bend right = 60,dark-green](4) edge (9);
\draw[line width=0.4mm ,bend right =60,dark-green](9) edge (6);

\tkzDefPoint(0,1.1){a}
\tkzDefPoint(1.1,0){b}
\tkzDefPoint(-1.1,0){c}
\tkzDrawPoints[size=4,fill=black,color=black](a,b,c)

\tkzDefPoint(0,-1.2){D}
\tkzCircumCenter(b,c,D)\tkzGetPoint{O1}
\tkzDrawArc[->, >= angle 60,line width=0.6mm,black](O1,c)(b)
\tkzDefPoint(0.9,0.9){E}
\tkzCircumCenter(a,b,E)\tkzGetPoint{O1}
\tkzDrawArc[->, >= angle 60,line width=0.6mm,black](O1,b)(a)
\tkzDefPoint(-0.9,0.9){F}
\tkzCircumCenter(a,c,F)\tkzGetPoint{O1}
\tkzDrawArc[->, >= angle 60,line width=0.6mm,black](O1,a)(c)

\tkzDefPoint(-1.15,0.4){A}
\tkzDefPoint(-1.1,-0.4){B}
\tkzDefPoint(-1,0){C}
\tkzDrawArc[line width=0.6mm,black,dashed](C,A)(B)
\tkzDefPoint(1.15,0.4){A}
\tkzDefPoint(1.1,-0.4){B}
\tkzDefPoint(1,0){C}
\tkzDrawArc[line width=0.6mm,black,dashed](C,B)(A)

\end{tikzpicture} \quad \begin{tikzpicture}[mydot/.style={
    circle,
    thick,
    fill=white,
    draw,
    outer sep=0.5pt,
    inner sep=1pt
  }, scale = 1.2]
\tkzDefPoint(0,0){O}\tkzDefPoint(1.4,1.4){1}
\tkzDefPointsBy[rotation=center O angle 90](1,2,3){2,3,4}
\tkzDefPoint(0.9,0.9){5}
\tkzDefPointsBy[rotation=center O angle 90](5,6,7){6,7,8}

\tkzDrawSquare[line width=0.5mm](1,2)
\tkzDefPointsBy[rotation=center 1 angle 45](5){9}
\tkzDrawSector[rotate,line width=0.5mm, black, fill = gray!40](1,9)(-90)
\tkzDefPointsBy[rotation=center 2 angle 45](6){10}
\tkzDrawSector[rotate,line width=0.5mm, black, fill = gray!40](2,10)(-90)
\tkzDefPointsBy[rotation=center 3 angle 45](7){11}
\tkzDrawSector[rotate,line width=0.5mm, black, fill = gray!40](3,11)(-90)
\tkzDefPointsBy[rotation=center 4 angle 45](8){12}
\tkzDrawSector[rotate,line width=0.5mm, black, fill = gray!40](4,12)(-90)

\tkzDefPoint(1.4,0){13}
\tkzDefPoint(-1.4,0){14}
\tkzDefPoint(0,1.4){15}
\tkzDefPoint(0,-1.4){16}

\tkzDrawPoints[fill =red,size=4,color=red](5,7)
\tkzDrawPoints[size=4,color=dark-green,mydot](6,8)
\draw[line width=0.4mm,dark-green](6) edge (8);
\draw[line width=0.4mm,bend left=40,dark-green](6) edge (14);
\draw[line width=0.4mm,bend right=40,dark-green](6) edge (15);
\draw[line width=0.4mm,bend left=40,dark-green](8) edge (13);
\draw[line width=0.4mm,bend right=40,dark-green](8) edge (16);

\tkzDefPoint(0,0){a}
\tkzDefPoint(1.4,0){b}
\tkzDefPoint(-1.4,0){c}
\tkzDefPoint(0,1.4){d}
\tkzDefPoint(0,-1.4){e}
\tkzDrawPoint[size=4,fill=black,color=black](a)
\tkzDrawPoints[rectangle,size=4,fill=blue,color=blue](b,c)
\tkzDrawPoints[cross,size=4,fill=orange,color=orange,line width=0.8mm](d,e)

\draw[->, >= angle 60,line width=0.6mm,black](b) -- (a);
\draw[->, >= angle 60,line width=0.6mm,black](c) -- (a);
\draw[->, >= angle 60,line width=0.6mm,black](a) -- (d);
\draw[->, >= angle 60,line width=0.6mm,black](a) -- (e);

\tkzDefPoint(0.5,0){B}
\tkzDefPoint(0,0.5){D}
\tkzDefPoint(-0.5,0){C}
\tkzDefPoint(0,-0.5){E}
\tkzDrawArc[line width=0.6mm,black,dashed](a,B)(D)
\tkzDrawArc[line width=0.6mm,black,dashed](a,C)(E)

\end{tikzpicture}

\caption{\label{fig:twosurfquiver} Quivers with relations arising from surface in \cref{fig:dissections}}
\end{figure} 

Firstly, in the rightmost example, we have to identify the blue square points and the orange cross points. This implies that the quiver associated to this dissection is exactly the gentle quiver $(Q_3,R_3)$.
$$\begin{tikzpicture}[->]
  \node (a) at (0,0) {${\color{orange}{1}}$};
  \node (b) at (2,0) {$2$};
  \node (c) at (4,0) {${\color{blue}{3}}$};
  \draw ([yshift=1mm]b.west)--node[above]{$\beta_{1}$}([yshift=1mm]a.east);
  \draw  ([yshift=-1mm]b.west)--node[below]{$\alpha_{1}$}([yshift=-1mm]a.east);
    \draw ([yshift=1mm]c.west)--node[above]{$\alpha_{2}$}([yshift=1mm]b.east);
  \draw  ([yshift=-1mm]c.west)--node[below]{$\beta_{2}$}([yshift=-1mm]b.east);
  \draw[dashed,-] ([xshift=.4cm]b.north) arc[start angle = 0, end angle = 180, x radius=.4cm, y radius =.2cm];
  \draw[dashed,-] ([xshift=-.4cm]b.south) arc[start angle = 180, end angle = 360, x radius=.4cm, y radius =.2cm];
  
\end {tikzpicture}
 $$

Secondly, note that we can have configurations which do not give us a gentle algebra, like for the leftmost example. In fact, any quiver with relations $(Q(\Delta^{\gpoint}), R(\Delta^{\gpoint}))$ containing an oriented cycle in $Q(\Delta^{\gpoint})$ without any of its length two paths in $R(\Delta^{\gpoint})$ defines an infinite-dimensional algebra, which by definition is not gentle.
This case arises precisely when there is a green marked point in the interior. The following result confirms that, if we consider surfaces with dissection $\Delta^{\gpoint}$ with no green marked points in the interior, then we produce all gentle algebras, and only gentle algebras.
\end{example}

\begin{thm}[\cite{BCS, OPS, PPP}] $ $

 \begin{enumerate}
  \item If~$\Delta^{\gpoint}$ is a~$\gpoint$-dissection of the marked surface~$(S,M)$ such that there are no green marked points in the interior of $S$, then~$\left(Q(\Delta^{\gpoint}), R(\Delta^{\gpoint})\right)$ is a gentle quiver.
  
  \item For any gentle quiver ~$(Q,R)$, there exists a marked surface~$(S,M)$ without green marked point in the interior of $S$ and with a~$\gpoint$-dissection~$\Delta^{\gpoint}$ such that~$(Q,R)$ is isomorphic to~$\left(Q(\Delta^{\gpoint}), R(\Delta^{\gpoint})\right)$.  Moreover,~$(S,M)$ and~$\Delta^{\gpoint}$ are unique up to oriented homeomorphism of surfaces and homotopy of arcs.
 \end{enumerate}
\end{thm}
An explicit recipe to construct~$(S,M)$ and~$\Delta^{\gpoint}$ from the gentle quiver~$(Q,R)$ is given in~\cite[Section 4.2]{PPP}. 

\subsection{Accordions and bands}
The representations of gentle quivers are related to certain arcs on the associated surface: these are the so-called \newword{accordions}.

Fix a marked surface~$(S,M)$, endowed with a~$\gpoint$-dissection~$\Delta^{\gpoint}$, and let~$(Q(\Delta^{\gpoint})$, $R(\Delta^{\gpoint}))$ be the associated gentle quiver.

A \newword{closed curve} in $S$ is a non-contractible curve contained in the interior of~$S$ avoiding the marked points of $M$; formally, it is a non-contractible continuous map from the circle~$S^1$ to~$S\setminus (\partial S \cup M)$. 
We say that a closed curve
is \newword{simple} if it does not intersect itself.
A closed curve~$\gamma$ is \newword{primitive} if there does not exist a closed curve~$\gamma_0$ such that~$\gamma$ is the concatenation of~$\gamma_0$ with itself~$m\geq 2$ times (up to homotopy).

\begin{definition}\label{def:closedAccordion}
A \newword{closed accordion} for~$(S,M,\Delta^{\gpoint})$ is a primitive closed curve~$\gamma$ on~$S$ which satisfies the following restrictions: Whenever~$\gamma$ enters a cell of~$\Delta^{\gpoint}$ by crossing an arc~$\alpha$,
\begin{enumerate}[(a)]
 \item it leaves the cell by crossing an arc~$\beta$ adjacent to~$\alpha$;
 \item the relevant segments of the arcs~$\alpha$,~$\beta$ and~$\gamma$ bound a disk that does not contain the unique marked point in~$M_{\rpoint}$ belonging to the cell.
\end{enumerate}
See \cref{fig:rules_closedaccord} for a picture illustrating the rules $(a)$ and $(b)$. 
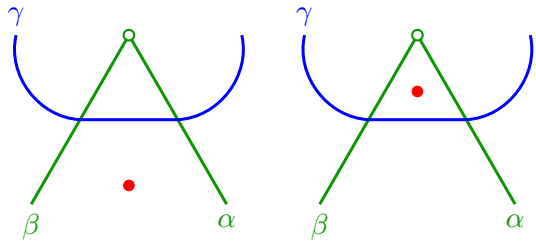
\begin{figure}[h!] 
\begin{center}
\begin{tikzpicture}[mydot/.style={
    circle,
    thick,
    fill=white,
    draw,
    outer sep=0.5pt,
    inner sep=1pt
  }]
\tkzDefPoint(0,0){O}\tkzDefPoint(0,1.5){1}
\tkzDefPointsBy[rotation=center O angle 120](1,2){2,3}
\tkzDefMidPoint(1,2)
\tkzGetPoint{i}
\tkzDefMidPoint(1,3)
\tkzGetPoint{j}
\tkzDefPoint(-1.5,1.5){4}
\tkzDefPoint(1.5,1.5){5}

\draw[line width=0.4mm,dark-green](1) edge (2);
\draw[line width=0.4mm,dark-green](1) edge (3);
\draw[line width=0.4mm,blue] (i) edge (j);
\draw[line width=0.4mm, bend right=50,blue] (4) edge (i);
\draw[line width=0.4mm,bend left=50,blue] (5) edge (j);

\tkzDrawPoints[size=4,color=dark-green,mydot](1);

\tkzDefPoint(0,-0.5){6}
\tkzDrawPoints[size=4,color=red](6);

\tkzLabelPoint[above](4){{\large $\color{blue}{\gamma}$}}
\tkzLabelPoint[below](3){{\large $\color{dark-green}{\alpha}$}}
\tkzLabelPoint[below](2){{\large $\color{dark-green}{\beta}$}}
\end{tikzpicture} \quad
\begin{tikzpicture}[mydot/.style={
    circle,
    thick,
    fill=white,
    draw,
    outer sep=0.5pt,
    inner sep=1pt
  }]
\tkzDefPoint(0,0){O}\tkzDefPoint(0,1.5){1}
\tkzDefPointsBy[rotation=center O angle 120](1,2){2,3}
\tkzDefMidPoint(1,2)
\tkzGetPoint{i}
\tkzDefMidPoint(1,3)
\tkzGetPoint{j}
\tkzDefPoint(-1.5,1.5){4}
\tkzDefPoint(1.5,1.5){5}

\draw[line width=0.4mm,dark-green](1) edge (2);
\draw[line width=0.4mm,dark-green](1) edge (3);
\draw[line width=0.4mm,blue] (i) edge (j);
\draw[line width=0.4mm, bend right=50,blue] (4) edge (i);
\draw[line width=0.4mm,bend left=50,blue] (5) edge (j);

\tkzDrawPoints[size=4,color=dark-green,mydot](1);

\tkzDefPoint(0,0.75){6}
\tkzDrawPoints[size=4,color=red](6);

\tkzLabelPoint[above](4){{\large $\color{blue}{\gamma}$}}
\tkzLabelPoint[below](3){{\large $\color{dark-green}{\alpha}$}}
\tkzLabelPoint[below](2){{\large $\color{dark-green}{\beta}$}}
\end{tikzpicture} \caption{\label{fig:rules_closedaccord} Drawings representing the rules that an accordion must satisfy: on the left, $\gamma$ satisfies the rule $(a)$ and $(b)$; on the right, $\gamma$ does not satisfy the rule $(b)$.} \end{center} \end{figure}
\end{definition}

\begin{example} \label{exam:closedaccord} See \cref{fig:closedaccord} for an example of a closed accordion over a torus. 
\begin{figure}[h!]
\centering
\begin{tikzpicture}[mydot/.style={
    circle,
    thick,
    fill=white,
    draw,
    outer sep=0.5pt,
    inner sep=1pt
  }, scale = 1]
\tkzDefPoint(0,0){O}\tkzDefPoint(1.4,1.4){1}
\tkzDefPointsBy[rotation=center O angle 90](1,2,3){2,3,4}
\tkzDefPoint(0.9,0.9){5}
\tkzDefPointsBy[rotation=center O angle 90](5,6,7){6,7,8}

\tkzDrawSquare[line width=0.5mm](1,2)
\tkzDefPointsBy[rotation=center 1 angle 45](5){9}
\tkzDrawSector[rotate,line width=0.5mm, black, fill = gray!40](1,9)(-90)
\tkzDefPointsBy[rotation=center 2 angle 45](6){10}
\tkzDrawSector[rotate,line width=0.5mm, black, fill = gray!40](2,10)(-90)
\tkzDefPointsBy[rotation=center 3 angle 45](7){11}
\tkzDrawSector[rotate,line width=0.5mm, black, fill = gray!40](3,11)(-90)
\tkzDefPointsBy[rotation=center 4 angle 45](8){12}
\tkzDrawSector[rotate,line width=0.5mm, black, fill = gray!40](4,12)(-90)

\tkzDefPoint(1.4,0){13}
\tkzDefPoint(-1.4,0){14}
\tkzDefPoint(0,1.4){15}
\tkzDefPoint(0,-1.4){16}

\tkzDrawPoints[fill =red,size=4,color=red](5,7)
\tkzDrawPoints[size=4,color=dark-green,mydot](6,8)
\draw[line width=0.4mm,dark-green](6) edge (8);
\draw[line width=0.4mm,bend left=40,dark-green](6) edge (14);
\draw[line width=0.4mm,bend right=40,dark-green](6) edge (15);
\draw[line width=0.4mm,bend left=40,dark-green](8) edge (13);
\draw[line width=0.4mm,bend right=40,dark-green](8) edge (16);

\draw[line width=0.6mm,blue](-1.4,-0.5) edge (1.4,0.5);
\draw[line width=0.6mm,bend right=40,blue](-1.4,0.5) edge (0,1.4);
\draw[line width=0.6mm,bend left=40,blue](0,-1.4) edge (1.4,-0.5);
\end{tikzpicture}

\caption{\label{fig:closedaccord} An example of a closed accordion (blue) over a \mbox{$\gpoint$-dissection} of the torus, see \cref{exam::three-surfaces}.}
\end{figure}
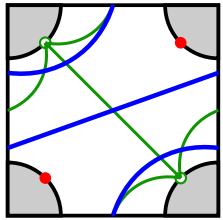
\end{example}

\begin{proposition}
 The map sending a band to the closed curve on~$S$ that follows the arrows of~$Q(\Delta^{\gpoint})$ drawn on~$S$ induces a bijection: 
 \begin{center}
 $\xymatrix{\{\text{\emph{bands for }} (Q(\Delta^{\gpoint}),R(\Delta^{\gpoint}))\} \ar@{<->}[d]^{1:1}\\
 \{\text{\emph{homotopy classes of simple closed accordions for}}~(S,M,\Delta^{\gpoint})\} }$
 \end{center} 
\end{proposition}

\begin{proof}
 The proof is essentially the same as that of~\cite[Proposition 4.23]{PPP}, where the case of strings is considered. The key point is that a closed accordion is uniquely determined up to homotopy by its sequence of points of intersection with the dissection~$\Delta^{\gpoint}$.
\end{proof}

\subsection{The dual dissection} \label{sub::dual}

For any $\gpoint$-dissection~$\Delta^{\gpoint}$, there exists a unique~$\rpoint$-dissection~$\Delta^{\rpoint}$ such that each arc of~$\Delta^{\gpoint}$ intersects exacly one arc of~$\Delta^{\rpoint}$ and vice versa. Specifically, each~$\gpoint$-arc~$\gamma$ in~$\Delta^{\gpoint}$ borders two cells, each of which contains precisely one marked point in~$M_{\rpoint}$. Draw a~$\rpoint$-arc joining those two marked points and crossing~$\gamma$.
The dissections~$\Delta^{\gpoint}$ and~$\Delta^{\rpoint}$ are said to be \newword{dual dissections}.  

As shown in~\cite[Proposition 3.13]{PPP}, closed accordions for~$(S,M,\Delta^{\gpoint})$ can equivalently be viewed as closed slaloms for~$\Delta^{\rpoint}$.

\begin{definition}
 A \newword{closed slalom} is a primitive closed curve~$\gamma$ on~$(S,M)$ such that, for any two cells of~$\Delta^{\rpoint}$ that~$\gamma$ intersects consecutively, their $\gpoint$-marked points in~$M_{\rpoint}$ are on opposite sides of~$\gamma$. 
\end{definition}

We say that curves~$C_1, \ldots, C_r$ are in \newword{minimal position} if they are chosen in their homotopy classes in such a way that~$C_i$ and~$C_j$ intersect transversally and a minimal number of times.  This choice is always possible, as explained for instance in the section {\it multicurves} before Section 1.2.5 of \cite{FM}.
Thus we will always assume that slaloms with respect to a dissection $\Delta^{\rpoint}$ are always drawn so as to intersect $\Delta^{\rpoint}$ as few times as
possible. 

\begin{example}\label{ex:closedslalom}
We can check that the closed accordion given in \cref{exam:closedaccord} for $\Delta^{\gpoint}$ is a closed slalom for $\Delta^{\rpoint}$. 
\begin{figure}[h!]
\centering
\begin{tikzpicture}[mydot/.style={
    circle,
    thick,
    fill=white,
    draw,
    outer sep=0.5pt,
    inner sep=1pt
  }, scale = 1]
\tkzDefPoint(0,0){O}\tkzDefPoint(1.4,1.4){1}
\tkzDefPointsBy[rotation=center O angle 90](1,2,3){2,3,4}
\tkzDefPoint(0.9,0.9){5}
\tkzDefPointsBy[rotation=center O angle 90](5,6,7){6,7,8}

\tkzDrawSquare[line width=0.5mm](1,2)
\tkzDefPointsBy[rotation=center 1 angle 45](5){9}
\tkzDrawSector[rotate,line width=0.5mm, black, fill = gray!40](1,9)(-90)
\tkzDefPointsBy[rotation=center 2 angle 45](6){10}
\tkzDrawSector[rotate,line width=0.5mm, black, fill = gray!40](2,10)(-90)
\tkzDefPointsBy[rotation=center 3 angle 45](7){11}
\tkzDrawSector[rotate,line width=0.5mm, black, fill = gray!40](3,11)(-90)
\tkzDefPointsBy[rotation=center 4 angle 45](8){12}
\tkzDrawSector[rotate,line width=0.5mm, black, fill = gray!40](4,12)(-90)

\tkzDefPoint(1.4,0){13}
\tkzDefPoint(-1.4,0){14}
\tkzDefPoint(0,1.4){15}
\tkzDefPoint(0,-1.4){16}

\tkzDrawPoints[fill =red,size=4,color=red](5,7)
\tkzDrawPoints[size=4,color=dark-green,mydot](6,8)
\draw[line width=0.4mm,dark-green](6) edge (8);
\draw[line width=0.4mm,bend left=40,dark-green](6) edge (14);
\draw[line width=0.4mm,bend right=40,dark-green](6) edge (15);
\draw[line width=0.4mm,bend left=40,dark-green](8) edge (13);
\draw[line width=0.4mm,bend right=40,dark-green](8) edge (16);

\draw[line width=0.4mm,red](5) edge (7);
\draw[line width=0.4mm,bend left=40,red](5) edge (15);
\draw[line width=0.4mm,bend right=40,red](5) edge (13);
\draw[line width=0.4mm,bend right=40,red](7) edge (14);
\draw[line width=0.4mm,bend left=40,red](7) edge (16);

\draw[line width=0.6mm,blue](-1.4,-0.5) edge (1.4,0.5);
\draw[line width=0.6mm,bend right=40,blue](-1.4,0.5) edge (0,1.4);
\draw[line width=0.6mm,bend left=40,blue](0,-1.4) edge (1.4,-0.5);
\end{tikzpicture}

\caption{\label{fig:closedslalom} An example of a closed slalom (blue) over a \mbox{$\rpoint$-dissection} of the torus as already seen in \cref{exam::three-surfaces}.}
\end{figure}
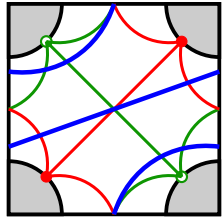
\end{example}

\begin{proposition}[Proposition 3.13 in \cite{PPP}]
 A closed curve on~$(S,M)$ is a closed accordion for~$\Delta^{\gpoint}$ if and only if it is a closed slalom for~$\Delta^{\rpoint}$.
\end{proposition}

\begin{remark}\label{rema::closed-slalom}
Note that a closed slalom crosses a~$\rpoint$-arc in~$\Delta^{\rpoint}$ precisely when the associated band reaches a deep (i.e. a subword of form~$\beta^{-1}\alpha$) or a peak (i.e. a subword of form~$\alpha\beta^{-1}$).
\end{remark}

\subsection{Band modules and morphisms}\label{ssec:band-morphisms}

In what follows, we will be interested in band modules which are also bricks.
We refer to these as \newword{band bricks}. Before looking at them, we need to recall a result of H.~Krause on morphisms between band modules \cite{Krause} and discuss its interpretation in terms of intersections of curves.

Let~$(Q,R)$ be a gentle quiver.  Let~$w = \alpha_r^{\varepsilon_r} \cdots \alpha_1^{\varepsilon_1}$ be a string  on~$(Q,R)$. A \newword{substring}~$\rho$ of~$w$ is either a string of the form $\displaystyle \rho = \alpha_b^{\varepsilon_b} \cdots \alpha_a^{\varepsilon_a}$ for $0 \leq a \leq  b\leq r$ or a lazy path $\rho = e_{v}$ where $v$ is a vertex in the support of $w$.  We say that~$\rho$ is \newword{at the bottom of~$w$} if
\begin{itemize}
 \item either~$a=1$ or $\varepsilon_{a-1} = 1$, and 
 \item either~$b=r$ or $\varepsilon_{b+1} = -1$.
\end{itemize}
Dually, we say that~$\rho$ is \newword{on top of~$w$} if
\begin{itemize}
 \item either~$a=1$ or~$\varepsilon_{a-1} = -1$, and 
 \item either~$b=r$ or~$\varepsilon_{b+1} = 1$.
\end{itemize}
The terminology is explained by the following diagrams; on the left we represent the cases where $\rho$ is on the top of $w$ and and on the right we show the cases where $\rho$ is at the bottom of $w$. The parts of the pictures enclosed in parentheses are optional.  
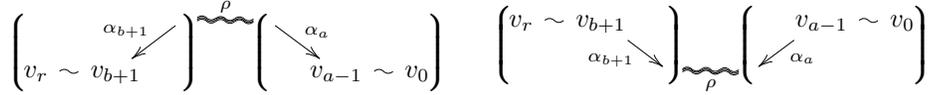
\begin{center}
$\displaystyle \xymatrix@C=1em @R=1em{
& & \ar[ld]_{\alpha_{b+1}}  \ar@{~}[rr]^{\mathbf{\rho}} \ar@{~} @<-0.2mm>[rr] \ar@{~} @<+0.2mm>[rr]& &  \ar[rd]^{\alpha_{a}} & & \\
v_r \ar@{~}[r]  & v_{b+1} & & & & v_{a-1}\ar@{~}[r] & v_0 
\save "2,1"."1,3"\POS!C*\frm{(},+L*++!R{}
\restore
\save "2,1"."1,3"\POS!C*\frm{)},+R*++!L{}
\restore
\save "1,5"."2,7"\POS!C*\frm{(},+L*+!R{}
\restore
\save "1,5"."2,7"\POS!C*\frm{)},+R*++!L{}
\restore}  \displaystyle \xymatrix@C=1em @R=1em{
v_r \ar@{~}[r] & v_{b+1} \ar[rd]_{\alpha_{b+1}} & & & & \ar[ld]^{\alpha_{a}} v_{a-1} \ar@{~}[r] & v_0 \\
& &  \ar@{~}[rr]_{\mathbf{\rho}} \ar@{~} @<-0.2mm>[rr] \ar@{~} @<+0.2mm>[rr]& &   & &
\save "1,1"."2,3"\POS!C*\frm{(},+L*++!R{}
\restore
\save "1,1"."2,3"\POS!C*\frm{)},+R*++!L{}
\restore
\save "2,5"."1,7"\POS!C*\frm{(},+L*+!R{}
\restore
\save "2,5"."1,7"\POS!C*\frm{)},+R*++!L{}
\restore}$
\end{center}

\begin{example}\label{exam::topbot} 
 Take the same gentle quiver as in \cref{ex::stringrep}:
 
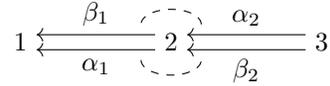
 $$\begin{tikzpicture}[->]
  \node (a) at (0,0) {$1$};
  \node (b) at (2,0) {$2$};
  \node (c) at (4,0) {$3$};
  \draw ([yshift=1mm]b.west)--node[above]{$\beta_{1}$}([yshift=1mm]a.east);
  \draw  ([yshift=-1mm]b.west)--node[below]{$\alpha_{1}$}([yshift=-1mm]a.east);
    \draw ([yshift=1mm]c.west)--node[above]{$\alpha_{2}$}([yshift=1mm]b.east);
  \draw  ([yshift=-1mm]c.west)--node[below]{$\beta_{2}$}([yshift=-1mm]b.east);
  \draw[dashed,-] ([xshift=.4cm]b.north) arc[start angle = 0, end angle = 180, x radius=.4cm, y radius =.2cm];
  \draw[dashed,-] ([xshift=-.4cm]b.south) arc[start angle = 180, end angle = 360, x radius=.4cm, y radius =.2cm];
  
\end {tikzpicture}
 $$

Consider the string~$w=\alpha_1\alpha_2\beta_2^{-1}\alpha_2\beta_2^{-1}\alpha_2$, as shown in Figure \ref{fig:strins_bot_top}. The left-most occurrence of the substring~$\alpha_2\beta_2^{-1}\alpha_2$ can easily be seen to be on top of~$w$. Similarly, the rightmost occurrence of~$\alpha_2\beta_2^{-1}\alpha_2$ can be seen to be at the bottom of~$w$.

\begin{figure}[h!]
\begin{center} 
$\displaystyle \xymatrix@R=1em@C=1em{
 & & {\color{red}{\mathbf{3}}} \ar@<-0.1ex>@[red][ld] \ar@<0.1ex>@[red][ld] \ar@[red][ld]_{\color{red}{\alpha_2}} \ar@<-0.1ex>@[red][rd] \ar@<0.1ex>@[red][rd] \ar@[red][rd]_{\color{red}{\beta_2}} & & {\color{red}{\mathbf{3}}} \ar@<-0.1ex>@[red][ld] \ar@<0.1ex>@[red][ld] \ar@[red][ld]_{\color{red}{\alpha_2}} \ar[rd]_{\beta_2} &  &  3 \ar[ld]_{\alpha_2} \\
 & {\color{red}{\mathbf{2}}} \ar[ld]_{\alpha_1} & & {\color{red}{\mathbf{2}}}& & 2   & \\
 1 & & & & & &  }$ \qquad
$\displaystyle \xymatrix@R=1em@C=1em{
 & & 3 \ar[ld]_{\alpha_2} \ar[rd]_{\beta_2}   & & {\color{red}{\mathbf{3}}} \ar@<-0.1ex>@[red][ld] \ar@<0.1ex>@[red][ld] \ar@[red][ld]_{\color{red}{\alpha_2}} \ar@<-0.1ex>@[red][rd] \ar@<0.1ex>@[red][rd] \ar@[red][rd]_{\color{red}{\beta_2}} &  &  {\color{red}{\mathbf{3}}} \ar@<-0.1ex>@[red][ld] \ar@<0.1ex>@[red][ld] \ar@[red][ld]_{\color{red}{\alpha_2}} \\
 & 2 \ar[ld]_{\alpha_1} & & {\color{red}{\mathbf{2}}}& & {\color{red}{\mathbf{2}}}   & \\
 1 & & & & & &  }$
 \caption{\label{fig:strins_bot_top} A representation of the string $\rho$ and of the string $\sigma = \alpha_2 \beta_2^{-1} \alpha_2$ appearing as a substring on the top (left) and at the bottom (right) of $\rho$.}
 \end{center}
\end{figure}
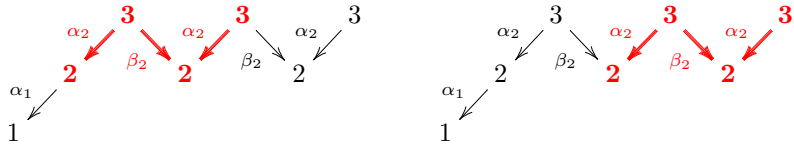

\end{example}

Substrings allow us to visualize the \newword{socle} of a representation.  Recall that the socle of~$M$ is the sum of all simple submodules of~$M$; if~$M$ is a string or band module, then its socle is obtained by considering all the substrings of length zero at the bottom of the defining string or band of~$M$. 

Substrings are also very useful in determining morphisms between string modules and band modules. The case of string modules was determined in \cite{CB}; in this paper, we will only need the result below, which applies to band modules\footnote{The theorem in \cite{Krause} is formulated in terms of ``admissible triples'' which involve morphisms from quivers of type~$A$ or~$\widetilde A$ into~$Q$; in the special case of band modules, we will give an equivalent formulation in terms of substrings of the universal cover of bands.}.  Before stating the theorem, some notations are needed.

Let~$z$ be a band walk.  Let~$z^\infty_\infty$ be the word, infinite on the left and on the right, obtained by concatenation of countably many copies of~$z$.  We view~$z^\infty_\infty$ as a ``universal cover'' of~$z$; in particular, there is an action of~$\bZ$ on the set of finite subwords of $z^\infty_\infty$ by translation.  We will consider subwords of~$z^\infty_\infty$ only up to this~$\bZ$-action.  Define~$\Sigma_{\bt}^\bZ(z)$ and~$\Sigma_{\tp}^\bZ(z)$ to be the set of substrings (up to~$\bZ$-action) at the bottom of~$z^\infty_\infty$ and on top of~$z^\infty_\infty$, respectively.

\begin{thm}[\cite{Krause}]\label{theo::band-morphisms}
 Let~$z,z'$ be two band walks,~$m,m'$ be positive integers and~$\lambda, \lambda'$ be two elements in~$k$.  Consider the band modules~$B_{z,m,\lambda}$ and~$B_{z',m',\lambda'}$ (see Section \ref{ssec::band} for definitions). 
 \begin{enumerate}
  \item\label{non-iso} If~$B_{z,1,\lambda}$ and~$B_{z',1,\lambda'}$ are not isomorphic, then a basis of the space of morphisms~$\Hom\left(B_{z,m,\lambda}, B_{z',m',\lambda'}\right)$ is in bijection with the set
  \[
   \left\{ (v,v') \in \Sigma_{\tp}^\bZ(z) \times \Sigma_{\bt}^\bZ(z') \ \big| \ [v]=[v'] \textrm{ or } [\widetilde v] = [v'] \right\} \times V_{m,m'},
  \]
  where~$V_{m,m'}$ is a basis of the space~$\Hom_k(k^{m}, k^{m'})$, whose dimension is~$mm'$, and $\widetilde v$ is the reversal of $v$.
  
  \item If~$B_{z,1,\lambda}$ and~$B_{z',1,\lambda'}$ are isomorphic, then a basis of the space of morphisms~$\Hom\left(B_{z,m,\lambda}, B_{z',m',\lambda'}\right)$ is in bijection to the set in (\ref{non-iso}) union a basis of the space~$\Hom_{k[t]}\left(k[t]/(t^m), k[t]/(t^{m'}) \right)$, whose dimension is~$\min(m,m')$.

 \end{enumerate}
\end{thm}

The set appearing in Theorem~\ref{theo::band-morphisms} has a nice interpretation in terms of crossings of curves.  

\begin{lemma}\label{lem::crossings-are-morphisms}
 Let~$(Q,R)$ be a gentle quiver with corresponding dissected surface~$(S,M,\Delta^{\gpoint})$.  Let~$z$ and~$z'$ be bands on~$(Q,R)$ with corresponding closed accordions~$\gamma$ and~$\gamma'$ on~$(S,M,\Delta^{\gpoint})$.  Assume that~$\gamma$ and~$\gamma'$ are in a minimal position. Let
  \[
   K(z,z') := \left\{ (v,v') \in \Sigma_{\tp}^\bZ(z) \times \Sigma_{\bt}^\bZ(z') \ \big| \ [v]=[v'] \textrm{ or } [\widetilde v] = [v'] \right\}.
  \]
  Then the cardinality of the set~$K(z,z') \sqcup K(z',z)$ is the number of crossings between the curves~$\gamma$ and~$\gamma'$.
\end{lemma}
\begin{proof}
 The proof of~\cite[Lemma 4.24]{PPP}, done therein for arcs in~$S$ but applicable without change to closed curves, provides a bijection between the set of crossings between~$\gamma$ and~$\gamma'$ and the set~$K(z,z')\sqcup K(z',z)$.
\end{proof}

\begin{corollary} \label{cor : bands and curves}
 Let~$z$ be a band with a corresponding closed accordion~$\gamma$.  The band module~$B_{z,m,\lambda}$ is a brick if and only if~$m=1$ and~$\gamma$ does not intersect itself.
\end{corollary}

\begin{example}\label{ex::Corbandscurves} Let us consider again the following gentle quiver.

 $$\begin{tikzpicture}[->]
  \node (a) at (0,0) {$1$};
  \node (b) at (2,0) {$2$};
  \node (c) at (4,0) {$3$};
  \draw ([yshift=1mm]b.west)--node[above]{$\beta_{1}$}([yshift=1mm]a.east);
  \draw  ([yshift=-1mm]b.west)--node[below]{$\alpha_{1}$}([yshift=-1mm]a.east);
    \draw ([yshift=1mm]c.west)--node[above]{$\alpha_{2}$}([yshift=1mm]b.east);
  \draw  ([yshift=-1mm]c.west)--node[below]{$\beta_{2}$}([yshift=-1mm]b.east);
  \draw[dashed,-] ([xshift=.4cm]b.north) arc[start angle = 0, end angle = 180, x radius=.4cm, y radius =.2cm];
  \draw[dashed,-] ([xshift=-.4cm]b.south) arc[start angle = 180, end angle = 360, x radius=.4cm, y radius =.2cm];
  
\end {tikzpicture}
 $$
 
Its associated marked surface, equipped with a dualizable $\gpoint$-dissection $(S,M,\Delta^\gpoint)$ is  the torus dissection we already saw in \cref{exam::three-surfaces}.

Now consider the band $\displaystyle z =  \beta_2^{-1}\alpha_2\beta_2^{-1}\beta_1^{-1}\alpha_1\alpha_2$. The associated closed accordion $\gamma$ corresponds to the one drawn in \cref{fig:closedaccord}. Note that $\gamma$ does not intersect itself. Hence following \cref{cor : bands and curves}, $B_{z,1,\lambda}$ is a band brick over $\Lambda_3$ for $\lambda \in k^\times$.

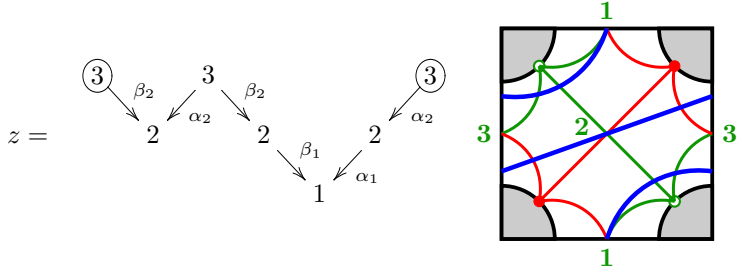
\begin{figure}[h!]
\centering
\begin{minipage}{6cm}
\begin{center}
$\displaystyle \xymatrix@C=1em @R=1em{ & *+[o][F-]{3} \ar[rd]^{\beta_2} & & 3 \ar[dl]^{\alpha_2}\ar[rd]^{\beta_2} & &  & & *+[o][F-]{3} \ar[ld]^{\alpha_2} \\
z = & & 2  & & 2 \ar[rd]^{\beta_1} & & 2\ar[dl]^{\alpha_1} & \\
& & & & & 1 & &} $ 
\end{center}
\end{minipage}
\begin{minipage}{4cm}
\begin{center}
\begin{tikzpicture}[mydot/.style={
    circle,
    thick,
    fill=white,
    draw,
    outer sep=0.5pt,
    inner sep=1pt
  }, scale = 1]
\tkzDefPoint(0,0){O}\tkzDefPoint(1.4,1.4){1}
\tkzDefPointsBy[rotation=center O angle 90](1,2,3){2,3,4}
\tkzDefPoint(0.9,0.9){5}
\tkzDefPointsBy[rotation=center O angle 90](5,6,7){6,7,8}

\tkzDrawSquare[line width=0.5mm](1,2)
\tkzDefPointsBy[rotation=center 1 angle 45](5){9}
\tkzDrawSector[rotate,line width=0.5mm, black, fill = gray!40](1,9)(-90)
\tkzDefPointsBy[rotation=center 2 angle 45](6){10}
\tkzDrawSector[rotate,line width=0.5mm, black, fill = gray!40](2,10)(-90)
\tkzDefPointsBy[rotation=center 3 angle 45](7){11}
\tkzDrawSector[rotate,line width=0.5mm, black, fill = gray!40](3,11)(-90)
\tkzDefPointsBy[rotation=center 4 angle 45](8){12}
\tkzDrawSector[rotate,line width=0.5mm, black, fill = gray!40](4,12)(-90)

\tkzDefPoint(1.4,0){13}
\tkzDefPoint(-1.4,0){14}
\tkzDefPoint(0,1.4){15}
\tkzDefPoint(0,-1.4){16}

\tkzDrawPoints[fill =red,size=4,color=red](5,7)
\tkzDrawPoints[size=4,color=dark-green,mydot](6,8)
\draw[line width=0.4mm,dark-green](6) edge (8);
\draw[line width=0.4mm,bend left=40,dark-green](6) edge (14);
\draw[line width=0.4mm,bend right=40,dark-green](6) edge (15);
\draw[line width=0.4mm,bend left=40,dark-green](8) edge (13);
\draw[line width=0.4mm,bend right=40,dark-green](8) edge (16);

\draw[line width=0.4mm,red](5) edge (7);
\draw[line width=0.4mm,bend left=40,red](5) edge (15);
\draw[line width=0.4mm,bend right=40,red](5) edge (13);
\draw[line width=0.4mm,bend right=40,red](7) edge (14);
\draw[line width=0.4mm,bend left=40,red](7) edge (16);

\draw[line width=0.6mm,blue](-1.4,-0.5) edge (1.4,0.5);
\draw[line width=0.6mm,bend right=40,blue](-1.4,0.5) edge (0,1.4);
\draw[line width=0.6mm,bend left=40,blue](0,-1.4) edge (1.4,-0.5);

\draw node[right] at (1.4,0) {\color{dark-green}{$\mathbf 3$}};
\draw node[left] at (-1.4,0) {\color{dark-green}{$\mathbf 3$}};

\draw node[above] at (0,1.4) {\color{dark-green}{$\mathbf 1$}};
\draw node[below] at (0,-1.4) {\color{dark-green}{$\mathbf 1$}};

\draw node[left] at (-0.1,0.1) {\color{dark-green}{$\mathbf 2$}};

\end{tikzpicture} 
\end{center} 
\end{minipage}
\caption{\label{fig:bandandaccordion} The band walk $z$ (on the left) and the closed accordion $\gamma$ over $(S,M,\Delta^\gpoint)$ (on the right) associated to $z$.}
\end{figure}
\end{example}

\subsection{Euler form}\label{euler}
Let $\Lambda$ be an algebra of finite global dimension $d$ with $n$ iso-classes of simple modules. There is a bilinear form $e( \cdot,\cdot)$
on $\mathbb Z^n$ called the \newword{Euler form} with the
property that for any two $\Lambda$-modules $X,Y$, we have

$$e( \vdim(X),\vdim(Y)) = \sum_{i=0}^d \dim \Ext^i(X,Y).$$ 
Here, $\vdim(X)$ and $\vdim(Y)$ denote the dimension vectors  of $X$ and $Y$.
(Note that it is not obvious that the righthand side of this formula depends only on $\vdim(X)$ and $\vdim(Y)$ rather than on $X$ and $Y$.) See \cite[Section III.3]{ASS} for background on the Euler form.

Since~$\Lambda$ has finite global dimension, one can show that the morphism of free abelian groups
\[
 \vdim : \bigoplus_{i=1}^n \bZ [P_i] \longrightarrow \bZ^n : [P_i] \longmapsto \vdim P_i
\]
is an isomorphism.  We define the bilinear form~$\la \cdot, \cdot \ra$ on~$\bigoplus_{i=1}^n \bZ[P_i] = \bZ^n$ by 
\[
 \la \mathbf{x}, \mathbf{y} \ra = e(\vdim (\mathbf{x}), \vdim (\mathbf{y})).
\]

\begin{lemma}
 If~$X$ and~$Y$ are~$\Lambda$-modules with projective dimension at most~$1$, then~$\la g(X), g(Y) \ra = e(\vdim(X), \vdim(Y))$.
\end{lemma}
\begin{proof}
 If~$0\to P_1^X \to P_0^X \to X \to 0$ is a projective resolution of~$X$, then $$\vdim(X) = \vdim(P_0^X) - \vdim(P_1^X) = \vdim \left(g(X)\right).$$  The same is true for~$Y$.  So $$\la g(X), g(Y) \ra = e\left(\vdim \left(g(X)\right), \vdim \left(g(Y)\right)\right) = e\left(\vdim(X), \vdim(Y)\right).$$
\end{proof}

Write $e_i$ for the standard basis vector having a 1 in the $i$-th position and zeros elsewhere.
To calculate the Euler form explicitly, it is convenient to note that
$$ \la e_i,e_j\ra = \la g(P_i),g(P_j)\ra = \dim \Hom(P_i,P_j)$$
This last quantity equals the dimension at vertex $i$ of the representation
$P_j$. 

Let us now specialize in the case that $\Lambda$ is a gentle algebra of
finite global dimension. The condition that $\Lambda$ has finite global dimension precisely means that the associated marked surface has no~$\rpoint$-punctures (we continue to assume that it has no~$\gpoint$-punctures).

Moreover, the next result gives a nice description of $\Ext^i(X,Y)$ for $i\geq 1$ whenever $X$ and $Y$ are band modules.

\begin{lemma}\label{lem::ext-is-hom} Let $X$ and $Y$ be band modules. Then $$\dim \Ext^1(X,Y) = \dim \Hom(Y,X)$$ and $\Ext^i(X,Y)=0$ for $i>1$. \end{lemma}

\begin{proof} The Auslander--Reiten formula says that
  $\Ext^1(X,Y)$ is dual to $\overline{\Hom}(Y,\tau X)$, where
  $\overline{\Hom}(Y,\tau X)$ is the space of morphisms from
  $Y$ to $\tau X$ quotiented by those morphisms which factor through
  an injective module. However, since $X$ is a band module, it is of
  projective dimension 1, and thus $\overline{\Hom}(Y,\tau X)\simeq
  \Hom(Y,\tau X)$ \cite[Corollary IV.2.14]{ASS}.
  Finally, since $X$ is a band module, $\tau X \simeq X$.

  The statement for $\Ext^i(X,Y)$ with $i>1$ is immediate from the fact that $X$ is of projective dimension 1.
  \end{proof}

From Lemma \ref{lem::ext-is-hom}, we deduce the following proposition:

\begin{proposition}\label{prop:euler-hom-hom}  If $\Lambda$ is a gentle algebra of finite global dimension and $X$ and $Y$ are band modules for $\Lambda$, then:

  $$\la g(X),g(Y)\ra = \dim \Hom(X,Y) - \dim \Hom(Y,X).$$
\end{proposition}

Combining the previous proposition with Lemma \ref{lem::crossings-are-morphisms}, we obtain a link between the dimension of $\Hom(X,Y)$ and the geometric model previously introduced.

\begin{corollary} Let $\Lambda$ be a gentle algebra 
corresponding to a marked surface with neither~$\gpoint$-punctures nor~$\rpoint$-punctures and with $\gpoint$-dissection $(S,M,\Delta^{\gpoint})$.  Let $X$ and $Y$ be two non-isomorphic band bricks. The dimension of $\Hom(X,Y)$ is one half of the sum of
  $\la g(X),g(Y) \ra$ and the number of crossings of the curves corresponding to $X$ and $Y$. \end{corollary}

\begin{example} Let us take up again the algebra $\Lambda_n$ from
  the introduction.
  For this algebra, we have 
$$\langle e_i,e_j\rangle = \langle \dim P_i,\dim P_j \rangle = \left \{ \begin{array}{ll} 0 &\textrm{if $i>j$},\\
  1 & \textrm{if $j=i$},\\
  2& \textrm {if $j>i$}\end{array}\right.$$

  Thus, $$\la (a_1,\dots,a_n),(b_1,\dots,b_n)\ra= \sum_{i=1}^n a_ib_i + \sum_{1\leq i<j\leq n} 2a_ib_j.$$
  Note that $\langle \cdot, \cdot \rangle$ is skew symmetric over the hyperplane given by $x_1+ \cdots+x_n = 0$, and any $g$-vector of a band brick is in this hyperplane.  
  \end{example}
\subsection{Band bricks and band semibricks}

A \newword{semibrick module} is a direct sum of finitely many bricks such that
there are no non-zero morphisms between distinct summands.
We will specifically be interested in semibrick modules which are built as a
direct sum of band bricks. We refer to these as \newword{band semibricks}.

Given a surface with
$\gpoint$-dissection, 
$(S,M,\Delta^{\gpoint})$, and dual $\rpoint$-dissection $\Delta^{\rpoint}$, there is a corresponding gentle quiver $(Q(\Delta^{\gpoint}),R(\Delta^{\gpoint}))$.
As we have seen, band brick for $(Q(\Delta^{\gpoint}),R(\Delta^{\gpoint}))$
come in one-parameter families,
and these one-parameter families correspond to simple, primitive, closed curves on $S$. 

We will define a \newword{simple closed multicurve on $S$}
to be a finite collection of primitive closed curves on $S$ which have no self-intersections, and which do not intersect each other. A \newword{simple closed
  multislalom on $S$} is a simple closed multicurve on $S$ each of whose
components is a slalom with respect to $\Delta^{\rpoint}$.
(Equivalently, each of the components is an accordion with respect to
$\Delta^{\gpoint}$, but it is the former perspective which will be more relevant for us.) 

For $\C$ such a simple closed multislalom, 
we say that a module
corresponds to $\C$ if it is a direct sum of one brick for each closed curve,
where the brick is chosen from the one-parameter family corresponding to
the closed curve, and if multiple curves appear corresponding to the same
family of bricks, then the corresponding bricks are chosen so as to be
non-isomorphic.

\begin{proposition}\label{prop::briques-courbes}  Let $(S,M,\Delta^{\gpoint})$ be a surface with a 
  $\gpoint$-dissection. Every band semibrick for $(Q(\Delta^{\gpoint}),R(\Delta^{\gpoint}))$
  corresponds to a simple closed multislalom on $S$. Conversely, any
  module corresponding to such a simple closed multislalom is a band semibrick.
  
  In this correspondence, band bricks correspond to individual slaloms.
  \end{proposition}

\begin{proof}
  We already know that each band brick defines a primitive closed slalom without
  self-intersections. Given a  band semibrick, the corresponding curves necessarily
  either coincide up to homotopy or can be drawn so as not to intersect,
  by Lemma \ref{lem::crossings-are-morphisms}. This shows that any band semibrick
  corresponds to a simple closed multislalom as stated.

  For the converse, we must also observe that if $B$ and $B'$ are
  two non-isomorphic band bricks corresponding to homotopic curves, then
  $\Hom(B,B')=0$.
  \end{proof}

The previous proposition means that understanding band semibricks for the gentle quiver $(Q(\Delta^{\gpoint}),R(\Delta^{\gpoint}))$ amounts to understanding the simple closed multislaloms on $S$.

\section{Combinatorics of band semibricks of \texorpdfstring{$\Lambda_n$}{Lambda} \texorpdfstring{}{n}}\label{sec:band-semi}
In this section, we specialize the considerations of the previous section
to analyze the band semibricks for $\Lambda_n$. 

\subsection{The surface model for \texorpdfstring{$\Lambda_n$}{Lambda} \texorpdfstring{}{n}}\label{sssection:surfaceModelLambda_n} 

As already mentioned, \cite{PPP} provides an explicit procedure to pass from
a gentle algebra to a surface with dual dissections from which the algebra can be recovered.
We know from \cite[Example 4.16]{PPP} that the surface~$S_n$ for the algebra~$\Lambda_n$ has~$1$ or~$2$ boundary components if~$n$ is odd or even, that is has genus $\lfloor (n-1)/2 \rfloor$, and that is has exactly two~$\gpoint$-marked points and two $\rpoint$-marked points, both on the boundary.  We will give a description of its dissection that is suited to our purpose. 

The dissection~$\Delta^{\rpoint}$ of the surface~$S_n$ consists of two $(n+1)$-gons,
which for convenience we number $\mathsf{P}_1$ and $\mathsf{P}_2$.
We number the edges of $\mathsf{P}_i$ clockwise
as $E_0^{(i)}, E_1^{(i)}, \dots, E_n^{(i)}$ for $i=1,2$.

The two polygons are glued together by identifying $E_j^{(1)}$ with
$E_j^{(2)}$ for $1\leq j\leq n$.  Proceeding clockwise along the edge $E_j^{(1)}$ in $\mathsf{P}_1$ is
identified with proceeding counter-clockwise along the edge
$E_j^{(2)}$ in $\mathsf{P}_2$. Note that the edges $E_0^{(1)}$ and
$E_0^{(2)}$ are not identified. We refer to the identified edges as
$E_1,\dots,E_n$.

For convenience, we draw the polygons  $\mathsf{P}_1$ and $\mathsf{P}_2$
with a long top edge corresponding to the boundary, so that proceeding clockwise along
the edges $E^{(i)}_1,\dots,E_n^{(i)}$ amounts to travelling from right to left as shown in \Cref{fig:polygone_dissection}

\begin{figure}[h!]
    \begin{tikzpicture}
        \begin{scope}[thick, decoration={markings, mark=at position 0.6 with {\arrow{Stealth[length=3mm]}}},mydot/.style={
    circle,
    thick,
    fill=white,
    draw,
    outer sep=0.7pt,
    inner sep=1.5pt
  }, scale = 1]
            \draw (0,0) -- node[above] {$E^{(1)}_0$} (4,0);
            \draw[line width=0.4mm,red,postaction={decorate}] (4,0) -- node[right] {$E^{(1)}_1$} (4,-1.5);
            \draw[line width=0.4mm,red,postaction={decorate}] (4,-1.5) -- node[red,below right] {$E^{(1)}_2$} (3,-2.5);
        \draw[line width=0.4mm,red,postaction={decorate}]  (3,-2.5) -- node[below] {$E^{(1)}_3$} (1,-2.5);
        \draw[line width=0.4mm,red,postaction={decorate}] (1,-2.5) -- node[below left] {$E^{(1)}_4$} (0,-1.5) ;
        \draw[line width=0.4mm,red,postaction={decorate}](0,-1.5) -- node[left] {$E^{(1)}_5$} (0,0) ;
        \draw[dark-green,mydot] (2,0) circle (2.5pt);
        \draw[red,fill] (4,0) circle (3pt);
        \draw[red,fill] (0,0) circle (3pt);
        \draw[red,fill] (0,-1.5) circle (3pt);
        \draw[red,fill] (4,-1.5) circle (3pt);
        \draw[red,fill] (3,-2.5) circle (3pt);
        \draw[red,fill] (1,-2.5) circle (3pt);
        \draw[red] node at (2,-1.25) {$\mathsf{P}_1$};
        \end{scope}
        \begin{scope}[xscale=-1,shift={(-10,0)},thick, decoration={markings, mark=at position 0.6 with {\arrow{Stealth[length=3mm]}}},mydot/.style={
    circle,
    thick,
    fill=white,
    draw,
    outer sep=0.4pt,
    inner sep=1pt
  }]
        \draw (0,0) -- node[above] {$E^{(2)}_0$} (4,0);
        \draw[line width=0.4mm,red,postaction={decorate}] (4,0) -- node[left] {$E^{(2)}_5$} (4,-1.5);
        \draw[line width=0.4mm,red,postaction={decorate}] (4,-1.5) -- node[below left] {$E^{(2)}_4$} (3,-2.5);
        \draw[line width=0.4mm,red,postaction={decorate}]  (3,-2.5) -- node[below] {$E^{(2)}_3$} (1,-2.5);
        \draw[line width=0.4mm,red,postaction={decorate}] (1,-2.5) -- node[below right] {$E^{(2)}_2$} (0,-1.5) ;
        \draw[line width=0.4mm,red,postaction={decorate}](0,-1.5) -- node[right] {$E^{(2)}_1$} (0,0) ;
        \draw[line width=0.4mm,red,dark-green,mydot] (2,0) circle (2.5pt);
        \draw[red,fill] (4,0) circle (3pt);
        \draw[red,fill] (0,0) circle (3pt);
        \draw[red,fill] (0,-1.5) circle (3pt);
        \draw[red,fill] (4,-1.5) circle (3pt);
        \draw[red,fill] (3,-2.5) circle (3pt);
        \draw[red,fill] (1,-2.5) circle (3pt);
        \draw[red] node at (2,-1.25) {$\mathsf{P}_2$};
        \end{scope}
    \end{tikzpicture}
    \caption{The two hexagons of the dissection~$\Delta^{\rpoint}$ of the surface~$S_5$ for the algebra $\Lambda_5$}
    \label{fig:polygone_dissection}
\end{figure}
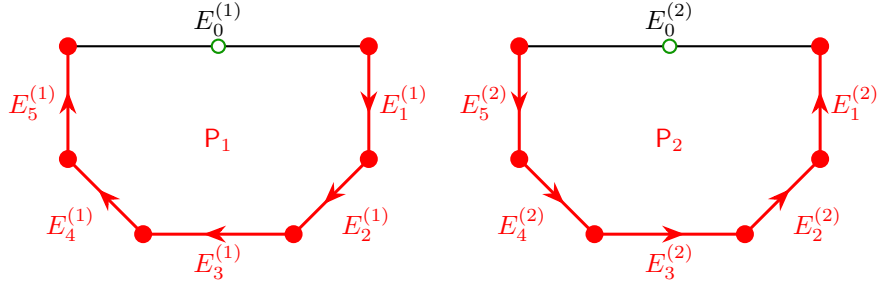

\begin{lemma}\label{slalom} Let $\mathcal{C}$ be a simple closed multislalom on $S_n$.
    Let $x$ be a point where $\mathcal{C}$ crosses $E_i$. On each of
    $\mathsf{P}_1$ and $\mathsf{P}_2$, we know that $\mathcal{C}$ connects $x$ to another crossing point.
  Either both these crossing points are to the left of $x$, or they are both to the right of $x$.
\end{lemma}

\begin{proof} This follows from the slalom condition.
  \end{proof}

\begin{lemma}\label{lem::good-signs} Let $\C$ be a simple closed multislalom on $S_n$. 
  Let $E_i$ be an edge which $\C$ crosses.
  Restrict attention to one of the two $(n+1)$-gons of $S_n$. Then
  either $\C$ connects every point of $\C\cap E_i$ to points on 
  edges to the left, or it connects every point of $\C\cap E_i$ to points
  on edges to the right. Which of these holds does not depend on which
  of the $(n+1)$-gons we were looking at.
\end{lemma}

\begin{proof} Suppose the polygon which we are considering is $\mathsf{P}_1$, and
  suppose that we have two points $x,y$ on $E^{(1)}_i$ such that $\C$ connects
  them (on $\mathsf{P}_1$) to, respectively, a point on an edge to the left
  of $E^{(1)}_i$ (i.e., an edge with index greater than $i$), and a point
  on an edge to the right of $E^{(1)}_i$ (i.e., with index less than $i$). 
  In order
  for $\C$ to have no crossings in $\mathsf{P}_1$, $x$ must be to the left of $y$
  on $\mathsf{P}_1$.

  Because of the way $\mathsf{P}_1$ and $\mathsf{P}_2$ are glued together,
  $x$ is to the right of $y$ on $\mathsf{P}_2$. By Lemma \ref{slalom},
  $x$ must be connected by $\C$ on $\mathsf{P}_2$ to an edge to the left of
  $i$, while $y$ must be connected by $\C$ on $\mathsf{P}_2$ to an edge to the right
  of $i$. This contradicts the hypothesis that $\C$ is simple i.e. has no
  self-intersections, as shown in Figure \ref{fig:lem}. 

  \begin{figure}
\begin{tikzpicture}
        \begin{scope}[thick, decoration={markings, mark=at position 0.6 with {\arrow{Stealth[length=3mm]}}},mydot/.style={
    circle,
    thick,
    fill=white,
    draw,
    outer sep=0.4pt,
    inner sep=1pt
  }]
        \draw (0,0) -- node[above] {$E^{(1)}_0$} (4,0);
        \draw[line width=0.4mm,red,postaction={decorate}] (4,0) -- node[right] {$E^{(1)}_1$} (4,-1.5);
        \draw[line width=0.4mm,red,postaction={decorate}] (4,-1.5) -- node[below right] {$E^{(1)}_2$} (3,-2.5);
        \draw[line width=0.4mm,red,postaction={decorate}]  (3,-2.5) -- node[below] {$E^{(1)}_3$} (1,-2.5);
        \draw[line width=0.4mm,red,postaction={decorate}] (1,-2.5) -- node[below left] {$E^{(1)}_4$} (0,-1.5) ;
        \draw[line width=0.4mm,red,postaction={decorate}](0,-1.5) -- node[left] {$E^{(1)}_5$} (0,0) ;
        \draw[dark-green,mydot] (2,0) circle (2.5pt);
        \draw[red,fill] (4,0) circle (3pt);
        \draw[red,fill] (0,0) circle (3pt);
        \draw[red,fill] (0,-1.5) circle (3pt);
        \draw[red,fill] (4,-1.5) circle (3pt);
        \draw[red,fill] (3,-2.5) circle (3pt);
        \draw[red,fill] (1,-2.5) circle (3pt);
        \draw[red] node at (2,-1.25) {$\mathsf{P}_1$};
        \draw node[below] at (1.3,-2.5) {$x$};
        \draw node[below] at (2.7,-2.5) {$y$};
        \draw[line width=0.6mm,blue] (1.3,-2.5) edge[bend right] (0,-1);
        \begin{scope}
            \clip (0,0) -- (4,0) -- (4,-1.5) -- (3,-2.5) -- (1,-2.5) -- (0,-1.5) -- (0,0);
            \draw[line width=0.6mm,blue] (2.7,-2.5) edge[bend left] (4,-1.7);
        \end{scope}
        \end{scope}
        \begin{scope}[xscale=-1,shift={(-10,0)},thick, decoration={markings, mark=at position 0.6 with {\arrow{Stealth[length=3mm]}}},mydot/.style={
    circle,
    thick,
    fill=white,
    draw,
    outer sep=0.5pt,
    inner sep=1pt
  }]
        \draw (0,0) -- node[above] {$E^{(2)}_0$} (4,0);
        \draw[line width=0.4mm,red,postaction={decorate}](4,0) -- node[left] {$E^{(2)}_5$} (4,-1.5);
        \draw[line width=0.4mm,red,postaction={decorate}] (4,-1.5) -- node[below left] {$E^{(2)}_4$} (3,-2.5);
        \draw[line width=0.4mm,red,postaction={decorate}]  (3,-2.5) -- node[below] {$E^{(2)}_3$} (1,-2.5);
        \draw[line width=0.4mm,red,postaction={decorate}] (1,-2.5) -- node[below right] {$E^{(2)}_2$} (0,-1.5) ;
        \draw[line width=0.4mm,red,postaction={decorate}](0,-1.5) -- node[right] {$E^{(2)}_1$} (0,0) ;
        \draw[dark-green,mydot] (2,0) circle (2.5pt);
        \draw[red,fill] (4,0) circle (3pt);
        \draw[red,fill] (0,0) circle (3pt);
        \draw[red,fill] (0,-1.5) circle (3pt);
        \draw[red,fill] (4,-1.5) circle (3pt);
        \draw[red,fill] (3,-2.5) circle (3pt);
        \draw[red,fill] (1,-2.5) circle (3pt);
        \draw[red] node at (2,-1.25) {$\mathsf{P}_2$};
        \draw node[below] at (1.3,-2.5) {$x$};
        \draw node[below] at (2.7,-2.5) {$y$};
        \draw[line width=0.6mm,blue] (1.3,-2.5) edge[bend left] (4,-0.7);
        \begin{scope}
            \clip (0,0) -- (4,0) -- (4,-1.5) -- (3,-2.5) -- (1,-2.5) -- (0,-1.5) -- (0,0);
            \draw[line width=0.6mm,blue] (2.7,-2.5) edge[bend right] (0,-1.7);
        \end{scope}
        \end{scope}
    \end{tikzpicture}
\caption{Example for the proof of \cref{lem::good-signs}}
    \label{fig:lem}
  \end{figure}
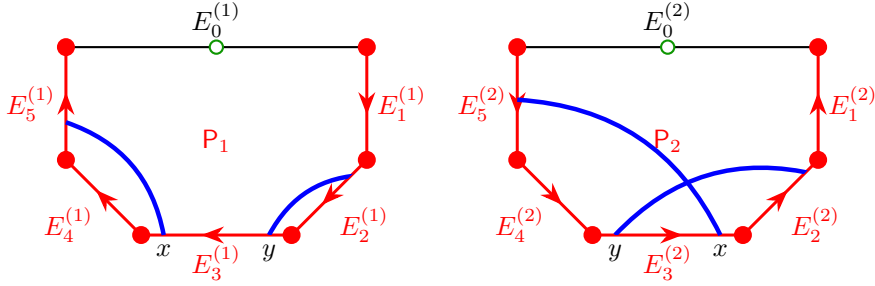

The fact that the result is independent of which $(n+1)$-gon we were looking at, follows from Lemma \ref{slalom}.
\end{proof}

\begin{remark} \label{remark-rep-th} The previous lemma can also be proved in more purely
  representation-theoretic terms. An intersection $x$ of $\C$ with $E_i$
corresponds to an 
appearance of the simple at $i$ in either the top or the socle of
the corresponding band semibrick,
depending on whether $\C$ connects the $x$ to edges to the right or
to the left (respectively), as noted in \cref{rema::closed-slalom}. The lemma now follows from the fact that a band semibrick cannot
have the same simple in its top and its socle.
\end{remark}

We associate to $\C$ an $n$-tuple of integers, $(a_1,\dots,a_n)$, which we
refer to as its \newword{$g$-vector}, adapting a formula of~\cite[Proposition 33]{HPS}. For $1\leq i \leq n$, define $|a_i|$ to be the
number of times that $\C$ crosses $E_i$. Assuming $|a_i|\ne 0$, define the
sign of $a_i$ to be positive if $\C$ connects the points of $\C\cap E_i$ to
edges further to the right, and to be negative if it connects the points of
$\C\cap E_i$ to edges further to the left. The sign is well-defined by
Lemma \ref{lem::good-signs}.

\begin{proposition} Let $\C$ be a simple closed multislalom on $S_n$, and let $X$ be a band semibrick 
  corresponding to $\C$. The $g$-vector of
  $X$ and the $g$-vector of $\C$ coincide. \end{proposition}

\begin{proof} Let~$z$ be the band walk defining~$X$.  The combinatorics of string and band modules gives us a formula for the~$g$-vector of~$X$ as follows: let~${\textsf {Top}}(z)$ be the multiset of vertices on top of~$z$ and~${\textsf {Bot}}(z)$ be the multiset of vertices at the bottom of~$z$.  Then~$g(X) = (g_1, \ldots, g_n)$, where~$g_i = \left| \left\{ a\in {\textsf {Top}}(z) \ | \ a = i \right\} \right| - \left| \left\{ b\in {\textsf {Bot}}(z) \ | \ b = i \right\} \right|$.  The result then follows directly from the definiton of the $g$-vector of~$\C$.
  \end{proof}

\begin{lemma} \label{reconstruct} Let $\C$ be a simple closed multislalom on $S_n$. Then $\C$ can be reconstructed (up to homotopy) from its $g$-vector.\end{lemma}

\begin{proof} Consider $\mathsf{P}_1$ first. Mark the correct number of crossing
  points on each edge. 
  Starting at the righthand end of
  $E_1^{(1)}$, proceed to the left. Whenever a crossing point on an edge $E_i$
  with $a_i< 0$ is reached, that crossing point must be connected to
  a previously passed crossing point. And, indeed, it must be connected
  to the last one which was seen and which is still available, since any
  other choice would eventually result in a crossing. This reconstructs
  $\C$ on $\mathsf{P}_1$ with no choices (except for the positions of the crossing
  points). Now do the same thing for $\mathsf{P}_2$, except that the positions of
  the crossing points are already fixed. This reconstructs $\C$.
  \end{proof}

\begin{corollary} Let $\C$ be a simple closed multislalom on $S_n$.
    Then the restrictions of $\C$ to $\mathsf{P}_1$ and $\mathsf{P}_2$
  coincide up to homotopy (as families of curves with endpoints on specified boundary segments).
\end{corollary}

\begin{proof} It is immediate from the reconstruction procedure given above that the reconstructed collection of curves are identical up to homotopy on $\mathsf{P}_1$ and $\mathsf{P}_2$.
  It must therefore also hold for $\C$.
\end{proof}

\begin{corollary}\label{bandSemibricksgvectors} If $X$ and $Y$ are band semibricks with
  $g$-vector $(a_1,\dots,a_n)$, then we can write $X = \bigoplus_{i=1}^r B_i$ and $Y = \bigoplus_{i=1}^r B'_i$ with~$B_i$ and~$B'_i$ band bricks that belong to the same one-parameter family.\end{corollary}

\begin{proof} This follows from the fact that there is only one collection
  of curves, up to homotopy, having this $g$-vector, so $X$ and $Y$ must
  both correspond to it.\end{proof}

\begin{remark}
 \cref{bandSemibricksgvectors} is a special case of a more general result for tame algebras: a band semibrick ~$X = \bigoplus_{i=1}^r B_i$ where the~$B_i$ belong to one-parameter families of bricks is determined by its~$g$-vector, up to replacing each~$B_i$ with a brick in the same one-parameter family.  This follows from \cite[Section 3]{GeissLabardiniSchroer}, see also \cite[Theorem 3.8]{PlamondonYurikusa} for a statement in terms of~$g$-vectors.  
\end{remark}

Since the $g$-vector $(a_1,\dots,a_n)$ allows us to
reconstruct $\C$ and therefore, up to choice of band bricks within one-parameter families, the
corresponding band semibrick, two questions naturally arise:
first of all, which $n$-tuples of integers $(a_1,\dots,a_n)$ arise as $g$-vectors of
band semibricks, and secondly, for $(a_1,\dots,a_n)$ which arises as the $g$-vector
of a band semibrick, how can we tell whether it is in fact the $g$-vector of a
single brick, rather than a direct sum of several band bricks? The first of these questions we will resolve in the next subsection. The latter question, we will address, but not completely solve, in Section \ref{sec:orthogonal}. 

\bigskip

\subsection{The Dyck path model}~\label{sec:dyck_path_model}
Given a simple closed multislalom $\C$ on $S_n$, there is a particular way
to redraw $S_n$ which makes the collection $\C$ especially simple to
describe.
Recall that a \newword{Dyck word} is a word $w$ over the binary alphabet $\{u,d\}$ such that $w$ contains as many $u$'s as $d$'s and each prefix of $w$ contains at least as many $u$'s as $d$'s.
An \newword{elementary step} is a segment of the form $[(x,y),(x+1,y+1)]$ or $[(x,y),(x+1,y-1)]$.
A \newword{lattice path} is a sequence of consecutive elementary steps in the plane $\mathbb{Z}\times \mathbb{Z}$, beginning at the origin.
A \newword{Dyck path} is a lattice path representing a Dyck word such that $u$ corresponds to an up-step given by  $[(x,y),(x+1,y+1)]$ and $d$ corresponds to  a down-step given by $[(x,y),(x+1,y-1)]$. For example, the Dyck path representing the word $uuuuddduuddd$ is drawn in Figure~\ref{fig:dyck_path}.
Dyck paths can be characterized as lattice paths which begin and end on the
$x$-axis and which never go below it.

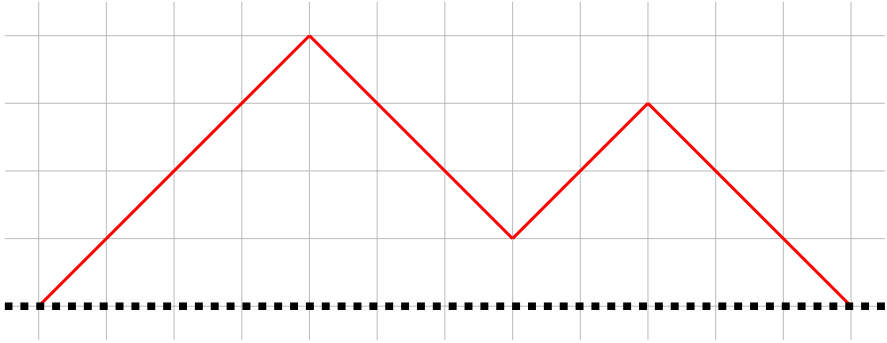
\begin{figure}[h!]
\begin{center}
\begin{tikzpicture}[scale=0.9]
\draw[step=1.0,gray!50,thin] (-0.5,-0.5) grid (12.5,4.5);
\tkzDefPoint(0,0){0}\tkzDefPoint(4,4){1} \tkzDefPoint(7,1){2} \tkzDefPoint(9,3){3} \tkzDefPoint(12,0){4} \tkzDefPoint(12.5,0){5} \tkzDefPoint(-0.5,0){6}
\draw[line width=0.4mm,red](1) edge (2);
\draw[line width=0.4mm,red](2) edge (3);
\draw[line width=0.4mm,red](3) edge (4);
\draw[line width=0.4mm,red](0) edge (1);
\draw[line width=1mm,dashed,black](5) edge (6);

\end{tikzpicture}
\end{center}
    \caption{Dyck path of Dyck word $uuuuddduuddd$.}
    \label{fig:dyck_path}
\end{figure} 

The $g$-vector $(a_1,\dots,a_n)$ of $\C$ can be rewritten as a word $\mathsf{d}(\C)$ on the alphabet $\{u,d\}$, where $a_i$ corresponds to $|a_i|$ copies of $u$ if $a_i<0$, or
of $d$ if $a_i>0$. Note the perhaps unexpected convention that when the entry in
the $g$-vector is \emph{negative}, this corresponds to a sequence of \emph{up}-steps. 

\begin{lemma} For $\C$ a simple closed multislalom on $S_n$, the word
  $\mathsf{d}(\C)$ is a Dyck word. \end{lemma}

\begin{proof} We begin by giving a different description of $\mathsf{d}(\C)$. Let us
    restrict our attention to $\mathsf{P}_1$. Label each point where $\C$ intersects $E^{(1)}_i$ by $u$ or $d$; we label it $u$ if $\C$ connects the point to an edge further
  to the left, and $d$ if $\C$ connects it to a point further to the right.
  Then $\mathsf{d}(\C)$ is the word we obtain by reading the labels of points where
  $\C$ intersects the boundary of $\mathsf{P}_1$ from right to left.

  Since, any time we read a $d$, we must already have read the $u$ to its right
  that $\C$ connects to it, any prefix of $\mathsf{d}(\C)$ has at least
  as many occurrences of $u$ as of $d$. Further, the total number of each in $\mathsf{d}(\C)$ is the same, because when
  we reach the end, all the crossing points will have been matched up.
  \end{proof}
  
The edge  $E_i^{(1)}$ is represented by a block of elementary steps of
length $a_i$. Hence, we label the corresponding steps of the Dyck path with $i$. 
Add an edge joining the two endpoints of the path, and running below the path, so that we now have a closed curve in the plane.
Identify the region inside this closed curve with $\mathsf{P}_1$. 
Each up-step has a matching down-step, since $\mathsf{d}(\C)$ is a Dyck word. We recover $\C$ by joining by a horizontal line each matching pair of up and down steps. This way of drawing a multislalom on $S_n$ will be referred to as the \newword{Dyck path model}. 

Figure~\ref{fig:dyck_surface} shows a curve on the surface $\mathsf{P}_1$ and its representation using the Dyck path model.
Note in particular that the same region can be identified with $\mathsf{P}_2$, and
$\C$ again consists of horizontal lines.

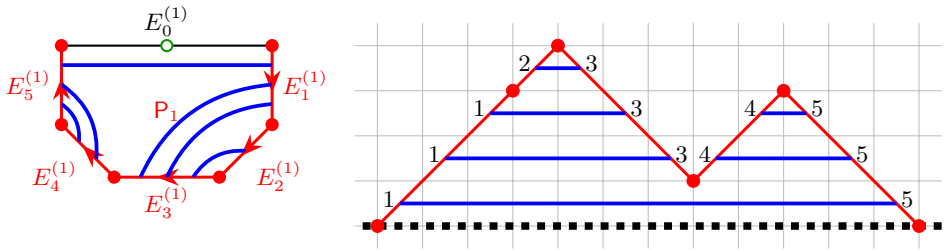
\begin{figure}
        \begin{tikzpicture}
            \begin{scope}[font=\small,scale=0.7,thick, decoration={markings, mark=at position 0.6 with {\arrow{Stealth[length=3mm]}}},mydot/.style={
    circle,
    thick,
    fill=white,
    draw,
    outer sep=0.5pt,
    inner sep=1pt
  }]
                \draw (0,0) -- node[above] {$E^{(1)}_0$} (4,0);
                \draw[line width=0.4mm,red,postaction={decorate}] (4,0) -- node[right] {$E^{(1)}_1$} (4,-1.5);
                \draw[line width=0.4mm,red,postaction={decorate}] (4,-1.5) -- node[below right] {$E^{(1)}_2$} (3,-2.5);
                \draw[line width=0.4mm,red,postaction={decorate}]  (3,-2.5) -- node[below] {$E^{(1)}_3$} (1,-2.5);
                \draw[line width=0.4mm,red,postaction={decorate}] (1,-2.5) -- node[below left] {$E^{(1)}_4$} (0,-1.5) ;
                \draw[line width=0.4mm,red,postaction={decorate}] (0,-1.5) -- node[left] {$E^{(1)}_5$} (0,0) ;
                \draw[dark-green,mydot] (2,0) circle (3pt);
                \draw[red,fill] (4,0) circle (3pt);
        \draw[red,fill] (0,0) circle (3pt);
        \draw[red,fill] (0,-1.5) circle (3pt);
        \draw[red,fill] (4,-1.5) circle (3pt);
        \draw[red,fill] (3,-2.5) circle (3pt);
        \draw[red,fill] (1,-2.5) circle (3pt);
                \draw[red] node at (2,-1.25) {$\mathsf{P}_1$};
                \begin{scope}[draw=blue]
                    \clip (0,0) -- (4,0) -- (4,-1.5) -- (3,-2.5) -- (1,-2.5) -- (0,-1.5) -- (0,0);
                    \draw[line width=0.5mm,blue] (0,-0.37) edge (5,-0.37);
                    \draw[line width=0.5mm,blue] (0,-0.74) edge[bend left] (0.66,-2.16);
                    \draw[line width=0.5mm,blue] (0,-1.11) edge[bend left] (0.33,-1.83);
                    \draw[line width=0.5mm,blue] (1.5,-2.5) edge[bend left] (4,-0.74);
                    \draw[line width=0.5mm,blue] (2,-2.5) edge[bend left] (4,-1.11);
                    \draw[line width=0.5mm,blue] (2.5,-2.5) edge[bend left] (3.5,-2);
                \end{scope}
            \end{scope}
            \begin{scope}[font=\small,scale=0.6,shift={(7,-4)}]
            \draw[step=1.0,gray!50,thin] (-0.5,-0.5) grid (12.5,4.5);
            \tkzDefPoint(0,0){0}\tkzDefPoint(4,4){1} \tkzDefPoint(7,1){2} \tkzDefPoint(9,3){3} \tkzDefPoint(12,0){4} \tkzDefPoint(12.5,0){5} \tkzDefPoint(-0.5,0){6}
            \draw[line width=0.4mm,red](1) edge (2);
            \draw[line width=0.4mm,red](2) edge (3);
            \draw[line width=0.4mm,red](3) edge (4);
            \draw[line width=0.4mm,red](0) edge (1);
            \draw[line width=1mm,dashed,black](5) edge (6);

            \draw[line width=0.5mm,blue] (0.5,0.5) -- (11.5,0.5);
            \draw[line width=0.5mm,blue] (1.5,1.5) -- (6.5,1.5);
            \draw[line width=0.5mm,blue] (2.5,2.5) -- (5.5,2.5);
            \draw[line width=0.5mm,blue] (3.5,3.5) -- (4.5,3.5);
            \draw[line width=0.5mm,blue] (7.5,1.5) -- (10.5,1.5);
            \draw[line width=0.5mm,blue] (8.5,2.5) -- (9.5,2.5);

            \draw node[left] at (0.6,0.6) {$1$};
            \draw node[left] at (1.6,1.6) {$1$};
            \draw node[left] at (2.6,2.6) {$1$};

            \draw node[left] at (3.6,3.6) {$2$};
            
            \draw node[right] at (4.4,3.6) {$3$};
	        \draw node[right] at (5.4,2.6) {$3$};
            \draw node[right] at (6.4,1.6) {$3$};

            \draw node[left] at (7.6,1.6) {$4$};
            \draw node[left] at (8.6,2.6) {$4$};
            \draw node[right] at (9.4,2.6) {$5$};
            \draw node[right] at (10.4,1.6) {$5$};
            \draw node[right] at (11.4,0.6) {$5$};
            
            \draw[red,fill] (0,0) circle (4pt);
            \draw[red,fill] (3,3) circle (4pt);
            \draw[red,fill] (4,4) circle (4pt);
            \draw[red,fill] (7,1) circle (4pt);
            \draw[red,fill] (9,3) circle (4pt);
            \draw[red,fill] (12,0) circle (4pt);
        \end{scope}
        \end{tikzpicture} 
    \caption{Left: the multislalom $\C$ given by the $g$-vector $(-3,-1,3,-2,3)$ drawn on either polygon of $S_5$. Right: the Dyck path model used to represent the same multislalom in a simpler way.}
    \label{fig:dyck_surface}
\end{figure}

\begin{proposition}\label{prop:conditionGvector} An $n$-tuple of integers $(a_1,\dots,a_n)$
  is the $g$-vector of some simple closed multislalom $\C$ on $S_n$ if and only if the partial sums $\sum_{i=1}^k a_i$ are non-positive for $1\leq k\leq n-1$, and $\sum_{i=1}^n a_i=0$. \end{proposition}

\begin{proof} These conditions on $(a_1,\dots,a_n)$
  amount to the requirement that
  $d(\C)$ is a Dyck word, which we have already shown is necessary.
  Conversely, if $(a_1,\dots,a_n)$ satisfies the conditions of the proposition,
  then it defines a Dyck word, and the Dyck path model tells us that
  it is possible to draw a simple closed multislalom on $S_n$ with this
  $g$-vector. \end{proof}

\begin{corollary}\label{corr:semibricks-gvec} An $n$-tuple of integers $\mathbf{g} = (a_1,\dots,a_n)$ is the
  $g$-vector of a semibrick band module for $\Lambda_n$ if and only if the partial sums
  $\sum_{i=1}^k a_i$ are non-positive for $1\leq k\leq n-1$, and $\sum_{i=1}^n a_i=0$. In this case, there is a collection of one-parameter families of band bricks $(\mathscr{F}_i)_{1\leqslant i \leqslant r}$, unique up to permutation, such that if
  $$X = \bigoplus_{i=1}^r B_i$$ where $B_i$ is chosen generically in $\mathscr{F}_i$, then $X$ is a band semibrick and satisfies $g(X) = \mathbf{g}$.
\end{corollary}

\begin{example}\label{ex::gvectorandband}
Let $n=3$ and $g = (-1,-1,2)$. Then we get the band module given in \cref{fig:bandandaccordion}.
\end{example} 

\section{Perfectly clustering words}
\label{sec:PCWs}

\subsection{Burrows--Wheeler transformation and perfectly clustering words}
\label{ssec:BWandPCW}

Consider an  alphabet $A$. Two words in the free monoid $A^*$ freely generated by $A$ are called \newword{conjugate} if for some words $x,y$ in $A^*$, they may be written $xy$ and $yx$. A \newword{conjugation class} is the set of conjugates of a word. A word is called \newword{primitive} if it is not the power of another word. A conjugation class is called \newword{primitive} if one of its elements is (equivalently, all of them are) primitive.

We assume now that the alphabet $A$ is totally ordered, and use the lexicographic order on $A^*$.

The \newword{Burrows--Wheeler transformation}, introduced in \cite{BW}, is the mapping from $A^*$ into itself, denoted $BW$, defined as follows. Call \newword{Burrows--Wheeler tableau} of a word $v$ the rectangular tableau whose rows from top to bottom are the conjugates of $v$, ordered lexicographically, from the smallest to the largest. See Figure \ref{BWtableau}. By definition, $BW(v)$ is the word obtained by reading the last column of the tableau, from top to bottom. 

\begin{definition} \label{def:MPA}
A word $v \in A^*$ is called \newword{perfectly clustering} if $v$ is primitive and $BW(v)$ is a weakly decreasing word.
\end{definition}

\begin{example}\label{ex:MPA} Let $v = acacacbbbc$. We can check easily that $v$ is primitive. Following Figure \ref{BWtableau}, we get $BW(v) = ccccbbbaaa$. Thus $v$ is a perfectly clustering word.

\begin{figure}[h!] 
$$
\begin{array}{ccccccccc c}
a&c&a&c&a&c&b&b&b& \pmb{c}\\
a&c&a&c&b&b&b&c&a&\pmb{c}\\
a&c&b&b&b&c&a&c&a&\pmb{c}\\
b&b&b&c&a&c&a&c&a&\pmb{c}\\
b&b&c&a&c&a&c&a&c&\pmb{b}\\
b&c&a&c&a&c&a&c&b&\pmb{b}\\
c&a&c&a&c&a&c&b&b&\pmb{b}\\
c&a&c&a&c&b&b&b&c&\pmb{a}\\
c&a&c&b&b&b&c&a&c&\pmb{a}\\
c&b&b&b&c&a&c&a&c&\pmb{a} 
\end{array}
$$
\caption{Burrows--Wheeler tableau of the perfectly clustering word $acacacbbbc$}\label{BWtableau}
\end{figure}
\end{example}

Clearly, two conjugate words have the same image under $BW$. Thus the mapping $BW$ is really a mapping from the set of conjugation classes of $A^*$ into $A^*$. If we restrict it to the set of primitive conjugation classes, it becomes injective. Indeed, suppose that $v$ is lexicographically minimum in its conjugation class (so that $v$ is a \newword{Lyndon word}). Let $BW(v)=w$. One recovers $v$ from $w$ in the following way. 
Consider the \newword{standard permutation} of $w$, denoted $st(w)$, obtained by numbering the letters of $w$, starting with the smallest one in $A$ and numbering its occurrences in $w$ from left to right, then the second smallest, and so on. For example $st(baaacaba)=61238475$. Let $\tau$ be the inverse permutation of $st(w)$; write $\tau$ as a product of disjoint cycles in the symmetric group; it turns out that there is a unique cycle; then replace the $i$ in the cycle by the $i$-th letter of $w$. Then one obtains one of the rows of the Burrows--Wheeler tableau of $v$. If the cycle ends with 1, it will be the first row, namely $v$ (see~\cite[Section 15.2]{R} for more details). 

\begin{example} \label{ex:inverseBW} Let $v$ as in \cref{ex:MPA}. The standard permutation of $BW(v) = c^4b^3a^3$ is $789X456123$ (we write $X$ for $10$), its inverse is $\tau=89X5671234$, in 
cycle form $\tau=(8293X45671)$, and replacing the $i$ in the cycle by the $i$-th letter of $w$, 
we obtain $acacacbbbc$, and so get back $v$.
\end{example}

One has a simple characterization of perfectly clustering words. In order to state it, call \newword{circular factor} of a word $w$ a factor of $ww$ 
whose length does not exceed the length of $w$. 

\begin{proposition} A word $v\in A^*$ is perfectly clustering if and only if for any circular factor of $v$ of the
form $aub$, with $a,b\in A$ and $u\in A^*$, there is no circular factor of $v$ 
of the form $a'ub'$ with $a',b'\in A$, $a<a'$ and $b<b'$. 
\end{proposition}

\begin{proof} Suppose that $v$ has the circular factors $aub$ and  $a'ub'$ with $u\in A^*$, $a,b,a',b'\in A$, $a<a'$ and $b<b'$. Then 
$v$ has two conjugates $v_1=ub\cdots a$ and $v_2=ub'\cdots a'$. Then $v_1<v_2$, so that $v_1$ is above $v_2$ in the 
Burrows--Wheeler tableau, and their last letters are respectively $a,a'$; since $a<a'$, $v$ is not perfectly clustering.

Conversely, suppose that $v$ is not perfectly clustering; then $v$ has two conjugates $v_1,v_2$, with $v_1<v_2$ and respective last letters $a,a'$ with $a<a'$. Since $v_1<v_2$ and since they have the same length, there exist factorizations 
$v_1=ub\cdots,v_2=ub'\cdots$, with $u\in A^*,b,b'\in A$ and $b<b'$. It follows that either $v_1=ub\cdots a,v_2=ub'\cdots a'$, and $w$ has the circular factors $aub$ and $a'ub'$; or $v_1=ub,v_2=ub'$, and denoting by $M(v)$ the multiset of letters appearing in the word $v$, we see that $M(ub)=M(ub')$ since these two words are conjugate, and therefore $b=b'$, a contradiction: this case cannot occur.
\end{proof}

The bijection $BW$ is a particular case of a bijection due to Gessel and one of the authors \cite{GR}. We present the slight variant given in \cite{GRR}. Each conjugation class of words is also called a \newword{necklace}.

Let $w=a_1\cdots a_n$ be a word and $\tau$ the inverse of its standard permutation. With each cycle $(j_1,\ldots,j_i)$ of $\tau$, associate 
the necklace $(a_{j_1},\ldots, a_{j_i})$. It turns out that this necklace is always primitive.

\begin{definition}\label{def:GRtransf}
 Let $\Phi$ be the map sending each word $w \in A^*$ to the multiset of primitive necklaces obtained by taking all cycles of the inverse of its standard permutation $\tau = st(w)^{-1}$, and replacing $i$ by the $i$-th letter of $w$.
\end{definition} 

\begin{example} \label{ex:GRex} Let $w=baacbcab$. We have $st(b a a c b c a b) = 41275836$, the inverse of the latter permutation is $4 1 2 7 5 8 3 6^{-1}
=23715846$. The cycles of $\tau$ are $(1,2,3,7,4)$, $(5)$, and $(6,8)$. The corresponding multiset of necklaces is therefore $\Phi(w)=
\{(baaac),(b),(cb)\}$.
\end{example}

\begin{thm}[\cite{GR}] \label{thm:GRbij}
The map $\Phi$ is bijective.
\end{thm}

In order to describe the inverse bijection, we follow \cite{GRR}. Given a multiset of primitive necklaces, consider the multiset of words obtained by taking the union with multiplicities of all corresponding conjugation classes (there are $i$ words for each necklace of length $i$); note that these words are all primitive. Put them in order, from top to bottom, in a right justified tableau, where the order is as follows: $u<v$ if and only if $uuuu....<vvvv...$ (lexicographic order for infinite words). If there are nontrivial multiplicities in the multiset, then there are repeated rows. Then the inverse image of the multiset is the last column, read from top to bottom.

\begin{example} \label{ex:GRex2} Take the previous multiset $\{(baaac),(b),(cb)\}$. The multiset of words is $\{baaac,aaacb,aacba,acbaa,cbaaa,b,cb,bc\}$. The order on the corresponding infinite words is $aaacb...<aacba...<acbaa...<baaac...<bb...<bc...<cbaaa...<cbcb...$. Thus the tableau is

$$
\begin{array}{rrrrrr}
a&a&a&c&\pmb{b} \\
a&a&c&b&\pmb{a} \\
a&c&b&a&\pmb{a} \\
b&a&a&a&\pmb{c} \\
&&&&\pmb{b} \\
&&&b&\pmb{c} \\
c&b&a&a&\pmb{a}\\
&&&c&\pmb{b}\\
\end{array}
$$
Its last column is $baacbcab$ as desired.
\end{example}

Therefore the mapping $BW$, viewed as a mapping on primitive necklaces, is the restriction of $\Phi^{-1}$ to the set of primitive necklaces.  See also \cite{PR} for a study of the similarity of the two constructions.

\subsection{Perfectly clustering words and band bricks}

The first step to show the link between perfectly clustering words and band bricks is to define a multiset of circular words from a simple closed multislalom $\mathcal{C}$ drawn on $S_n$, or equivalently, from its $g$-vector $(a_1,\dots, a_n)$. To obtain the multiset of circular words associated with $\mathcal{C}$: follow $\mathcal{C}$ and write down the label of the edges of the surface $S_n$ crossed by $\mathcal{C}$.

More precisely, from the $g$-vector $(a_1,\dots, a_n)$ draw two copies of the Dyck path ($\mathsf P_1$ and $\mathsf P_2$) and the simple closed multislalom $\mathcal{C}$ with labels as in Section~\ref{sec:dyck_path_model}.
We denote by $M_{(a_1,\dots,a_n)}$ the multiset of circular words defined from the simple closed multislalom $\mathcal{C}$ with $g$-vector $(a_1,\dots,a_n)$.
We construct all the circular words in $M_{(a_1,\dots,a_n)}$ with the following algorithm. We choose arbitrarily an unvisited step on a copy of the Dyck path. From there, we follow $\mathcal{C}$ (rightwards on the first copy, and leftwards on the second copy) until we reach another step $x$ on the Dyck path, we write down the label of $x$ and go to the step identified with $x$ on the other copy of the Dyck path. We continue to follow $\mathcal{C}$ until we reach the step where we started. At this point, we add the circular word to $M_{(a_1,\dots, a_n)}$. If we have visited all the steps on both copies of the Dyck path, the algorithm stops. Otherwise, we start again from an unvisited step.
For example, the multiset of circular words of the $g$-vector $(-3,-1,3,-2,3)$ shown in \Cref{fig:dyck_surface} is $M_{(-3,-1,3,-2,3)} = \{(54545131),(3231)\}$. 

\begin{remark}
    The number of circular words in $M_{(a_1,\dots, a_n)}$ corresponds to the number of curves in the simple closed multislalom $\mathcal{C}$.
\end{remark}

If all entries of a $g$-vector except the first one are non-negative, then $M_{(a_1,\dots, a_n)}$ can be obtained from $\Phi(n^{a_n}\dots 2^{a_2})$ by inserting 1's so that every second letter in each of the cycles is a 1.
For example, $M_{(-8,2,2,4)} = \{(41314121)(41314121)\}$ is the multiset of circular words of the simple closed multislalom with $g$-vector  $(-8,2,2,4)$, while $\Phi(44443322) = \{(4342)(4342)\}$. Before proving this key result linking perfectly cluster words and band bricks, 
we need to highlight a property of circular words of a weakly decreasing word under the bijection $\Phi$. 

\medskip

\begin{lemma}~\label{l:multisetClustering}
    Let $w$ be a weakly decreasing word. Then, each circular word in $\Phi(w)$ is perfectly clustering.
\end{lemma}

\begin{proof}
    Let $u$ be a necklace of $\Phi(w)$. 
    Take $u_1$ and $u_2$ two conjugates of $u$ such that $u_1 \neq u_2$. We write $u_1^\omega$ for the infinite word $u_1u_1\dots$.  We have that $u_1^{\omega} < u_2^{\omega}$ if and only if $u_1 < u_2$, since $u_1$ and $u_2$ have the same length. This implies that the conjugates of $u$ appear in the same order in the tableau of $u$ and the tableau of $\Phi(w)$. Therefore, the last column of the tableau of $u$ is a weakly decreasing word since the last column of the tableau of $\Phi(w)$ i.e $w$ is a weakly decreasing. Thus, any circular word in $\Phi(w)$ is perfectly clustering.
\end{proof}

\medskip

\begin{thm}~\label{prop:curves2words}
  Let $(a_1,\dots,a_n)$ be the $g$-vector of a simple closed multislalom with $a_1$ a negative integer and
  $a_i$ a non-negative integer for $2\leq i \leq n$.
  Let $M_{(a_1,\dots,a_n)}$ be the multiset of circular words defined by $(a_1,\dots, a_n)$. Then, \begin{align}~\label{eq:multi_set}f(M_{(a_1,\dots,a_n)}) = \Phi(n^{a_n}\dots 2^{a_2}),\end{align}
where $f$ is the erasing morphism $f(1) = \varepsilon$ and $f(i) = i$ for $i \in \{2,\dots,n\}$.
\end{thm}

\begin{proof} 
  To prove \Cref{eq:multi_set}, we view $f(M_{(a_1,\dots, a_n)})$ and $\Phi(n^{a_n}\dots 2^{a_2})$ as defining permutations of the letters of $n^{a_n}\dots 2^{a_2}$  and show they are inverses of each other, from which we show the desired result follows. 

    The Dyck path given by the $g$-vector $(a_1,\dots,a_n)$ with $a_1$ negative and $a_i$ non-negative is given by $|a_1|$ up-steps labelled by $1$ followed by $|a_1|$ down-steps labelled by $2^{a_2}3^{a_3} \dots n^{a_n}$. The simple closed multislalom $\mathcal{C}$ is obtain by the construction explained in \Cref{sec:dyck_path_model}.
    
    The shape of the Dyck path means that the factors of length $3$ in a circular word in $M_{(a_1,\dots,a_n)}$ are either $x1y$ or $1x1$ with $x,y \in \{2,3,\dots, n\}$. We will study the factors of the form $x1y$.
    A factor $x1y$ in $M_{(a_1,\dots,a_n)}$ starts at an up-step on $\mathsf{P}_1$.  Let $s'$ be its matching down-step. The step $s'$ is labelled by a letter $x$, so we write down $x$. 
    Suppose that $x$ is the $q$-th occurence of the letter $x$ in $n^{a_n} \dots 2^{a_2}$. Therefore, it is at the height $(a_n + a_{n-1} + \dots + a_{x+1} + q)$ in the Dyck path of $\mathsf{P}_1$.
    The step $s'$ in $\mathsf{P}_1$ is identified with the down-step at position $(a_n + a_{n-1} + \dots + a_{x+1} + a_x + 1 - q)$ in $\mathsf{P}_2$ since the gluing between $\mathsf{P}_1$ and $\mathsf{P_2}$ is along edges which are oriented in opposite directions. Walking to the next matching up-step we write a $1$ and are at the same height on the Dyck path, i.e. the $(a_n + a_{n-1} + \dots + a_{x+1}  + a_x + 1 - q)$-th $1$ counting from the bottom.
    The up-step labelled by the $(a_n + a_{n-1} + \dots + a_{x+1} + a_x + 1 -  q)$-th $1$ in $\mathsf{P}_2$ is identified with the $|a_1| + 1 - (a_n + a_{n-1} + \dots + a_{x+1} + a_x + 1 -  q) = a_2 + \dots + a_{x-1} +  q$-th up-step in $\mathsf{P}_1$. We go to the matching down-step to write the letter at position $a_2 + \dots + a_{x-1}  + q$ in $n^{a_n} \dots 2^{a_2}$.
    Hence, our factor of length $3$ is the letter at position $(a_n +a_{n-1} + \dots + a_{x+1} + q)$ in $ n^{a_n} \dots 2^{a_2}$ followed by a $1$ followed by the letter at position $a_2 + \dots + a_{x-1} +  q$ in $n^{a_n} \dots 2^{a_2} $. 
   
Let us now look at $f(M_{(a_1,\dots,a_n)})$. 
The position of the $q$-th occurence of $x$ in $w$ is $(a_n + a_{n+1} + \dots + a_{x-1} + q)$ and its standardization is $a_2 + \dots + a_{x-1} + q$. Recall that $\Phi(w)$ is defined to be the inverse of the permutation given by the standardization. Thus, we have shown that $\Phi(w)$ and $f(M_{(a_1,\dots,a_n)})$ define inverse permutations of the letters of $ n^{a_n} \dots 2^{a_2}$.

Now let $v$ be a circular word in $\Phi(w)$.
By Proposition~\ref{l:multisetClustering}, $v$ is perfectly clustering. Hence, $v$ is conjugate to $\widetilde{v}$ (the reversal of $v$) 
by~\cite[Theorem 4.3]{PS}. This means that each circular word in $\Phi(w)$ is conjugate to a different circular word in $f(M_{(a_1,\dots,a_n)})$ i.e. both multisets are equal since conjugates of circular words are equal. Thus, the Equation~(\ref{eq:multi_set}) holds. 
\end{proof}

From a primitive word $w$, one can construct a band walk on the quiver for $\Lambda_n$. For each integer $i \in \{2,\dots,n\}$, we associate the cycle $z_i=\alpha_1\alpha_2\cdots\alpha_{i-1}\beta_{i-1}^{-1}\cdots\beta_2^{-1}\beta_1^{-1}$ on the quiver $Q_n$. Let $w = w_1 w_2 \dots w_r$ be a primitive word over the alphabet $\{2,\dots,n\}$, we defined $\psi(w) = z_{w_1} z_{w_2}\dots z_{w_r}$. For example,  if we take $Q_3$ as in Example~\ref{ex::bandrep} and $w = 23223$, then $$\psi(23223) = \alpha_1\beta_1^{-1}\alpha_1\alpha_2\beta_2^{-1}\beta_1^{-1}\alpha_1\beta_1^{-1}\alpha_1\beta_1^{-1}\alpha_1\alpha_2\beta_2^{-1}\beta_1^{-1}.$$
The band $\Lambda_n$-module $B_{\psi(w),1,\lambda}$ is a band module for any $\lambda \in k^\times$ since $\psi(w)$ is a band walk. 
\begin{thm}
\label{thm: brick bands and pcw}

A primitive word $w$ on the alphabet $\{2,\dots, n\}$ is perfectly clustering if and only if the band $\Lambda_n$-module $B_{\psi(w),1,\lambda}$ 
is a band brick for some (equivalently any) $\lambda \in k^\times$. 
\end{thm}

\begin{proof}
    Let $w$ be a perfectly clustering word of the alphabet $\{2,\dots,n\}$.  The $n$-tuple of integers $(-|w|,|w|_2,|w|_3,\dots, |w|_n)$ is the $g$-vector of some simple closed multislalom $\mathcal{C}$ by Proposition~\ref{prop:conditionGvector}. Moreover, we have that $\Phi(n^{|w|_n} \dots 2^{|w|_2}) = \{(w)\} = f(M_{(-|w|,|w|_2,|w|_3,\dots, |w|_n)})$ since $w$ is perfectly clustering and by Theorem~\ref{prop:curves2words}. Hence, the simple closed multislalom $\mathcal{C}$ is a simple closed curve. By \cref{prop::briques-courbes}, the band $\Lambda_n$-module $B_{\psi(w),1,\lambda}$ is brick. 

    Let $B_{\psi(w),1,\lambda}$ be a band brick with $g$-vector $(a_1,\dots a_n)$.
    Using the Dyck path representation, one can obtain some simple closed multislalom $\mathcal{C}$ associated to $(a_1,\dots, a_n)$. In fact, $\mathcal{C}$ is a simple closed slalom, since $B_{\psi(w),1,\lambda}$ is a brick by \cref{prop::briques-courbes}. Hence, $$|M_{(a_1,\dots, a_n)}| = |\Phi(n^{a_n}\dots 2^{a_2})| = 1$$ by Proposition~\ref{prop:curves2words} and $\Phi(n^{a_n}\dots 2^{a_2}) = \{(w)\}$ meaning that $w$ is a perfectly clustering word.
\end{proof}

\section{Band semibricks}
\label{sec:orthogonal}

\subsection{The band brick compatibility fan} 
We have already established in Corollary \ref{corr:semibricks-gvec} that there is a bijection between band semibricks of $\Lambda_n$ and $g$-vectors $(a_1,\dots,a_n)$ satisfying that the initial partial sums are non-positive and the total sum is zero.

The question arises: which $g$-vectors  correspond to a single one-parameter
family of bricks, rather than a direct sum of multiple band bricks? We refer to
such $g$-vectors as band brick $g$-vectors.

Further, given a $g$-vector corresponding to a band semibrick, how can we decompose it into the band brick $g$-vectors of the band bricks that make it up? We refer to this as the \newword{band semibrick decomposition problem}. 

It turns out to be natural to address the two questions together.

A \newword{fan} is a collection~$\mathcal{F}$ of cones such that any face of a cone in~$\mathcal{F}$ is a cone in~$\mathcal{F}$ and such that the intersection of any two cones in~$\mathcal{F}$ is a face of each.

We say that two one-parameter families of band bricks
(or two band brick $g$-vectors) are compatible if they can
appear in a band semibrick together, i.e., if the two one-parameter 
families of bricks do not have any non-zero morphisms between them.

\begin{definition}
 The \newword{band brick compatibility fan}~$\Sigma_n$ of~$\Lambda_n$ is the collection of cones in the space of $g$-vectors which
are spanned by a collection of mutually compatible band brick $g$-vectors.
\end{definition}

It follows from results of C. Geiss, D. Labardini-Fragoso, and J. Schröer on tame algebras that
$\Sigma_n$ is indeed a fan \cite[Theorem 3.2]{GeissLabardiniSchroer} (see also~\cite[Theorem 3.8]{PlamondonYurikusa}). The following result is a reformulation in our special cases knowing that gentle algebras are tame.

\begin{proposition}\label{prop:fan} The set $\Sigma_n$ is a simplicial fan, whose support consists of the band semibrick $g$-vectors. 
  The band brick $g$-vectors which appear in the band semibrick decomposition of a band semibrick $g$-vector $\mathbf a$ are the rays of the minimal cone of $\Sigma_n$ containing $\mathbf a$.\end{proposition}

The ensuing corollary is also needed for some arguments in the next subsection.

\begin{corollary} \label{basis}
  Let  $\mathbf{a}^1,\dots,\mathbf{a}^r$ be pairwise distinct $g$-vectors of mutually compatible
  bricks. Then $\mathbf{a}^1, \mathbf{a}^2, \dots, \mathbf{a}^r$ are linearly independent.
\end{corollary}

\subsection{The Euler form and the maximum number of mutually compatible bricks}

In this subsection, we apply the Euler form to establish the maximal dimension of a cone in $\Sigma_n$, or, which is the same thing, the maximum number of mutually compatible bricks from distinct one-parameter families (or equivalently, with distinct $g$-vectors).

\begin{lemma}\label{necc}
  In order for $\mathbf a$ and $\mathbf b$ to be the $g$-vectors of two compatible band bricks, we must have $\langle \mathbf a, \mathbf b\rangle = 0$. 
\end{lemma}

\begin{proof} This is immediate from Proposition \ref{prop:euler-hom-hom}.\end{proof}

Note that the converse is false in general.

\begin{example}
  Let $n=4$, $\mathbf{a}=(-1,-2,-2,5)$, $\mathbf{b}=(-3,0,-4,7)$ satisfy
  $\langle \mathbf{a},\mathbf{b}\rangle=0$. However, $\mathbf{a}$ and
  $\mathbf{b}$ are not compatible bricks. This can be seen either by constructing the corresponding modules and checking that there are morphisms between them, by drawing the corresponding curves and seeing that they must cross, or by discovering that $\frac12(\mathbf {a}+\mathbf{b})$ is itself a $g$-vector of a band brick, which would be impossible if the band bricks with dimension vectors
  $\mathbf a, \mathbf b$ were compatible.
\end{example} 

The following lemma is a crucial step towards one of our main results.

\begin{lemma} Let $V$ be an $n$-dimensional  vector space equipped with a
  non-degenerate bilinear form $\langle\cdot,\cdot\rangle$.
  Let $H$ be a codimension-one subspace of
  $V$ such that for $x,y\in H$, we have $\langle x,y\rangle=-
  \langle y,x\rangle$. The maximum dimension of an isotropic subspace in $H$ is
  $\lceil (n-1)/2 \rceil$.
\end{lemma}

\begin{proof} Let $\Rad$ be the radical of the form restricted to $H$.
  For any vector $x\in \Rad$, the linear form on $V$ given by $\la x,-\ra$
  lies in the one-dimensional subspace of forms having $H$ as kernel.
  Thus, $\Rad$ is at most one-dimensional. Given a skew-symmetric form on $H$,
  we know that $H$ necessarily admits a decomposition as $\Rad \oplus J$,
  where the skew-symmetric form restricted to $J$ is non-degenerate, and
  $J$ is necessarily even-dimensional 
  \cite[Theorem XV.8.1]{Lang}.
  An isotropic subspace of $J$ is contained in its orthogonal, so its dimension is at most half the dimension of $J$.
  The maximum dimension of an isotropic subspace of $H$ is then
  half the dimension of $J$ plus the dimension of $\Rad$, which works out to
  $\lceil (n-1)/2 \rceil$. \end{proof}

\begin{thm} \label{maxsize} The maximum size of a set of mutually compatible band bricks from distinct one-parameter families is
  $\lceil (n-1)/2\rceil$. \end{thm}

\begin{proof}
We know that~$g$-vectors of  band semibricks lie in the codimension one subspace of $\bZ^n$ where the sum of all coordinates is zero, by \cref{prop:conditionGvector}.
We have seen that, restricted to this subspace, the Euler form is
skew-symmetric (\cref{prop:euler-hom-hom}), while on $\bZ^n$, it is non-degenerate.
The above lemma shows that the maximum dimension of an
isotropic subspace is $\lceil (n-1)/2 \rceil$. By Corollary \ref{basis} and Lemma \ref{necc},
the $g$-vectors of mutually compatible band bricks from distinct
one-parameter families form the basis of an
isotropic subset. There are thus at most $\lceil (n-1)/2 \rceil$ distinct
band bricks in the set.\end{proof}

\begin{example}
  Consider for $1\leq i \leq \lfloor (n-1)/2\rfloor$, $$\mathbf{b}^i = (-2,0, \ldots,0, \underset{i+1}{1}, 0 \ldots, 0, \ldots,0, \underset{n-i+1}{1}, 0, \ldots, 0)$$ and if $n$ is even, $\mathbf{b}^{n/2} = (-1, 0,\cdots,0, \underset{(n/2)+1}{1},0, \cdots, 0)$. Then $\mathbf{b}^1, \ldots, \mathbf{b}^{\lceil (n-1)/2\rceil}$  is a maximal family of mutually compatible band brick $g$-vectors.
\end{example}

One of our main results now follows from \cref{prop:curves2words}. 

\bigskip 

\begin{corollary}~\label{length}
  Let $(a_2,\dots a_n)$ be a vector of non-negative integers. The number of distinct primitive
  circular words
  in $\Phi(n^{a_n} \dots 2^{a_2})$ is at most $\lceil (n-1)/2 \rceil$.
\end{corollary}

\begin{proof}
  This follows from \Cref{maxsize} since the distinct primitive circular words in $\Phi(n^{a_n} \dots 2^{a_2})$ correspond by Theorem \ref{prop:curves2words} to the one-parameter families of brick summands of a band semibrick with $g$-vector
  $(-a_2-a_3 -\dots-a_n,a_2,\dots a_n)$.
\end{proof}

\Cref{length} implies Conjecture 3.33 of \cite{L2020} (cited as \cref{lengthconj} in the introduction), since the number of distinct lengths of primitive circular words is obviously bounded above by the number of distinct primitive circular words. Note that Conjecture 3.33 of \cite{L2020} is stated in terms of the structure of the discrete interval exchange transformation which is equivalent to the structure of multisets of circular words given by $\Phi$ as shown by \cite[Theorem 3.8]{L2020}.

\subsection{Characterizing band brick \texorpdfstring{$g$}{g}-vectors for \texorpdfstring{$n=4$}{n=4}}
The following result gives us a characterization of the $g$-vectors that are associated to a one-parameter family of band bricks of $\Lambda_n$ for $n=4$.

\begin{thm}\label{thm:bricksfour}
Let $g = (a,b,c,d)$ be a $g$-vector. Then $g$ corresponds to an one-parameter family of band bricks over $\Lambda_n$ if and only if
\begin{center}
$\mathsf{gcd}(a+b, b+c) = 1$ and $a+b\ne 0$, or $(a,b,c,d)=(-1,1,0,0)$ or $(0,0,-1,1)$.
\end{center} 
\end{thm}

This result can be reformulated in the language of meanders, and part of
the theorem has already been proven in those terms by Coll, Giaquinto, and
Magnant \cite{CGM}. 
The case where $a,b<0$; $c,d>0$ corresponds to what they call
the \newword{opposite maximal parabolic} case, while the cases 
where
$a<0$ and $b,c,d>0$ or $a,b,c<0$; $d>0$ correspond to what they call the
\newword{submaximal parabolic} case.
\cite[Corollary 4.9, Theorem 4.10]{CGM} prove those two parts of our theorem
(respectively). The case $a,c<0$; $b,d>0$ is new. We provide a uniform
proof which covers all three cases.

\begin{proof}
The $g$-vectors of semibricks consist of the lattice points in
a full-dimensional cone in the
three-dimensional space $H=\{(x_1,x_2,x_3,x_4)\mid \sum_{i=1}^4 x_i=0\}$.
The Euler form, restricted to $H$,
has a radical, which is generated by
$(-1,1,-1,1)$.

Applying Theorem \ref{maxsize} to the case $n=4$, we see that the maximal size
of a set of mutually compatible bricks is two.
Further, the $g$-vectors
of a pair of compatible bricks must lie in an isotropic subspace of $H$.
According to the above analysis, the isotropic subspaces of $H$
are just the two-dimensional subspaces which
contain the radical.

We therefore focus our attention on a two-dimensional subspace $V$ containing
the radical.

The
direction $(-1,1,-1,1)$ itself lies in the relative interior of a facet of
the cone of $g$-vectors. Thus, the intersection of $V$ with the
$g$-vector cone is necessarily two-dimensional.

One special case is when $V$ is the plane spanned by the facet of the
cone containing the radical. This is the plane where $a+b=0$. The extreme rays of this facet are
$(-1,1,0,0)$ and $(0,0,-1,1)$. The corresponding bricks are orthogonal, and
so all the lattice points on this facet (including those in the radical)
decompose as a linear combination of these two. This confirms the part of
the statement of the theorem that applies when $a+b=0$.

In the remaining cases, we see that the intersection of $V$ with the $g$-vector cone has the ray generated by $(-1,1,-1,1)$ as one of its
two extreme rays. The other ray is necessarily generated by some integer
vector $v$ lying on one of the two faces of the $g$-vector cone not
containing the radical, i.e., having its first coordinate or last coordinate
zero. We choose $v$ so that it generates the subgroup of lattice points on the
ray.

Let $C$ be the cone generated by $v$ and $(-1,1,-1,1)$, with the ray
$(-1,1,-1,1)$ removed.

\begin{lemma}\label{lattice-basis} The vectors $v$ and $(-1,1,-1,1)$ are a basis for
  the ambient lattice
  $\mathbb Z^4 \cap V$. \end{lemma}

\begin{proof} Without loss of generality, suppose that the multiples of
  $v$ are the vectors with first coordinate zero. Given any integer vector in
  $V$, by subtracting a suitable integer
  multiple of $(1,-1,1,-1)$, we may arrange that
  its first coordinate be zero, which shows that what is left is an
  integer multiple of $v$.
  \end{proof}

\begin{proposition} The vectors $\mathcal D = \{v+p(-1,1,-1,1)\mid p\in \mathbb Z_{\geq 0}$
  are brick $g$-vectors.\end{proposition}

\begin{proof} For any $g$-vector in $C$,
  the only isotropic plane
  containing it is $V$, so the bricks in its orthogonal brick decomposition
  must also correspond to lattice points in $V$. Further, the lattice points
  in the ray generated by $(-1,1,-1,1)$ are not available, since they do not
  correspond to $g$-vectors of bricks.
  The points of $\mathcal D$ admit no non-trivial
  decomposition as a sum of lattice points from $C$, so they must all be
  bricks.
\end{proof}

We now establish that the bricks corresponding to two consecutive points in
$\mathcal D$ are compatible, by showing that it is possible to draw the
corresponding curves on $S_4$ in such a way that they do not cross each
other.

\begin{lemma} Let $(a,b,c,d)$ and $(a+1,b-1,c+1,d-1)$ be two consecutive
elements of $\mathcal D$. It is possible to draw curves with these two
$g$-vectors on $S_4$ in such a way that they do not intersect. \end{lemma}

\begin{proof}
Suppose first that $b$ and $c$ are negative.
The Dyck path for $(a,b,c,-(a+b+c))$ consists of
$a+b+c$ upsteps followed by $a+b+c$ downsteps. The curve $C$ with
$g$-vector $(a,b,c,d)$
consists of horizontal lines at height $i+\frac12$ for $0\leq i\leq -a-b-c-1$.
We then check that it is possible to draw a curve $C'$ with $g$-vector $(a+1,b-1,c+1,d-1)$, so that each
of its lines lies between successive lines of $C$. Position
the righthand endpoints of $C'$
at the midpoints of the segments resulting from dividing the lefthand
side of the surface into $-a-b-c-1$ equal segments. For the lefthand side,
we first divide it into a bottom part of length $-a$, followed by a part of length
$-b$, followed by a part of length $-c$. We divide the part of length $-c$ into
$-c-1$ equal-sized segments, and put right endpoints at the midpoint of each
of these segments; we then treat the second and third parts similarly.
We observe that we have placed one lefthand endpoint and one righthand
endpoint between each successive pair of horizontal lines of $C$ (but neither below the bottommost endpoints on either side, nor above the topmost), meaning that it is possible to draw $C'$ in such a way as not to intersect $C$.
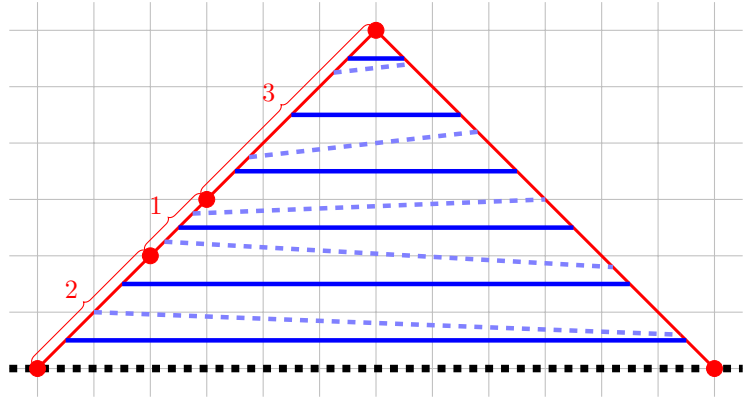
\begin{figure}
\begin{tikzpicture}[scale=.75]
\draw[step=1.0,gray!50,thin] (-0.5,-0.5) grid (12.5,6.5);\draw[line width=1mm, black,dashed] (-0.5,0)--(12.5,0);
  \draw[line width=0.4mm, red] (0,0)--(6,6)--(12,0);
  \draw[line width=0.6mm, blue] (.5,.5)--(11.5,.5);
  \draw[line width=0.6mm, blue] (1.5,1.5)--(10.5,1.5);
  \draw[line width=0.6mm, blue] (2.5,2.5)--(9.5,2.5);
  \draw[line width=0.6mm, blue] (3.5,3.5)--(8.5,3.5);
  \draw[line width=0.6mm, blue] (4.5,4.5)--(7.5,4.5);
  \draw[line width=0.6mm, blue] (5.5,5.5)--(6.5,5.5);
  \draw[line width=0.6mm, blue!50,dashed] (1,1)--(11.4,.6);
  \draw[line width=0.6mm, blue!50,dashed] (2.25,2.25)--(10.2,1.8);
  \draw[line width=0.6mm, blue!50,dashed](2.75,2.75)--(9,3);
  \draw[line width=0.6mm, blue!50,dashed] (3.75,3.75)--(7.8,4.2);
  \draw[line width=0.6mm, blue!50,dashed] (5.25,5.25)--(6.6,5.4);
  \draw[decorate, decoration=brace,red](-0.1,0.1)--(1.9,2.1);
  \node[red] (a) at (.6,1.4) {2};
  \draw[decorate, decoration=brace,red](1.9,2.1)--(2.9,3.1);
  \node[red] (b) at (2.1,2.9) {1};
  \draw[decorate, decoration=brace,red](2.9,3.1)--(5.9,6.1);
  \node[red] (c) at (4.1,4.9) {3};
  \draw[red,fill] (0,0) circle (4pt);
    \draw[red,fill] (6,6) circle (4pt);
    \draw[red,fill] (12,0) circle (4pt);
    \draw[red,fill] (2,2) circle (4pt);
    \draw[red,fill] (3,3) circle (4pt);
    
\end{tikzpicture}
  \caption{The surface for $g$-vector $(-2,-1,-3,6)$, with the corresponding arcs drawn in blue, and with the arcs for
    $(-1,-2,-2,5)$ superimposed as dotted lines, as in the first case of the proof} 
\end{figure}

Clearly, if $b$ and $c$ are both negative, a completely equivalent argument can be made.

Suppose next that $b$ is negative and $c$ is positive. We proceed very much
as in the previous case. As before, we consider the Dyck path model of the
surface, and draw the curve $C$ with $g$-vector $(a,b,c,d)$. We wish to
draw a curve $C'$ with $g$-vector $(a+1,b-1,c+1,d-1)$ in such a way that it does not intersect
$C$. Divide the lefthand side of the Dyck path into a first part of length
$a$, and a second part of length $b$. Divide the part of length $a$ into
$a-1$ segments of equal length, and locate endpoints of $C'$ at the midpoints
of these segments. Similarly, divide the part of length $b$ into $b+1$ segments of equal length, and locate endpoints of $C'$ at the midpoints of these
segments. Then do the same thing on the righthand side. The result is that
we have located an endpoint of $C'$ between each successive pair of endpoints
of $C$, including above the topmost endpoints of $C$ on each side, but not
below the bottommost one on each side. We can draw $C'$ so as to connect them in a noncrossing way.

\begin{figure}
\begin{tikzpicture}[scale=.9]
\draw[step=1.0,gray!50,thin] (-0.5,-0.5) grid (10.5,5.5);\draw[line width=1mm, black,dashed] (-0.5,0)--(10.5,0);
  \draw[line width = 0.4mm, red] (0,0)--(5,5)--(10,0);
  \draw[line width = 0.6mm, blue] (.5,.5)--(9.5,.5);
  \draw[line width = 0.6mm, blue] (1.5,1.5)--(8.5,1.5);
  \draw[line width = 0.6mm, blue] (2.5,2.5)--(7.5,2.5);
  \draw[line width = 0.6mm, blue] (3.5,3.5)--(6.5,3.5);
  \draw[line width = 0.6mm, blue] (4.5,4.5)--(5.5,4.5);
  \draw[line width = 0.6mm, blue!50, dashed] (1,1)--(9.33,.66);
  \draw [line width = 0.6mm, blue!50, dashed] (2.375,2.375)--(8,2);
  \draw [line width = 0.6mm, blue!50, dashed] (3.125,3.125)--(6.66,3.33);
  \draw [line width = 0.6mm, blue!50, dashed] (3.875,3.875)--(5.75,4.25);
  \draw [line width = 0.6mm, blue!50, dashed] (4.625,4.625)--(5.25,4.75);
  \draw[decorate, decoration=brace,red](-0.1,0.1)--(1.9,2.1);
  \node[red] (a) at (.6,1.4) {2};
  \draw[decorate, decoration=brace,red](1.9,2.1)--(4.9,5.1);
  \node[red] (b) at (3.1,3.9) {3};
  \draw[decorate, decoration=brace,red](5.1,5.1)--(6.1,4.1);
  \node[red] (c) at (5.9,4.9) {1};
  \draw[decorate, decoration=brace,red](6.1,4.1)--(10.1,0.1);
  \node[red] (d) at (8.4,2.4) {4};
  \draw[red,fill] (0,0) circle (4pt);
    \draw[red,fill] (5,5) circle (4pt);
    \draw[red,fill] (10,0) circle (4pt);
    \draw[red,fill] (2,2) circle (4pt);
    \draw[red,fill] (6,4) circle (4pt);
\end{tikzpicture}
  \caption{The surface for $g$-vector $(-2,-3,1,4)$, with the corresponding arcs drawn in blue, and with the curves for
    $(-1,-4,2,3)$ superimposed as dotted lines, as in the second case of the proof} 
\end{figure}
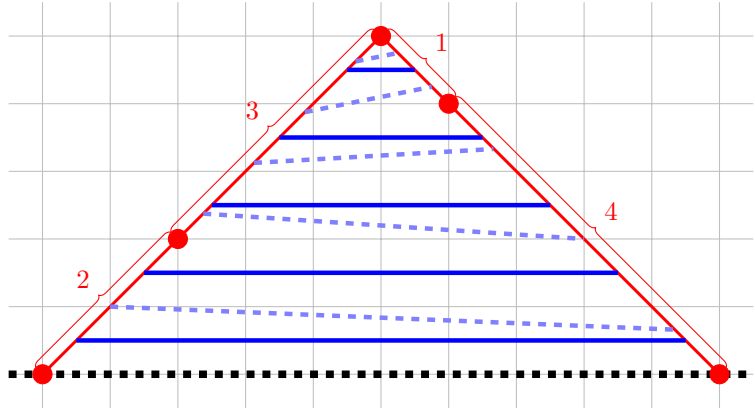

Suppose finally that $b$ is positive and $c$ is negative. Again, essentially the same idea works. \qedhere

\begin{figure}
  \begin{tikzpicture}[scale=.9]
  \draw[step=1.0,gray!50,thin] (-0.5,-0.5) grid (12.5,4.5);\draw[line width=1mm, black,dashed] (-0.5,0)--(12.5,0);
    \draw[line width = 0.4mm, red](0,0)--(4,4)--(7,1)--(9,3)--(12,0);
    \draw[line width = 0.6mm, blue](0.5,0.5)--(11.5,0.5);
    \draw[line width = 0.6mm, blue] (1.5,1.5)--(6.5,1.5);
    \draw[line width = 0.6mm, blue] (2.5,2.5)--(5.5,2.5);
    \draw[line width = 0.6mm, blue] (3.5,3.5)--(4.5,3.5);
    \draw[line width = 0.6mm, blue] (7.5,1.5)--(10.5,1.5);
    \draw[line width = 0.6mm, blue] (8.5,2.5)--(9.5,2.5);
    \draw[line width = 0.6mm, blue!50,dashed](0.66,0.66)--(11.25,.75);
    \draw[line width = 0.6mm, blue!50,dashed](2,2)--(6.25,1.75);
    \draw[line width = 0.6mm, blue!50,dashed](3.33,3.33)--(4.75,3.25);
    \draw[line width = 0.6mm, blue!50,dashed](8,2)--(9.75,2.25);
    \draw[red,fill] (0,0) circle (4pt);
    \draw[red,fill] (4,4) circle (4pt);
    \draw[red,fill] (7,1) circle (4pt);
    \draw[red,fill] (9,3) circle (4pt);
    \draw[red,fill] (12,0) circle (4pt);
    \end{tikzpicture}
    \caption{The surface for $g$-vector $(-4,3,-2,3)$, with the corresponding arcs drawn in blue, and with the arcs for
      $(-3,2,-1,2)$ superimposed as dotted lines, as in the third case of the proof}
    \end{figure}
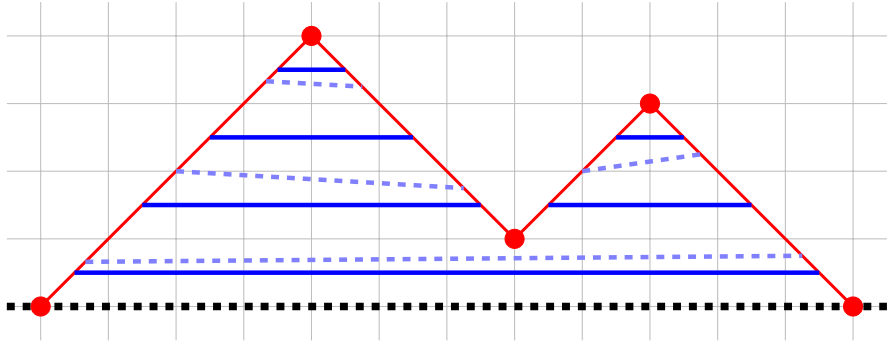
\end{proof} 

Since consecutive points of $\mathcal D$ are compatible, we now see
that any semibrick with $g$-vector in $C$ can be decomposed as a direct
sum of bricks from $\mathcal D$. This implies that $\mathcal D$ accounts
for all the $g$-vectors of bricks in $V$. 

We now check that the points of $\mathcal D$ are exactly the points of
$C$ which satisfy that $\gcd(a+b,b+c)=1$.
By Lemma \ref{lattice-basis}, any lattice point in $C$ can
be expressed as $pv+q(-1,1,-1,1)$ with $p$ positive and $q$ non-negative.
This lattice point satisfies the condition from the statement of the
theorem if and only if $pv$ does, and it is clear that $pv$ satisfies the
condition of the theorem if and only if $p=1$, which gives exactly the
desired condition.
\end{proof}

\section{Example: the Christoffel case \texorpdfstring{$(n=3)$}{(n=3)}}

In this extended example, we consider the case $n=3$.

\subsection{Gentle algebra} $Q_3$ is the quiver:

$$\begin{tikzpicture}[->]
  \node (a) at (0,0) {$1$};
  \node (b) at (2,0) {$2$};
  \node (c) at (4,0) {$3$};
  \draw ([yshift=1mm]b.west)--node[above]{$\beta_{1}$}([yshift=1mm]a.east);
  \draw  ([yshift=-1mm]b.west)--node[below]{$\alpha_{1}$}([yshift=-1mm]a.east);
    \draw ([yshift=1mm]c.west)--node[above]{$\alpha_{2}$}([yshift=1mm]b.east);
  \draw  ([yshift=-1mm]c.west)--node[below]{$\beta_{2}$}([yshift=-1mm]b.east);
  \draw[dashed,-] ([xshift=.4cm]b.north) arc[start angle = 0, end angle = 180, x radius=.4cm, y radius =.2cm];
  \draw[dashed,-] ([xshift=-.4cm]b.south) arc[start angle = 180, end angle = 360, x radius=.4cm, y radius =.2cm];
  
\end {tikzpicture}
$$ 

The paths in $Q_3$ are as follows: there are three lazy paths (one at each vertex), four paths of
length 1 (the arrows), and four paths of length two: $\alpha_1\alpha_2$, $\alpha_1\beta_2$, $\beta_1\alpha_2$, $\beta_1\beta_2$. These eleven paths form a basis
for the path algebra $kQ_3$. 

The set~$R_3$ consists of the two paths $\alpha_1\beta_2$ and
$\beta_1\alpha_2$ (indicated by the dotted lines in the figure), and the ideal $I_3$ consists of linear combinations of these two paths.

Then~$(Q_3,R_3)$ is a gentle quiver.  We define the algebra $\Lambda_3=kQ_3/I_3$.

\subsection{Dissected surface} The dissected surface for $\Lambda_3$ consists
of two $4$-gons $\mathsf P_1$ and $\mathsf P_2$, with edges identified as in \cref{fig:rect}.

\begin{figure}[h!]
\centering
    \begin{tikzpicture}
        \begin{scope}[thick, decoration={markings, mark=at position 0.6 with {\arrow{Stealth[length=3mm]}}},mydot/.style={
    circle,
    thick,
    fill=white,
    draw,
    outer sep=0.4pt,
    inner sep=1pt
  }]
            \draw (0,0) -- node[above] {$E^{(1)}_0$} (4,0);
            \draw[line width=0.4mm,red,postaction={decorate}] (4,0) -- node[right] {$E^{(1)}_1$} (4,-2);
            \draw[line width=0.4mm,red,postaction={decorate}] (4,-2) -- node[below] {$E^{(1)}_2$} (0,-2);
        \draw[line width=0.4mm,red,postaction={decorate}]  (0,-2) -- node[left] {$E^{(1)}_3$} (0,0);
        \draw[dark-green,mydot] (2,0) circle (2.5pt);
        \draw[red,fill] (4,0) circle (3pt);
        \draw[red,fill] (0,0) circle (3pt);
        \draw[red,fill] (0,-2) circle (3pt);
        \draw[red,fill] (4,-2) circle (3pt);
        \draw[red] node at (2,-1) {$\mathsf{P}_1$};
        \end{scope}
        \begin{scope}[xscale=-1,shift={(-10,0)},thick, decoration={markings, mark=at position 0.6 with {\arrow{Stealth[length=3mm]}}},mydot/.style={
    circle,
    thick,
    fill=white,
    draw,
    outer sep=0.4pt,
    inner sep=1pt
  }]
        \draw (0,0) -- node[above] {$E^{(2)}_0$} (4,0);
        \draw[line width=0.4mm,red,postaction={decorate}] (4,0) -- node[left] {$E^{(2)}_3$} (4,-2);
        \draw[line width=0.4mm,red,postaction={decorate}] (4,-2) -- node[below] {$E^{(2)}_2$} (0,-2);
        \draw[line width=0.4mm,red,postaction={decorate}]  (0,-2) -- node[right] {$E^{(2)}_1$} (0,0);
        \draw[dark-green,mydot] (2,0) circle (2.5pt);
        \draw[red,fill] (4,0) circle (3pt);
        \draw[red,fill] (0,0) circle (3pt);
        \draw[red,fill] (0,-2) circle (3pt);
        \draw[red,fill] (4,-2) circle (3pt);
        \draw[red] node at (2,-1) {$\mathsf{P}_2$};
        \end{scope}
    \end{tikzpicture}
        \caption{\label{fig:rect} The two $4-$gons representing the dissected surfaced associated to $\Lambda_3$.} 
       \end{figure}
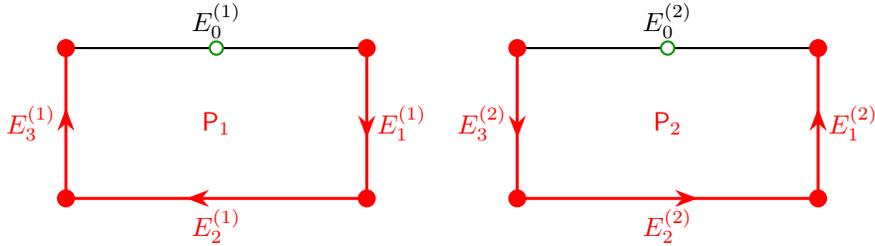

A closed curve drawn on this surface is a closed slalom if it is possible to
orient the curve so that the segments of the curve on $\mathsf P_1$ go from right to
left, while those on $\mathsf P_2$ go from left to right. Simple closed slaloms
correspond to band bricks. Simple closed multislaloms correspond to
band semibricks.

Given a closed slalom on the surface, it corresponds to a walk on the quiver
$\Lambda$ in the following way. Start on $\mathsf P_1$ from the righthand endpoint of a segment of the slalom on $\mathsf P_1$. If the slalom connects $E_{i_0}^{(1)}$ to $E_{i_1}^{(1)}$ (so by hypothesis $i_1>i_0$), walk from vertex $i_0$ to vertex $i_1$ along arrows labelled $\beta$ (in the reverse of the direction of the arrows). Then, continue to follow the slalom on $\mathsf P_2$. If it
goes from $E_{i_1}^{(2)}$ to $E_{i_2}^{(2)}$ (necessarily with $i_1>i_2$, by the slalom condition), then walk forwards
along arrows labelled $\alpha$ from vertex $i_1$ to vertex $i_2$. Continue from there, back on $\mathsf P_1$, and keep going in the same way, until eventually returning to the beginning point, which eventually causes the walk to close up at the initial vertex $i_0$.
Such a  walk $w$ corresponds to a family of band representations,
$B_{w,1,\lambda}$, for $\lambda\in k^\times$, as described in
Section \ref{ssec::band}.

  \begin{figure}[h!]
\centering
    \begin{tikzpicture}
        \begin{scope}[thick, decoration={markings, mark=at position 0.6 with {\arrow{Stealth[length=3mm]}}},mydot/.style={
    circle,
    thick,
    fill=white,
    draw,
    outer sep=0.4pt,
    inner sep=1pt
  }]
            \draw (0,0) -- node[above] {$E^{(1)}_0$} (4,0);
            \draw[line width=0.4mm,red,postaction={decorate}] (4,0) -- node[right] {$E^{(1)}_1$} (4,-2);
            \draw[line width=0.4mm,red,postaction={decorate}] (4,-2) -- node[below] {$E^{(1)}_2$} (0,-2);
        \draw[line width=0.4mm,red,postaction={decorate}] (0,-2) -- node[left] {$E^{(1)}_3$} (0,0);
        \draw[dark-green,mydot] (2,0) circle (2.5pt);
        \draw[red,fill] (4,0) circle (3pt);
        \draw[red,fill] (0,0) circle (3pt);
        \draw[red,fill] (0,-2) circle (3pt);
        \draw[red,fill] (4,-2) circle (3pt);
        \draw[red] node at (2,-1) {$\mathsf{P}_1$};
        
        \draw[line width=0.6mm,blue, bend left=20]  (0,-1.3) edge node[below left,blue] {$1$} (2,-2);
        \draw[line width=0.6mm,blue,bend left = 20]  (0,-0.8) edge node[below right,blue] {$3$}(4,-0.8);
        \end{scope}
        \begin{scope}[xscale=-1,shift={(-10,0)},thick, decoration={markings, mark=at position 0.6 with {\arrow{Stealth[length=3mm]}}},mydot/.style={
    circle,
    thick,
    fill=white,
    draw,
    outer sep=0.4pt,
    inner sep=1pt
  }]
        \draw (0,0) -- node[above] {$E^{(2)}_0$} (4,0);
        \draw[line width=0.4mm,red,postaction={decorate}] (4,0) -- node[left] {$E^{(2)}_3$} (4,-2);
        \draw[line width=0.4mm,red,postaction={decorate}] (4,-2) -- node[below] {$E^{(2)}_2$} (0,-2);
        \draw[line width=0.4mm,red,postaction={decorate}]  (0,-2) -- node[right] {$E^{(2)}_1$} (0,0);
        \draw[dark-green,mydot] (2,0) circle (2.5pt);
        \draw[red,fill] (4,0) circle (3pt);
        \draw[red,fill] (0,0) circle (3pt);
        \draw[red,fill] (0,-2) circle (3pt);
        \draw[red,fill] (4,-2) circle (3pt);
        \draw[red] node at (2,-1) {$\mathsf{P}_2$};
        
         \draw[line width=0.6mm,blue, bend right=20]  (4,-1.3) edge node[below left,blue] {$4$} (2,-2);
        \draw[line width=0.6mm,blue,bend left = 20]  (0,-0.8) edge node[below right,blue] {$2$}(4,-0.8);
        \end{scope}
    \end{tikzpicture}
        \caption{\label{fig:rect2} A slalom, with a possible order in which to follow its edges indicated} 
       \end{figure}
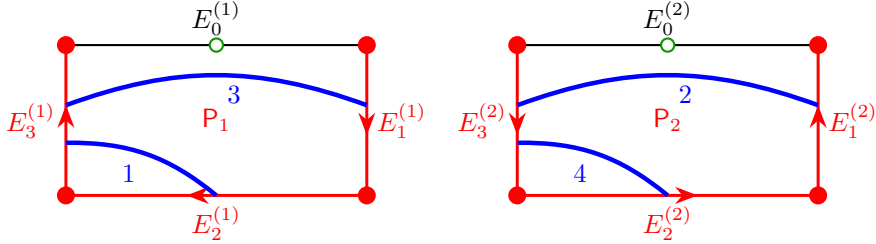

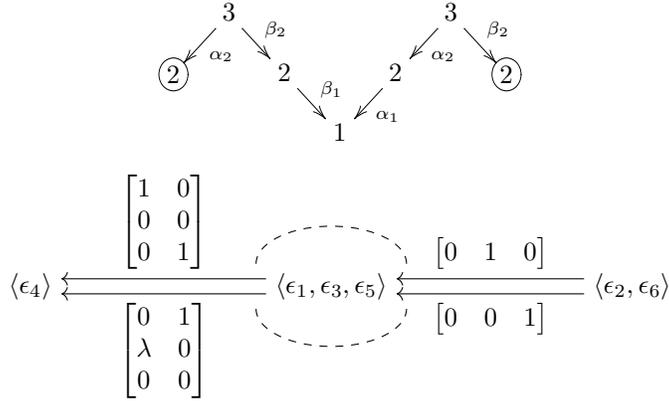
\begin{figure}[h!]
\centering
  \begin{center}
    $\displaystyle \xymatrix@C=1em @R=1em{
      & 3\ar[rd]^{\beta_2} \ar[ld]^{\alpha_2}&  &   &  & 3\ar[rd]^{\beta_2} \ar[ld]^{\alpha_2}& \\
     *+[o][F-]{2}&  &2\ar[rd]^{\beta_1} &   & 2\ar[ld]^{\alpha_1}&  &*+[o][F-]{2}\\
      &  &  & 1 &  &  & }$  

 $$ \begin{tikzpicture}[->]
  \node (a) at (-1,0) {$\langle \epsilon_4 \rangle$};
  \node (b) at (3,0) {$\langle \epsilon_1,\epsilon_3,\epsilon_5 \rangle$};
  \node (c) at (7,0) {$\langle \epsilon_2,\epsilon_6\rangle$};
  \draw ([yshift=1mm]b.west)--node[above]{$ \left[\begin{matrix}
1&0\\0&0\\0&1
\end{matrix} \right]$}([yshift=1mm]a.east);
  \draw  ([yshift=-1mm]b.west)--node[below]{$\left[\begin{matrix}
0&1\\ \lambda&0\\0&0
      \end{matrix} \right]  $}([yshift=-1mm]a.east);
    \draw ([yshift=1mm]c.west)--node[above]{$\left[\begin{matrix}
 0&1&0
\end{matrix} \right]  $}([yshift=1mm]b.east);
  \draw  ([yshift=-1mm]c.west)--node[below]{$\left[\begin{matrix}
 0&0&1
\end{matrix} \right]  $}([yshift=-1mm]b.east);
  \draw[dashed,-] ([xshift=1cm]b.north) arc[start angle = 0, end angle = 180, x radius=1cm, y radius =.5cm];
  \draw[dashed,-] ([xshift=-1cm]b.south) arc[start angle = 180, end angle = 360, x radius=1cm, y radius =.5cm];
  \end {tikzpicture}$$
\end{center}\caption{Top: the band walk corresponding to the slalom of Figure \ref{fig:rect2}. Bottom: the band representation (up to isomorphism) associated to it, following the same construction as \cref{ex::bandrep}.}
  \label{fig:walk}
\end{figure}

For example, the slalom of Figure \ref{fig:rect2} corresponds to the walk
shown in Figure \ref{fig:walk}, which then corresponds to the
representation below:

The representation corresponding to a closed slalom is a band brick if it is simple, i.e., it can be drawn so as to have no self-intersections. The above
slalom from Figure \ref{fig:rect2} has no self-intersections, so the
corresponding representation is a brick. This fact can also be checked directly using the criterion from Section \ref{ssec:band-morphisms} that there is no substring of this walk which can be found simultaneously on top and at the bottom, so the only
endomorphisms of this module are multiples of the identity.

\subsection{\texorpdfstring{$g$}{g}-vectors} The possible $g$-vectors of a closed, simple multislalom are integer vectors
of the form $(-a,-b,a+b)$ with $a,b \geq 0$, or of the form $(-a-b,a,b)$,
with $a,b\geq 0$.

In the first case, it is convenient to redraw the polygons with edges 1 and 2
parallel. If the lengths of the sides are chosen to agree with the
absolute values of the corresponding entries in the $g$-vector, then the
multislalom can be drawn as equally spaced horizontal lines. See Figure \ref{triangles} (left) for an example. In this case, the same picture suffices for
both $\mathsf P_1$ and $\mathsf P_2$; 

Similarly, in the second case, it is convenient to draw the edges 2 and 3 parallel. See Figure \ref{triangles} (right) for an example.

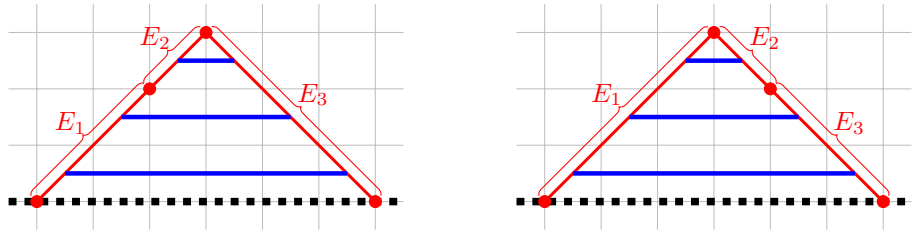
\begin{figure}\begin{tikzpicture}[scale=.75]
\draw[step=1.0,gray!50,thin] (-0.5,-0.5) grid (6.5,3.5);\draw[line width=1mm, black,dashed] (-0.5,0)--(6.5,0);
    \draw[line width=0.4mm,red] (0,0)--(3,3)--(6,0);
    \draw[line width=0.6mm,blue] (.5,.5)--(5.5,.5);
    \draw[line width=0.6mm,blue] (1.5,1.5)--(4.5,1.5);
    \draw[line width=0.6mm,blue] (2.5,2.5)--(3.5,2.5);
    \draw[decorate, decoration=brace,red](-0.1,0.1)--(1.9,2.1);
  \node[red] (a) at (.6,1.4) {$E_1$};
  \draw[decorate, decoration=brace,red](1.9,2.1)--(2.9,3.1);
  \node[red] (b) at (2.1,2.9) {$E_2$};
  \draw[decorate, decoration=brace,red](3.1,3.1)--(6.1,.1);
  \node[red] (c) at (4.9,1.9) {$E_3$};
  \draw[red,fill] (6,0) circle (3pt);
  \draw[red,fill] (0,0) circle (3pt);
  \draw[red,fill] (2,2) circle (3pt);
  \draw[red,fill] (3,3) circle (3pt);
  \end{tikzpicture}\qquad\qquad
  \begin{tikzpicture}[scale=.75]
  \draw[step=1.0,gray!50,thin] (-0.5,-0.5) grid (6.5,3.5);\draw[line width=1mm, black,dashed] (-0.5,0)--(6.5,0);
     \draw[line width = 0.4mm, red] (0,0)--(3,3)--(6,0);
    \draw[line width = 0.6mm, blue] (.5,.5)--(5.5,.5);
    \draw[line width = 0.6mm, blue] (1.5,1.5)--(4.5,1.5);
    \draw[line width = 0.6mm, blue] (2.5,2.5)--(3.5,2.5);
    \draw[decorate, decoration=brace,red](-0.1,0.1)--(2.9,3.1);
  \node[red] (a) at (1.1,1.9) {$E_1$};
  \draw[decorate, decoration=brace,red](3.1,3.1)--(4.1,2.1);
  \node [red](b) at (3.9,2.9) {$E_2$};
  \draw[decorate, decoration=brace,red](4.1,2.1)--(6.1,.1);
  \node[red] (c) at (5.4,1.4) {$E_3$};
  \draw[red,fill] (0,0) circle (3pt);
  \draw[red,fill] (4,2) circle (3pt);
  \draw[red,fill] (3,3) circle (3pt);
  \draw[red,fill] (6,0) circle (3pt);
  \end{tikzpicture}

    \caption{Left: the polygons adapted to the slalom with $g$-vector $(-2,-1,3)$, together with the corresponding slalom. Right: the polygons adapted to the slalom with $g$-vector $(-3,1,2)$, together with the corresponding slalom.\label{triangles}}
\end{figure} 

There is a bijection between these two families of slaloms which sends a
slalom to its mirror image (on each $\mathsf P_i$).

Note that there is exactly
one slalom which falls in both classes: this is the slalom that consists of a single curve between $E_1$ and $E_3$ on each of $\mathsf P_1$ and $\mathsf P_2$.
It is fixed under this bijection.

\subsection{Christoffel words}

When the alphabet has two letters, a perfectly clustering word, which is also a Lyndon word, is called a \newword{lower Christoffel word}. For example, for $A=\{a<b\}$, $a,b,ab,aab,aabab, ababb$ are lower Christoffel words. Equivalently, take any primitive perfectly clustering words, and pick in its conjugation class the smallest element for the lexicographical order: this is a lower Christoffel word (an upper Christoffel word is obtained by picking the largest element).

The {\it slope} of a word $w$ in $\{a<b\}^*$ is the ratio $|w|_b/|w|_a$, where $|w|_x$ denotes the number of occurrences of the letter $x$ in $w$; the slope is an element of $\mathbb Q\cup \{\infty\}$. It turns out that for any two lower Christofflel word $u,v$, one has $u<_{lex} v$ if and only if the slope of $u$ is smaller than that of $v$ (result due to Borel and Laubie; see \cite{R} Proposition 13.1.1). For example, $aabab<_{lex} ababb$ and their slopes are respectively $2/3$ and $3/2$.

The one-parameter families of band bricks whose $g$-vectors are of the form $(-a-b,a,b)$
correspond bijectively to Christoffel words by restricting Theorem 
\ref{thm: brick bands and pcw} to the $n=3$ case. The Christoffel word with
$a$ occurrences of the letter $x$ and $b$ occurrences of the letter $y$
correponds to the band brick with $g$-vector $(-a-b,a,b)$.
In particular, we see that there is a band brick with $g$-vector
$(-a-b,a,b)$ with $a,b\geq 0$ if and only if $a$ and $b$ are relatively prime.

The band bricks whose
$g$-vectors are of the form $(-a,-b,a+b)$ are also in bijective correspondence with Christoffel words, by first applying the bijection mentioned above from the families of band bricks with $g$-vectors $(-a,-b,a+b)$ to $(-a-b,a,b)$, and then applying the bijection to Christoffel words.

\subsection{Euler form}
The Euler form (see Section \ref{euler}) is given on $g$-vectors by
$$\langle (a_1,a_2,a_3),(b_1,b_2,b_3)\rangle =
a_1b_1 + 2a_1b_2 + 2a_1b_3 +a_2b_2 + 2a_2b_3 + a_3b_3$$

Let $v$ and $w$ be two Christoffel words. Consider the representations
corresponding to them under the correspondence from Theorem \ref{thm: brick bands and pcw}, whose $g$-vectors are $g_v=(-|v|_x-|v|_y,|v|_x,|v|_y)$ and
$g_w=(-|w|_x-|w|_y,|w|_x,|w|_y)$ respectively. 
Substituting these into the Euler form, we obtain:
\begin{eqnarray*}
  \langle g_v,g_w\rangle &=& (|v|_x+|v|_y)(|w|_x+|w|_y)-2(|v|_x+|v|_y)|w|_x
  -2(|v|_x+|v|_y)|w|_y\\&&\qquad + |v|_x|w|_x +2|v|_x|w|_y + |v|_y|w|_y\\
  &=& |v|_x|w|_y - |v|_y|w|_x\end{eqnarray*}

This expression is well-known as playing an important rôle in the combinatorics of Christoffel words.

\subsection{Compatible bricks}

Two distinct one-parameter families of bricks are never compatible. This
is a special case of Theorem \ref{maxsize}. It also follows from the above formula for the Euler form, which shows that $\langle g_v,g_w\rangle$ will be non-zero if $v$ and $w$ have distinct slopes, while there is only one one-parameter family of bricks with each possible slope.

\end{document}